\documentclass[11pt]{article}

\usepackage[margin=1in]{geometry}
\usepackage{amsthm,amsmath,amsfonts,amssymb}
\usepackage[numbers,sort&compress]{natbib}
\usepackage[colorlinks,citecolor=blue,urlcolor=blue]{hyperref}
\usepackage{graphicx}
\usepackage{authblk}

\theoremstyle{plain}

\newtheorem{theorem}{Theorem}[section]
\newtheorem{lemma}[theorem]{Lemma}
\newtheorem{corollary}[theorem]{Corollary}
\newtheorem{proposition}[theorem]{Proposition}

\theoremstyle{definition}
\newtheorem{definition}[theorem]{Definition}
\newtheorem{example}{Example}

\newtheorem{remark}{Remark}
\newtheorem{assumption}{Assumption}

\newcommand{\SGD}{\text{SGD}}

\newcommand{\Ridge}{\text{Ridge}}

\input{mystyle}

\title{Statistical Inference for Misspecified Contextual Bandits}

\author[1]{Yongyi Guo\thanks{Corresponding author. Email: \texttt{guo98@wisc.edu}.}}
\author[2]{Ziping Xu\thanks{Email: \texttt{zipingxu@unc.edu}. The work was partially done when Ziping was a postdoctoral fellow at Harvard University with Professor Susam A. Murphy.}}

\affil[1]{Department of Statistics, University of Wisconsin--Madison}
\affil[2]{School of Data Science and Society, University of North Carolina at Chapel Hill}

\date{} % Empty date for arXiv; remove this line to show the compilation date

\begin{document}
\maketitle

% \begin{frontmatter}

%\title{A sample article title with some additional note\thanksref{t1}}
% \runtitle{Statistical Inference for Misspecified Contextual Bandits}
%\thankstext{T1}{A sample additional note to the title.}

% \begin{aug}
%%%%%%%%%%%%%%%%%%%%%%%%%%%%%%%%%%%%%%%%%%%%%%%
%% Only one address is permitted per author. %%
%% Only division, organization and e-mail is %%
%% included in the address.                  %%
%% Additional information can be included in %%
% \author[A]{\fnms{Yongyi}~\snm{Guo} \thanks{[\textbf{Corresponding author indication should be put in the Acknowledgment section if necessary.}]}\ead[label=e1]{guo98@wisc.edu}},
% \author[B]{\fnms{Ziping}~\snm{Xu}\ead[label=e2]{zipingxu@unc.edu}\orcid{0000-0000-0000-0000}}
% % \and
% % \author[B]{\fnms{Third}~\snm{Author}\ead[label=e3]{third@somewhere.com}}
% %%%%%%%%%%%%%%%%%%%%%%%%%%%%%%%%%%%%%%%%%%%%%%
% %% Addresses                                %%
% %%%%%%%%%%%%%%%%%%%%%%%%%%%%%%%%%%%%%%%%%%%%%%
% \address[A]{Department of Statistics,
% University of Wisconsin-Madison %\textbf{[Additional affiliations should be put in the Acknowledgments section]}
% \printead[presep={ ,\ }]{e1}}

% \address[B]{School of Data Science and Society,
% University of North Carolina at Chapel Hill %\textbf{[Additional affiliations should be put in the Acknowledgments section]}
% \printead[presep={,\ }]{e2}}
% \end{aug}

\begin{abstract}

Contextual bandit algorithms have transformed modern experimentation by enabling real-time adaptation for personalized treatment and efficient use of data. Yet these advantages create challenges for statistical inference due to adaptivity. A fundamental property that supports valid inference is policy convergence, meaning that action-selection probabilities converge in probability given the context. Convergence ensures replicability of adaptive experiments and stability of online algorithms. In this paper, we highlight a previously overlooked issue: widely used algorithms such as LinUCB may fail to converge when the reward model is misspecified, and such non-convergence creates fundamental obstacles for statistical inference. This issue is practically important, as misspecified models---such as linear approximations of complex dynamic system---are often employed in real-world adaptive experiments to balance bias and variance.

Motivated by this insight, we propose and analyze a broad class of algorithms that are guaranteed to converge even under model misspecification. Building on this guarantee, we develop a general inference framework based on an inverse-probability-weighted Z-estimator (IPW-Z) and establish its asymptotic normality with a consistent variance estimator. Simulation studies confirm that the proposed method provides robust and data-efficient confidence intervals, and can outperform existing approaches that exist only in the special case of offline policy evaluation. Taken together, our results underscore the importance of designing adaptive algorithms with built-in convergence guarantees to enable stable experimentation and valid statistical inference in practice.
\end{abstract}

% \begin{keyword}[class=MSC]
% \kwd[Primary ]{68T01}
% \kwd{62F12}
% %\kwd{00X00}
% \kwd[; secondary ]{62M99}
% \end{keyword}

% \begin{keyword}
% \kwd{Statistical inference}
% \kwd{contextual bandits}
% \kwd{adaptive experiments}
% \kwd{policy convergence}
% \end{keyword}

% \end{frontmatter}
%%%%%%%%%%%%%%%%%%%%%%%%%%%%%%%%%%%%%%%%%%%%%%
%% Please use \tableofcontents for articles %%
%% with 50 pages and more                   %%
%%%%%%%%%%%%%%%%%%%%%%%%%%%%%%%%%%%%%%%%%%%%%%
%\tableofcontents

\section{Introduction}
Adaptive experimental designs play an increasingly prominent role across science and industry. In contrast to traditional randomized experiments, they update arm allocation probabilities sequentially in response to observed outcomes, thereby allowing real-time improvement in treatment strategies and more efficient sample usage. For example, in mobile health, adaptive designs personalize digital interventions---such as motivational messages or prompts tailored to users’ states---to promote healthy behaviors \citep{nahum2016just, hardeman2019systematic}. In online advertising and recommender systems, content and ad delivery are dynamically adjusted to maximize engagement or click-through rates \citep{li2010contextual, tang2013automatic}. In education technology, adaptive experimentation has been used to optimize tutoring systems and learning platforms by tailoring exercises to student performance \citep{kizilcec2020scaling, fischer2020mining}. %In addition, a line of work has focused on dynamically reallocating samples to increase statistical power for hypothesis tests, thereby reducing total sample sizes needed for a given significance level \citep{russo2016simple}.

A common approach to implementing such adaptive experiments is through contextual bandit algorithms \citep{li2010contextual, offer2021adaptive, coughlin2024mobile, lauffenburger2024impact, nahum2024optimizing}. Contextual bandits provide a principled framework for sequential decision-making by balancing the tradeoff between exploration and exploitation. At each round $t$, the agent observes a context $\bX_t$ (e.g., user features), selects an action $A_t$ according to a behavior policy $\pi_t(\cdot \mid \bX_t)$, and then receives an outcome $Y_t$. The goal is to maximize cumulative reward while learning to improve the decision rule over time. 
%To effectively learn the unknown environment and improve its decision rule over time, the agent must balance exploration (trying actions with uncertain outcomes) and exploitation (favoring actions known to yield high outcomes). 
This continual updating of the policy based on observed outcomes captures the essence of adaptivity in practice. Popular algorithms such as LinUCB \citep{li2010contextual} and Thompson Sampling \citep{russo2014learning} are widely used in both research and applications, and enjoy strong theoretical guarantees on sample efficiency.

In practice, however, many contextual bandit algorithms operate under a \emph{misspecified reward model}. Misspecification is common in adaptive decision-making, especially in complex environments where data are limited and outcomes are noisy. One source is that the true reward-generating process is rarely known, leaving adaptive policies subject to model errors, particularly when there is rich heterogeneity or when the outcome mechanism is poorly understood \citep{dimakopoulou2017estimation, trella2025deployed}. Misspecification may also arise by design: practitioners often adopt simpler working models to stabilize learning and balance bias against variance \citep{athey2022contextual, tewari2017ads}. It can further occur when covariates are measured with error or when relevant features are unobserved. Because of its prevalence, misspecification poses important challenges for adaptive experimentation, and there has been a line of literature on designing online policies robust to various misspecified environments \citep{ghosh2017misspecified, foster2020adapting, lattimore2020learning, krishnamurthy2021tractable, krishnamurthy2021adapting}.

In this work, we tackle a different problem: \textbf{statistical inference with data collected by contextual bandit algorithms in a misspecified environment}. Post-experiment statistical inference is important because practitioners and scientists often seek more than an effective online policy: they need valid confidence intervals, hypothesis tests, and replicable findings for downstream scientific or policy questions. A substantial literature has studied inference with adaptively collected data under various modeling assumptions—for example, generalized linear and partial linear models \citep{deshpande2018accurate, boruvka2018assessing, qian2021estimating, lin2023semi}, structured nested mean models \citep{syrgkanis2023post}, and general nonlinear regression or M-estimation frameworks \citep{klimko1978conditional, lai1994asymptotic, zhang2021statistical}. However, the majority of these works focus on inferring parameters in a correctly specified outcome model, and only a few consider inference under general data distributions. The latter include studies focused on specific inferential targets \citep{chen2021statistical, zhan2021off, bibaut2021post} or on a different regime with bounded horizons and a diverging number of trajectories \citep{zhang2022statistical, zhang2024replicable}. This paper focuses on the largely unexplored problem of statistical inference under misspecification when the number of time points grows. We assume no well-specified outcome model, either for decision-making or for after-study analysis. This setting is both practically important and theoretically challenging.

%\YG{I guess many works on adaptive inference for SNMM/factor model can be doubly robust, so strictly speaking they don't have to have a completely correct model}

One key difficulty we uncover is that statistical inference under misspecification is hindered \emph{from the very start of data collection}: popular bandit algorithms such as LinUCB may fail to converge as a policy, and this lack of convergence can fundamentally compromise subsequent inference. Here, by \emph{policy convergence} we mean that the action-selection probabilities stabilize in probability given the context (see Definition \ref{aspt:policy_convergence}). Convergence (or related notions of policy stability) has been recognized as important in many adaptive experimental settings---not only for ensuring valid statistical inference \citep{zhang2022statistical, khamaru2024inference, halder2025stable}, but also for achieving replicability, i.e., obtaining consistent results when an experiment is repeated under the same conditions \citep{zhang2024replicable}. Replicability, in turn, is critical for validating scientific findings and for building confidence in data-driven decisions. Our results reveal that while policy convergence is typically not hard to establish under a well-specified model, it may fail entirely under misspecification---even for standard bandit algorithms---leading to pathological estimator behavior, breakdowns of asymptotic normality, and ultimately invalid inference (see an example in Section \ref{sec::example-policy-nonconvergence}).

To better understand when convergence is preserved, we establish general conditions for policy convergence without a well-specified reward model, and identify a broad class of adaptive policies that remain stable even in complex environments. At the core is a simple principle: a policy will converge as long as it bases decisions on \emph{summary statistics} that converge to a limit where the policy mapping is continuous. 
% \ziping{People who know the literature might know that for MAB we need regular environment (no two optimal arms)} 
This principle applies broadly to most reinforcement learning algorithms used in practice, since they typically rely on summary statistics of past observations for decision-making---for instance, an $\epsilon$-greedy algorithm uses the empirical means of each arm. From this principle, we can identify sufficient conditions for policy convergence and draw practical guidance for policy design.
%From this principle, we can identify sufficient conditions for policy convergence that are easy to verify in practice. Our analysis also provides concrete guidance for designing stable adaptive policies. 
For example, policies that avoid relying on overly complex reward models for aggressive exploration---such as multi-armed bandit algorithms---tend to be more stable than those that tightly couple exploration to a potentially misspecified model. Likewise, continuous policies that are not overly steep (e.g., Boltzmann sampling with large temperature) are generally more reliable than policies with sharp or discontinuous decision boundaries (e.g., LinUCB). These insights are related to \citep{zhang2022statistical, zhang2024replicable}, which study policy convergence in a different regime with bounded horizons and many trajectories. %Other related work \citep{han2024ucb, halder2025stable} analyzes the stability of specific policies such as UCB and Thompson Sampling in multi-armed bandit setting without misspecification.
%In the simpler multi-armed bandit (MAB) setting without context, it is well known that many algorithms are stable in a “regular” environment (i.e. when there is a non-zero suboptimality gap between the optimal and suboptimal arms). Recent works have shown that even in an irregular environment with a zero suboptimality gap, online MAB algorithms (e.g., UCB \citep{han2024ucb} and a refined Thompson Sampling \citep{halder2025stable}) can be stable. In contrast, there has been little discussion of policy stability in the contextual bandit setting. MAB problem aims at mean of each arm, thereby always have the correctly specified outcome model, when the environemnt is stationary. It is not clear whether common contextual bandit algorithms in misspecified environments are unstable as it is the case in the batched setting \citep{zhang2024replicable}.

Building on these convergence insights, we then develop a general statistical inference framework based on an inverse-propensity-weighted Z-estimator (IPW-Z). This estimator targets a broad class of parameters (eq. (\ref{eq::theta-a-*})) and remains valid under mild conditions.
%Policy convergence provides the foundation for our inference framework. Given a converging behavior policy, we develop a general statistical inference method based on an inverse-probability-weighted Z-estimator (IPW-Z), applicable to a broad class of target parameters defined by the joint distribution of contexts and outcomes (see eq.(\ref{eq::theta-a-*})). 
It is well known that inference with adaptively collected data violates the usual i.i.d. assumptions, and prior work has relied on reweighting techniques together with martingale central limit theorems to obtain valid inference under correctly specified models \citep{zhang2020inference, zhang2021statistical, chen2021statistical, lin2023semi}. The problem becomes even more challenging without a well-specified model. In this setting, inverse probability weights are crucial, as they ensure that the components of the Z-estimator form a martingale difference sequence, thereby guaranteeing consistency. Policy convergence further stabilizes the conditional variances, enabling asymptotic normality. To complement these theoretical guarantees, we also examine the performance of the IPW-Z estimator through simulations in a range of complex environments. The results, presented in Section~\ref{sec::simulation}, show that the method delivers confidence intervals with reliable coverage and,  is at least as efficient as---often more efficient than---existing approaches that only exists in the off-policy evaluation setting, a special case of our framework.

\subsection{Our Contributions}

Our main contributions are summarized as follows:
\begin{itemize}
    \item \textbf{A new inference framework.} We propose an inverse-propensity-weighted Z-estimator (IPW-Z) for statistical inference with adaptively collected data, without assuming a correctly specified outcome model. To our knowledge, this is the first estimator to cover a broad class of target parameters (eq.~(\ref{eq::theta-a-*})) in this setting, and we establish its consistency and asymptotic normality under mild conditions.
    \item \textbf{Role of policy convergence.} We show that policy convergence---the stabilization of action-selection probabilities given the context---is fundamental for valid inference. A key finding is that widely used bandit algorithms such as LinUCB may fail to converge under misspecification, leading to estimator pathologies and invalid inference. This issue has been largely overlooked in prior work on adaptive inference for contextual bandits, despite its central importance for both inferential validity and replicability.
    \item\textbf{Principles for policy convergence.} We establish general conditions under which adaptive policies converge and derive practical design principles for stability. For example, simpler policies and smoother decision rules are more robust than policies that tightly couple exploration to complex or discontinuous reward models. To our knowledge, this is the first work to comprehensively study the convergence behavior of common policies in contextual bandits and to articulate general rules for when convergence can be expected.
    \item\textbf{Technical contributions.} We develop the theoretical foundations that make inference under misspecification possible. On the inference side, we develop variance-stabilization techniques that allow martingale CLTs to be applied to general inferential targets without requiring a well-specified outcome model. On the policy side, we establish convergence guarantees with several proof strategies that are largely novel in this context. One example is the application of stochastic approximation theory, which is nontrivial since the policies depend on summary statistics whose evolution dynamics may fail to contract; in such cases, we instead identify hidden parameters that yield contractive dynamics. Finally, we present numerical examples where standard contextual bandit algorithms interact with misspecified environments in irregular ways---such as oscillating or converging to multiple limits---that are both interesting in their own right and informative for future systematic research on policy behavior.
\end{itemize}

\subsection{Related work}

Statistical inference with adaptively collected data has been a long-standing but challenging problem. Although such data are common in practice, classical inference procedures designed for i.i.d. samples can break down under adaptive collection \citep{nie2018adaptively, zhang2020inference}. A growing body of work has therefore focused on establishing conditions on the data generating process or developing new methodologies to ensure valid inference. Most existing work focuses on estimating parameters in correctly specified outcome models. For example, \citep{anderson1979strong, christopeit1980strong, lai1982least} analyze adaptive linear regression under various conditions; \citep{klimko1978conditional, lai1994asymptotic} study least squares estimation in adaptive nonlinear regression; and \citep{chen1999strong} establishes asymptotic properties of maximum quasi-likelihood estimators for generalized linear models with adaptive designs.

With the emergence of modern data collection schemes such as reinforcement learning, recent work has developed inference procedures accommodating more general mechanisms such as contextual bandits, which are not covered by earlier results. \citep{deshpande2018accurate, khamaru2021near} propose online debiasing estimators for adaptive linear regression with weaker restrictions on the data collection process. \citep{zhang2020inference} study batched linear regression under general contextual bandit algorithms. \citep{chen2021statistical} develop inference for linear contextual bandits with $\epsilon$-greedy policies, and \citep{chen2021statistical1} extend their results to nonlinear reward models with parameter updates via weighted stochastic gradient descent. \citep{zhang2021statistical} establish inference for $M$-estimators under contextual bandit sampling, and \citep{lin2023semi} consider adaptive inference in generalized partial linear outcome models. \citep{syrgkanis2023post} study inference for structural parameters in structural nested mean models with data collected via reinforcement learning algorithms. 

In the absence of a well-specified reward model, several works have examined statistical inference with adaptively collected data for specific target estimands, primarily various forms of average outcomes. \citep{khamaru2024inference, han2024ucb} study inference on arm means under data collected by UCB-type algorithms, while \citep{halder2025stable} consider the same estimand with Thompson sampling variants. \citep{zhan2021off, hadad2021confidence, bibaut2021post, waudby2024anytime} analyze off-policy evaluation in bandits with general adaptive behavior policies, targeting the average reward of a specified evaluation policy. \citep{liao2021off, liao2022batch} develop inference for the long-run average reward in Markov decision processes with adaptive policies. Additional targets arising in longitudinal and causal panel data settings include mean responses to dynamic treatment regimes and average causal effects; see, for example, \citep{laan2003unified, chakraborty2013statistical}. 

Literature on inference for general target parameters beyond average outcomes without a well-specified reward model is comparatively sparse. A result closely related to ours is from \citep{chen2021statistical} who study inference in linear contextual bandits under both well-specified and misspecified reward models. In the misspecified case, they propose a weighted least squares estimator for the least false parameter, defined as the best linear projection of the true reward, and establish its asymptotic properties under an $\epsilon$-greedy behavior policy constructed from the same estimator. \citep{zhang2022statistical, zhang2024replicable} analyze inference after adaptive sampling in a general longitudinal data setting, but in a different regime where the time horizon is fixed and the number of trajectories diverges. 
Finally, a concurrent work \citep{LeinerDunnRamdas2025}, which generalizes \citep{bibaut2021post} on off-policy evaluation with adaptively collected data, investigates M-estimation with respect to a fixed target policy under model misspecification. Without requiring policy convergence, the authors achieve valid inference by reweighting with consistent estimators for a certain conditional variance sequence. However, their proposed estimators rely on consistent estimation of the true conditional moments of the score function given context and action, which can be difficult to guarantee under model misspecification. Moreover, they assume either access to independent external data or the ability to control the behavior policy so that the dataset $\mathcal{D}$ can be partitioned into two independent sub-trajectories.  %They also restrict attention to score functions that are linear in the outcome. 
In contrast, we take a different approach by directly considering converged policies, which requires no assumption on consistent moment estimation of the score function, access to  auxiliary data, or control over the data collection process.

% This work emphasizes the importance of policy stability in the context of statistical inference for adaptively collected data. In the literature, stochastic approximation methods have been widely used in Reinforcement Learning to analyze the convergence of TD learning and Q-learning \citep{borkar2000ode,lee2019unified,liu2025ode} under finite state and action spaces, or for linear function approximation \citep{carvalho2020new}. The former is under a correct model, therefore a direct application of the theory. The linear function approximation results typically assume bounded rewards, proper exploration/coverage, diminishing step sizes (or averaging), and either i.i.d. sampling or sufficiently mixing Markovian noise, sometimes with projections for stability. Current literature lack discussion on the misspecified model, the focus of this work.

This work highlights why \emph{policy convergence} matters for valid statistical inference with adaptively collected data. Stochastic approximation has been the main framework for convergence analyses of value-based RL, including TD and Q-learning in finite state-action spaces, and extensions that cover linear function approximation under additional stability and sampling conditions \citep{borkar2000ode,lee2019unified,carvalho2020new,liu2025ode}. These analyses typically assume correctly specified models. Although works on robustness and misspecification exist, they are less unified. For example, \citep{roy2017reinforcement} propose a robust version of Q-learning under model mismatch. We instead explored a broad family of policies that are stable without a well-specified model.

\paragraph*{Section layout} The remainder of the paper is organized as follows. Section~\ref{sec::problem-setup} introduces notation and the problem setup for adaptive inference in contextual bandits, along with three motivating examples: misspecified linear bandits, bandits with noisy contexts, and off-policy evaluation. Section~\ref{sec::inference-guarantee} presents the proposed inverse-probability-weighted Z-estimator (IPW-Z), establishes its consistency and asymptotic normality, and provides a consistent variance estimator. We also highlight the importance of policy convergence by illustrating, through a numerical example, how policy nonconvergence can induce nonnormality and pathological estimator behavior. Section~\ref{sec::policy-convergence} develops general sufficient conditions for policy convergence and introduces a broad family of policies that satisfy them, including multi-armed bandit algorithms that ignore context, policies based on the IPW-Z estimator, and Boltzmann exploration with ridge or stochastic gradient descent estimators. Section~\ref{sec::simulation} presents simulation studies that validate our asymptotic results across the three inference targets. We also compare with \citep{bibaut2021post} and \citep{zhan2021off} in the off-policy evaluation setting, showing that our method is more robust and often outperforms existing approaches.

\section{Problem Setup}\label{sec::problem-setup}

We consider the problem of statistical inference with an adaptively collected dataset $\cD = \{\bX_t, \pi_t, A_t, Y_t\}_{t=1}^T$ from a contextual bandit environment. The data collection process proceeds at each time $t$ as follows:
\begin{itemize}
    \item \textbf{Context:} The environment reveals a context $\bX_t\in \cX\subseteq\RR^{d_X}$.
    \item \textbf{Action Selection:} Based on the current context $\bX_t$ and past history $\cH_{t-1}:=  \{\bX_\tau, \pi_\tau,\\ A_\tau, Y_\tau\}_{\tau<t}$, the agent selects an action $A_t \in \cA$ according to a stochastic behavior policy $\pi_t(\cdot \mid \bX_t, \cH_{t-1}) \in \Delta(\cA)$, where $\Delta(\cA)$ denotes the set of probability distributions over the action space $\cA$. The realized selection probability is   recorded as $\pi_t := \pi_t(A_t \mid \bX_t, \cH_{t-1})$. %\ziping{For simplicity, we will use $\pi_t(\cdot) = \pi_t(\cdot \mid \bX_t, \cH_{t-1})$ as a shorthand.}
    \item \textbf{Outcome:} After choosing the action, the agent observes outcome $Y_t\in \RR$. 
\end{itemize}

We consider a finite action space $\cA$ and, without loss of generality, write $\cA = \{1, \ldots, K\}$. Adopting the potential outcomes framework \citep{imbens2015causal}, we let $\{Y_t(a): a \in \cA\}$ denote the potential outcomes for each action, with the observed outcome satisfying $Y_t = Y_t(A_t)$. We assume a stochastic contextual bandit environment in which $\{\bX_t, Y_t(a): a\in\cA\}\overset{\text{i.i.d.}}{\sim}\cP$, for $t = 1, \ldots, T$. In addition, in this adaptive experimental setting, we assume the following unconfoundedness condition.
\begin{assumption}\label{aspt:unconfoundedness}
    $A_t\perp \{Y_t(a)\}_{a\in\cA}|(\cH_{t-1}, \bX_t)$, for $t = 1, \ldots, T$.
\end{assumption}
Note that even though the potential outcomes are i.i.d., the observations in $\cD$ are not. This is because each action $A_t$ is selected based on the evolving history $\cH_{t-1}$, introducing temporal dependence into the observations. This dependence poses additional challenges for valid estimation and inference.

\subsection{Inference Targets}

Our goal is to infer the parameter $\btheta_a^* \in \RR^d$ associated with a treatment arm $a \in \cA$, or jointly the collection $\{\btheta_a^*\}_{a \in \cA}$. Each $\btheta_a^*$ is defined as the solution to the equation
\begin{equation}\label{eq::theta-a-*}
    \EE [\bg(\bX, Y(a); \btheta_a^*)] = \mathbf{0}
\end{equation}
for some known score function $\bg: \cX\times \RR\times \RR^d\rightarrow \RR^d$.%
\footnote{For simplicity, we assume a common score function $\bg$ across actions. Our analysis extends to action-specific score functions; see Appendix \ref{apdx::proof-thm::asymptotic-normality-joint-general} for details.} Here, $(\bX, Y(a))$ denotes a generic observation drawn from the same distribution as $(\bX_t, Y_t(a))$. Unlike many prior works on statistical inference with adaptively collected data \citep{lai1982least, zhang2021statistical, chen2021statistical1, lin2023semi, zhang2024replicable}, we do not assume a well-specified outcome model. In particular, the validity of our inference procedure does not rely on correctly modeling the conditional distribution of $Y_t$ given $A_t$ and $\bX_t$ in $\cP$. This feature makes our approach especially well-suited for adaptive experiments conducted in complex environments, where data may be limited, noise is non-negligible, and model misspecification is common. We illustrate this setup with three examples below.

\begin{example}[Misspecified linear bandits \citep{chen2021statistical}]\label{ex::misspecified-linear-bandits}
Consider a target parameter $\btheta_a^*$ which solves (\ref{eq::theta-a-*}) with 
\begin{equation}\label{eq::target-parameter-misspecified-linear-bandits}
\bg(\bx, y;\btheta) := \bx(y - \bx^\top \btheta).
\end{equation}
This score function corresponds to the best linear approximation of $Y_t(a)$ based on $\bX_t$ for arm $a \in \cA$. When the true dependence of $Y_t(a)$ on $\bX_t$ is linear, (\ref{eq::target-parameter-misspecified-linear-bandits}) yields the true linear parameter. Otherwise, (\ref{eq::target-parameter-misspecified-linear-bandits}) defines the best linear projection of $Y_t(a)$ onto the covariates $\bX_t$ in the least squares sense. 
\end{example}
\begin{example}[Bandits with noisy contexts \citep{guo2024online}]\label{ex::bandits-noisy-contexts}
Suppose the potential outcome $Y_t(a)$ follows a linear model based on the unobserved true covariates $\bS_t$:
$$
Y_t(a) = \bS_t^\top\btheta_a^* + \eta_t,
$$
where $\eta_t$ is a mean-zero noise term. Instead of observing $\bS_t$, the observed context $\bX_t$ is a noisy proxy $\bX_t = \bS_t + \bepsilon_t$, where $\bepsilon_t$ is mean zero, uncorrelated with $\eta_t$, and has covariance matrix $\bSigma_e$. Although the outcome model is linear in $\bS_t$, we do not assume any parametric form for the distribution of the measurement error $\bepsilon_t$, making it difficult to characterize the conditional distribution of $Y_t(a) \mid \bX_t$.

With a non-adaptive data collection process, this model has been well studied in statistics literature and is called the measurement error model \citep{carroll1995measurement, fuller2009measurement}. Assuming $\bSigma_e$ is known, then $\btheta_a^*$ solves (\ref{eq::theta-a-*}) with 
\begin{equation}\label{eq::target-parameter-bandits-noisy-contexts}
    \bg(\bx, y; \btheta) := \bx y - (\bx \bx^\top - \bSigma_e)\btheta.
\end{equation}
\end{example}

The measurement error model is motivated by practice in behavioral psychology. For example, in a mobile health study about reducing negative affect to improve medication adherence \citep{xu2025reinforcement}, the negative affect can only be measured through a short survey, a noisy proxy of the latent negative affect. However, the scientists are truly interested in the relationship between medication adherence and the latent negative affect rather than the noisy proxy.

\begin{example}[Off-policy evaluation in contextual bandits]\label{ex::ope}
Our goal is to estimate the average outcome under a target policy $\pi^e: \cX \to \Delta(\cA)$, defined as $V^* = \EE_{\bX}\EE_{A\sim \pi^e(\cdot|\bX)}Y(A)$. This target can be expressed as $V^* = \sum_{a \in \cA} \btheta_a^*$, where each $\btheta_a^*$ solves (\ref{eq::theta-a-*}) with an arm-specific score function $\bg$:
\begin{equation}\label{eq::target-parameter-ope}
\bg(\bx, y; \btheta) = \bg_a(\bx, y; \btheta) := \pi^e(a|\bx)y - \btheta.
\end{equation}
\end{example}
These examples will be discussed in more detail in Section \ref{seq::examples}.

\subsection{Challenge}
A central challenge in the absence of a well-specified outcome model is that standard Z-estimation approaches \citep{zhang2021statistical, lin2023semi, zhang2024replicable}, which analyze estimators $\hat\btheta_a'$ satisfying
\begin{equation}\label{eq::naive-z-estimator}
\frac1T\sum_{t=1}^T 1_{\{A_t = a\}}\bg\big(\bX_t, Y_t, \hat\btheta_a'\big) = o_p\big(1/\sqrt{T}\big),
\end{equation}
fail to yield valid inference. This failure essentially stems from the interaction between the policy and the complex environment. Specifically, for a general distribution $\cP$, the solution $\btheta$ to the conditional moment equation $\EE[\bg(\bX, Y(a); \btheta) \mid \bX = \bx] = \mathbf{0}$ can vary with $\bx$. As a result, the solution to (\ref{eq::naive-z-estimator}) depends on the behavior policy used to collect the data, and is generally not consistent. This issue, as it arises in Examples \ref{ex::misspecified-linear-bandits} and \ref{ex::bandits-noisy-contexts}, is discussed in detail in \citep{chen2021statistical} and \citep{guo2024online}, respectively.

In this work, we develop statistical inference of the target parameters $\{\btheta_a^*\}_{a\in\cA}$ by studying the asymptotic properties of the inverse probability weighted Z-estimators $\{\hat\btheta_a^{(T)}\}_{a\in\cA}$, where $\hat\btheta_a^{(T)}$ satisfies 
\begin{equation}\label{eq::estimating-equation-general}
    \bG_T(\btheta) := \frac{1}{T} \sum_{t=1}^T \frac{1}{\pi_t(A_t)} 1_{\{A_t = a\}} \bg(\bX_t, Y_t; \btheta) = o_p\big(1/\sqrt{T}\big).
\end{equation}
Here, $\pi_t(a)$ abbreviates the action selection probability $\pi_t(a \mid \cH_{t-1}, \bX_t)$ for $a\in\cA$. The use of inverse probability weights $1/\pi_t(A_t)$ is standard in the off-policy evaluation literature (e.g., \citep{li2011unbiased, uehara2022review}) and has also been employed in specific adaptive settings by \citep{chen2021statistical} and \citep{guo2024online} for various purposes. Our goal is to show that, in our general setting, the inverse probability weights effectively decouple the policy from the underlying environment, enabling valid inference. Specifically, we will establish the joint asymptotic normality of $\{\hat\btheta_a^{(T)}\}_{a \in \cA}$ under mild conditions on the environment and for a broad class of behavior policies.

\section{Statistical Inference Guarantees}
\label{sec::inference-guarantee}
The main results of this section establish the joint asymptotic normality of the proposed estimators $\{\hat\btheta_a^{(T)}\}_{a \in \cA}$, along with consistent estimators of the asymptotic variance, which enable statistical inference for $\{\btheta_a^*\}_{a \in \cA}$. To achieve these guarantees, we highlight the critical role of policy convergence (Definition~\ref{aspt:policy_convergence}), an important condition on the behavior policy used to collect the data. In Section \ref{sec::example-policy-nonconvergence}, we present a concrete example demonstrating how the failure of this condition can lead to the breakdown of asymptotic normality. We begin by formally stating the condition below.

\begin{definition}[Policy convergence] 
\label{aspt:policy_convergence} 
The behavior policy $\pi = \{\pi_t(\cdot)\}_{t\geq 1}$ is said to satisfy policy convergence, or to converge, at action $a \in \cA$ if there exists a stationary policy $\bar\pi: \cX \mapsto \Delta(\cA)$ such that 
\begin{equation}\label{eq::def-policy-convergence}
\pi_t(a|\bX_t, \cH_{t-1}) - \bar\pi(a|\bX_t)\xrightarrow{p} 0\quad\text{ as }t\rightarrow\infty.
\end{equation}
We say that $\pi$ satisfies policy convergence---or simply converges---if (\ref{eq::def-policy-convergence}) holds for every action $a \in \cA$ with respect to a common stationary policy $\bar\pi$.
\end{definition}

The policy convergence condition essentially requires that, in the long run, the action-selection probabilities stabilize in probability across the context. Beyond its practical importance, this condition plays a central theoretical role: it ensures that the conditional variances of each term in the estimating equation (\ref{eq::estimating-equation-general}) given past history stabilize, which guarantees the martingale central limit theorems and thereby yields the desired asymptotic normality of the estimators (see Appendix \ref{apdx::proof-thm::asymptotic-normality-joint-general} for details). In Section~\ref{sec::policy-convergence}, we further investigate sufficient conditions for policy convergence in general environments without a well-specified reward model. These conditions are straightforward to verify, offering concrete guidance for implementing adaptive policies in online settings.

\begin{remark}
Policy convergence, or related notions of policy stability, is a common assumption in the literature on inference with adaptively collected data---particularly in the absence of a well-specified model. For instance, in misspecified linear bandits (Example~\ref{ex::misspecified-linear-bandits}), \citep{chen2021statistical} study inference under an $\epsilon$-greedy algorithm with a weighted online least squares (LS) estimator---a special case that satisfies policy convergence in their setting (see Section~\ref{sec::policy-convergence}). Similar convergence conditions are also assumed in other adaptive inference problems under model misspecification \citep{zhang2022statistical, zhan2021off, zhang2024replicable}. In settings without model misspecification, prior work has shown that various forms of stability conditions lead to valid statistical inference. These conditions are typically weaker than, but often closely related to, policy convergence. For example, \citep{lai1982least} studies stochastic linear regression and derives asymptotic normality of the OLS estimator under a stability condition that requires the design matrix to behave regularly over time. \citep{khamaru2024inference, han2024ucb} analyze the UCB algorithm in multi-armed bandits and show valid inference under stability conditions where the ratio between the number of arm pulls and a diverging sequence converges to one.
\end{remark} 

We next introduce the technical assumptions required to establish the asymptotic properties of the estimator $\hat\btheta_a^{(T)}$ for a single action $a \in \cA$.

\begin{assumption}[Well-separated solution]\label{aspt:identifiability-general}
    $\forall \epsilon>0$, $\inf_{\|\btheta - \btheta_a^*\|_2>\epsilon}\|\EE[\bg(\bX_t, Y_t(a);\btheta)]\|_2>0$.
\end{assumption}

\begin{assumption}[Bounded moments]\label{aspt:boundedness-general}
    There exist constants $R_\theta$, $M_2$ such that \\
    (i) $\|\EE[\bg(\bX_t, Y_t(a);\btheta_a^*)\bg(\bX_t, Y_t(a);\btheta_a^*)^\top|\bX_t]\|_2\leq M_2$, a.e. $\bX_t$; \\
    (ii) $\|\btheta_a^*\|_2<R_\theta$, $\sup_{\|\btheta\|_2\leq R_\theta}\EE[\|\bg(\bX_t, Y_t(a);\btheta)\|_2^2]<\infty$;
    (iii) $\EE [\|\bg(\bX_t, Y_t(a);\btheta_a^*)\|_2^4]<\infty$.
\end{assumption}

\begin{assumption}[Smoothness]\label{aspt:smoothness-general}
(i) The function $\bg(\bx, y; \btheta)$ is twice differentiable with respect to $\btheta$, with $\EE[\nabla\bg(\bX_t, Y_t(a);\btheta_a^*)]$ nonsingular; 
(ii) There exists a function $\phi: \RR^{d_X}\times \RR\mapsto \RR$ such that $\forall \bx, y$, $\sup_{\|\btheta\|_2\leq R_\theta}\|\nabla\bg(\bx, y; \btheta)\|_2\leq \phi(\bx, y)$, and $\EE[\phi(\bX_t, Y_t(a))^2]<\infty$; 
(iii) There exists a constant $\epsilon_0>0$ and a function $\Phi: \RR^{d_X}\times \RR\mapsto \RR$ such that $\sup_{\|\btheta - \btheta_a^*\|_2\leq \epsilon_0, i\in[d]}\|\nabla^2\bg^{(i)}(\bx, y; \btheta)\|_2\leq\Phi(\bx, y)$ and $\EE[\Phi(\bX_t, Y_t(a))]<\infty$. Here $\bg^{(i)}(\bx, y; \btheta)$ denotes the $i$-th entry of $\bg(\bx, y; \btheta)$.
\end{assumption}

\begin{assumption}[Minimum sampling probability] \label{aspt:min-sampling-prob}
$\pi_t(a)\geq \pi_{\min}$ almost surely for some constant $\pi_{\min} \in(0, 1)$.
\end{assumption}

\begin{remark}
We provide a few comments on these assumptions. First, Assumption \ref{aspt:identifiability-general} is a standard condition in Z-estimation that ensures the identifiability of the target parameter $\btheta_a^*$ \citep{van2000asymptotic}. Assumption \ref{aspt:boundedness-general} imposes mild regularity conditions on the boundedness of moments and conditional moments of the score function $\bg(\bX_t, Y_t(a); \btheta_a^*)$ evaluated at the true parameter $\btheta_a^*$. Notably, it does not require the potential outcomes $Y_t(a)$ themselves to be bounded or to satisfy sub-Gaussian tail conditions. Similar assumptions appear in related works, such as \citep{zhan2021off, chen2021statistical, zhang2021statistical, zhang2024replicable}. Assumption \ref{aspt:smoothness-general} imposes smoothness conditions on the score function $\bg$, a standard requirement in classical Z-estimation as well as in recent work on Z- and M-estimation with adaptively collected data \citep{zhang2021statistical, zhang2022statistical}. Finally, Assumption \ref{aspt:min-sampling-prob} imposes a minimum sampling probability, which is frequently assumed in adaptive inference without a well-specified outcome model (e.g., \citep{chen2021statistical, zhang2021statistical, zhang2022statistical}). In practice, in adaptive experiments with highly noisy and complex environment, keeping a minimum exploration rate ensures statistical power for flexible post-hoc analysis \citep{yao2021power, lauffenburger2024impact}, particularly when the analysis objective is not pre-specified at the time of data collection. It also enables policy updates and re-optimization for future users, accommodating potential non-stationarity across trials \citep{liao2020personalized, yang2024targeting}.
\end{remark}

We now state the first main result of this section, which establishes the asymptotic properties of the estimator $\hat\btheta_a^{(T)}$ for a single action $a \in \cA$. The proof is provided in Appendix~\ref{apdx::proof-thm::asymptotic-normality-general}.

\begin{theorem}\label{thm::asymptotic-normality-general}
Suppose Assumptions \ref{aspt:unconfoundedness} and  \ref{aspt:identifiability-general}--\ref{aspt:min-sampling-prob} hold for a given action $a\in\cA$. If the behavior policy $\pi$ converges to a policy $\bar \pi$ at action $a$ in the sense of Definition \ref{aspt:policy_convergence}, then there exists an estimator sequence $\{\hat\btheta_a^{(T)}\}_{T\geq 1}$ satisfying the estimating equation~\eqref{eq::estimating-equation-general} with $\|\hat\btheta_a^{(T)}\|_2 \leq R_\theta$ for all $T$. Moreover, for any such sequence, as $T\rightarrow \infty$, 
\begin{equation}\label{eq::asymptotic-normality-general}
    \sqrt{T}(\hat\btheta_a^{(T)} - \btheta_a^*)\xrightarrow{d} \cN\left(\mathbf{0}, \bSigma_{a}^*\right).
\end{equation}
Here $\bSigma_{a}^*:= \bJ_a^{-1}\bar \bI_a \bJ_a^{-1, \top}$, with
$\bar\bI_a := \EE\big[\frac{1}{\bar\pi(a|\bX_t)}\bg(\bX_t, Y_t(a);\btheta_a^*)\bg(\bX_t, Y_t(a);\btheta_a^*)^\top\big]$, \\
$\bJ_a:= \EE[\nabla\bg(\bX_t, Y_t(a);\btheta_a^*)]$.
\end{theorem}

\emph{The role of inverse probability weights.} To build intuition for Theorem \ref{thm::asymptotic-normality-general}, we present a heuristic argument illustrating how the inverse probability weights in the estimating equation (\ref{eq::estimating-equation-general}), which defines our estimator $\hat\btheta_a^{(T)}$, contribute to achieving valid inference. Specifically, a key step in the proof of Theorem \ref{thm::asymptotic-normality-general} is to show that 
$
\bG_T(\btheta_a^*) = \frac1T\sum_{t=1}^T \bZ_t
$ 
forms a sum of a martingale difference sequence with respect to the filtration $\{\cH_{t}\}_{t=0}^T$, where 
$$
\bZ_t:= \frac{1}{\pi_t(A_t)} 1_{\{A_t = a\}} \bg(\bX_t, Y_t; \btheta_a^*).
$$ 
Define $w_a(A_t):= \frac{1}{\pi_t(A_t)}1_{\{A_t = a\}}$. Then we have
\begin{align}
\EE[\bZ_t|\cH_{t-1}]
&\stackrel{(a)}{=} \EE_{\bX_t}\!\left[\EE_{A_t\sim \pi_t(\cdot), Y_t(a)}[\bZ_t|\cH_{t-1},\! \bX_t]\Big|\cH_{t-1}\right]\nonumber\\
& \stackrel{(b)}{=} \EE_{\bX_t}\Big[\EE_{A_t\sim \pi_t(\cdot)}\big[w_a(A_t)\big|\cH_{t-1}, \!\bX_t\big]\!\cdot \!\EE_{Y_t(a)}[\bg(\bX_t, Y_t(a);\btheta_a^*)|\cH_{t-1},\! \bX_t]\Big|\cH_{t-1}\Big]\nonumber\\
& \stackrel{(c)}{=} \EE_{\bX_t}\!\left[1\cdot \EE_{Y_t(a)}[\bg(\bX_t, Y_t(a);\btheta_a^*)|\cH_{t-1},\! \bX_t]\Big|\cH_{t-1}\right]\nonumber\\
& \stackrel{(d)}{=} \EE\left[\bg(\bX_t, Y_t(a);\btheta_a^*)\right] = \mathbf{0}.\nonumber
\end{align}
Here, step (a) follows from the law of iterated expectations. Step (b) uses Assumption~\ref{aspt:unconfoundedness}, which ensures that $A_t$ and $Y_t(a)$ are independent condition on $(\cH_{t-1}, \bX_t)$. Step (c) follows from a direct computation:
$$
\EE_{A_t\sim \pi_t(\cdot)}\big[w_a(A_t)\big|\cH_{t-1}, \!\bX_t\big] = \sum_{a'\in\cA}\pi_t(a')\cdot \frac{1_{\{a'=a\}}}{\pi_t(a')} = 1.
$$
Step (d) again applies the law of iterated expectations along with the assumption that $\{\bX_t, Y_t(a): a\in\cA\}$ are i.i.d. over time. In contrast, if the estimator is derived without incorporating inverse probability weights---as in (\ref{eq::naive-z-estimator})---we instead have
$$
\EE[\bZ_t|\cH_{t-1}]=\EE_{\bX_t}\!\left[\pi_t(a|\cH_{t-1},\! \bX_t)\cdot \EE_{Y_t(a)}[\bg(\bX_t, Y_t(a);\btheta_a^*)|\cH_{t-1},\! \bX_t]\Big|\cH_{t-1}\right],
$$
which clearly depends on the behavior policy $\pi_t(\cdot)$ and is generally nonzero. This in turn prevents the unweighted Z-estimator in (\ref{eq::naive-z-estimator}) from being consistent.

\emph{Comparison to prior work.} In the setting of misspecified linear bandits (Example \ref{ex::misspecified-linear-bandits}), Theorem 4.1 of \citep{chen2021statistical} establishes the asymptotic normality of $\hat\btheta_a^{(T)}$ under a specific converging behavior policy: an $\epsilon$-greedy policy paired with a weighted online LS estimator. In contrast, Theorem \ref{thm::asymptotic-normality-general} applies to any online decision-making algorithm that satisfies policy convergence and a minimum sampling probability condition. Section \ref{sec::policy-convergence} demonstrates that a broad class of policies meet these requirements. In practice, statisticians seeking to conduct inference often do not control the algorithm used to collect the data. Therefore, Theorem \ref{thm::asymptotic-normality-general} offers a broader and more flexible generalization of Theorem 4.1 in \citep{chen2021statistical}, both in terms of the allowable data collection mechanisms and the inferential targets. %\YG{Do we need to compare to OPE here? They use completely different inference methods so we could briefly mention if we empirically outperform them otherwise we can just skip this part for now? We need to deal with this in section 5 anyways}

In practice, it is often desirable to jointly infer the collection $\{\btheta_a^*\}_{a\in\cA}$, where for each arm $a\in\cA= \{1, \ldots, K\}$, $\btheta_a^*$ denotes the solution to (\ref{eq::theta-a-*}). Such joint inference is particularly relevant for tasks like estimating individual treatment effects or evaluating the value of a general target policy. To achieve this goal, Theorem \ref{thm::asymptotic-normality-joint-general} below establishes the joint asymptotic normality of the estimators $\{\hat\btheta_a^{(T)}\}_{a\in\cA}$, with the proof provided in Appendix \ref{apdx::proof-thm::asymptotic-normality-joint-general}. For simplicity, we assume a common score function $\bg$ across actions, though our analysis extends to action-specific score functions; see Appendix \ref{apdx::proof-thm::asymptotic-normality-joint-general} for details. 

\begin{theorem}\label{thm::asymptotic-normality-joint-general}
Suppose Assumption \ref{aspt:unconfoundedness} holds, and Assumptions \ref{aspt:identifiability-general}--\ref{aspt:min-sampling-prob} hold for every action $a\in\cA$. If the behavior policy $\pi$ satisfies policy convergence to a policy $\bar \pi$ in the sense of Definition \ref{aspt:policy_convergence}, then there exist estimators $\{\hat\btheta_a^{(T)}\}_{a\in\cA, T\geq 1}$ such that (\ref{eq::estimating-equation-general}) holds for each $a\in\cA$, and $\|\hat\btheta_a^{(T)}\|_2\leq R_\theta$ for all $a\in\cA, T\geq 1$. In addition, any such estimators satisfy
\begin{equation}\label{eq::asymptotic-normality-joint-general}
    \sqrt{T}(\hat\btheta^{(T)} - \btheta^*)\xrightarrow{d}\cN
    \left(
    \mathbf{0},\bSigma^*
    \right)
\end{equation}
as $T\rightarrow \infty$. Here $\hat\btheta^{(T)} = \big((\hat\btheta_1^{(T)})^\top, \ldots, (\hat\btheta_K^{(T)})^\top\big)^\top$, $\btheta^* = \big(\btheta_1^{*, \top}, \ldots, \btheta_K^{*, \top}\big)^\top$, and $\bSigma^* = \diag(\bSigma_1^*, \ldots, \bSigma_K^*)$, with each $\bSigma_a^*$ defined as in Theorem \ref{thm::asymptotic-normality-general}.
\end{theorem}

Theorem \ref{thm::asymptotic-normality-joint-general} indicates that the estimators $\{\hat\btheta_a^{(T)}\}_{a\in\cA}$ are asymptotically uncorrelated. Intuitively, this is because each $\hat\btheta_a^{(T)}$ is constructed using data exclusively from time points when action $a$ is selected, and these sets of time points are disjoint across different actions.

The asymptotic variances in Theorems \ref{thm::asymptotic-normality-general} and \ref{thm::asymptotic-normality-joint-general} can be consistently estimated from data, as shown in the proposition below, thereby enabling valid statistical inference. The proof is provided in Appendix \ref{apdx::proof-thm::consistent-var-estimator-general}.

\begin{proposition}\label{thm::consistent-var-estimator-general}
Under the same conditions of Theorem \ref{thm::asymptotic-normality-general}, the asymptotic variance $\bSigma_a^*$ can be consistently estimated by 
\begin{equation}\label{eq::estimator-asympt-var}
    \hat\bSigma_a = \left[\hat{\dot{\bG}}_{a, T}\right]^{-1}\hat \bI_{a, T}\left[\hat{\dot{\bG}}_{a, T}\right]^{-1, \top},
\end{equation}
where 
\begin{align*}
\hat{\dot{\bG}}_{a, T} &:= \frac1T\sum_{t=1}^T\frac1{\pi_t(A_t)}1_{\{A_t = a\}}\nabla\bg(\bX_t, Y_t; \hat \btheta_a^{(T)}),\\
\hat\bI_{a, T} & := \frac1T\sum_{t=1}^T\frac{1}{\pi_t(A_t)^2}1_{\{A_t = a\}}\bg(\bX_t, Y_t; \hat \btheta_a^{(T)})\bg(\bX_t, Y_t; \hat \btheta_a^{(T)})^\top.
\end{align*}
\end{proposition}

Importantly, Proposition \ref{thm::consistent-var-estimator-general} shows that consistent estimation of the asymptotic variance does not rely on knowledge of the limit policy $\bar\pi$.

\subsection{Examples of Policy Non-convergence}\label{sec::example-policy-nonconvergence}

% \YG{Probably also cite some examples/arguments that in the special case of off-policy evaluation, IPW/AIPW estimator can be nonnormal without any constraint on the behavior policy e.g. in \citep{hadad2021confidence, bibaut2021post}?} \ziping{should I broadly mention the role of stability here?}

% \ziping{we know IPW can be non-normal without minimum sampling probability because the inverse probability weight can be unbounded. Mention Koulik's paper on policy non-convergence leading to non-normal. We give an example of IPW non-normal because of policy non-convergence.}

Previous work has shown that for off-policy evaluation with adaptively collected data, inverse probability weighting (IPW) estimators can exhibit non-normal asymptotic behavior when no constraints are imposed on the behavior policy \citep{hadad2021confidence, bibaut2021post}---for instance, when the behavior policy lacks a minimum sampling probability. In these works, the authors attribute such non-normality to the presence of unbounded inverse probability weights. In this section, we highlight another fundamental but less emphasized contributing factor to the non-normality of adaptive inference: \emph{policy non-convergence}. Using a concrete numerical example, we show that under a misspecified environment, even standard policies may fail to converge, and such non-convergence may in turn lead to non-normal asymptotic behavior of IPW-Z estimators. Similar phenomena have been observed in the context of estimating arm means with Thompson Sampling behavior policies \citep{halder2025stable}.

% In this section, we show that the non-convergence issue may arise from model misspecification under many commonly used contextual bandit algorithms, such as LinUCB \citep{chu2011contextual} and Thompson Sampling \citep{russo2018tutorial} even if the sampling probability is bounded below.

% Non-no of behavior policy has been a long-standing issue in the contextual bandit literature. Previous works have shown that the non-convergence issue may arise when 

% In the case of off-policy evaluation, IPW/AIPW estimators can be non-normal \citep{hadad2021confidence, bibaut2021post}. In this paper, we show that the non-convergence issue may arise from model misspecification under many commonly used contextual bandit algorithms, such as LinUCB \citep{chu2011contextual} and Thompson Sampling \citep{russo2018tutorial} even if the sampling probability is bounded below. 
To illustrate, we consider a two-armed bandit environment with contexts generated as $\bX_t \sim \text{Uniform}(\{-4, 1\})$ and mean reward $y(\bx, a) = \mathbb{E}[Y_t(a) | \bX_t = \bx]$ given by 
$$
    y(-4, a_0) = y(1, a_0) = 1/2, \quad y(-4, a_1) = y(1, a_1) = 1/12.
$$
The rewards are subject to Gaussian noise. Under this setting, we independently run two algorithms---LinUCB and Random---across $2500$ replications, each for $10^4$ steps. LinUCB employs a working reward model $Y_t(a) = \btheta_{a}^{\top} \bX_t$ with unknown parameter $\btheta_{a} \in \mathbb{R}$, and its action probabilities are clipped at $0.01$ to ensure bounded inverse probability weights. The Random policy selects each action independently with equal probability $1/2$ at every time step, and by construction it naturally satisfies policy convergence as in Definition \ref{aspt:policy_convergence}.

In Figure \ref{fig:non-convergence}, we plot the distributions of key quantities across the $2500$ replications: (a) the last-step ridge regression estimator for $\btheta_a$ (panel a), (b) the sampling probabilities at context $\bx = -4$ (panel b), and (c) the proposed IPW-Z estimator defined in (\ref{eq::estimating-equation-general}) for the target parameter in (\ref{eq::target-parameter-misspecified-linear-bandits}) (panel c). In panel (d), we present a QQ plot comparing the empirical distribution of the IPW-Z estimator with the standard Gaussian distribution. We observe that (i) the ridge regression estimator under LinUCB exhibits two distinct convergence modes, evident from the bimodal histogram in panel a, which does not occur under the Random algorithm; (ii) the LinUCB policy itself fails to converge, as shown by the bimodal distribution of the sampling probability at $x=-4$ in panel b, in contrast to the Random algorithm; and (iii) because the LinUCB policy does not converge, the asymptotic distribution of the IPW-Z estimator substantially deviates from Gaussianity, whereas the estimator based on data collected by the convergent Random algorithm remains asymptotically normal (panels c and d).

\begin{figure}[tb]
    \centering
    \includegraphics[width=\linewidth]{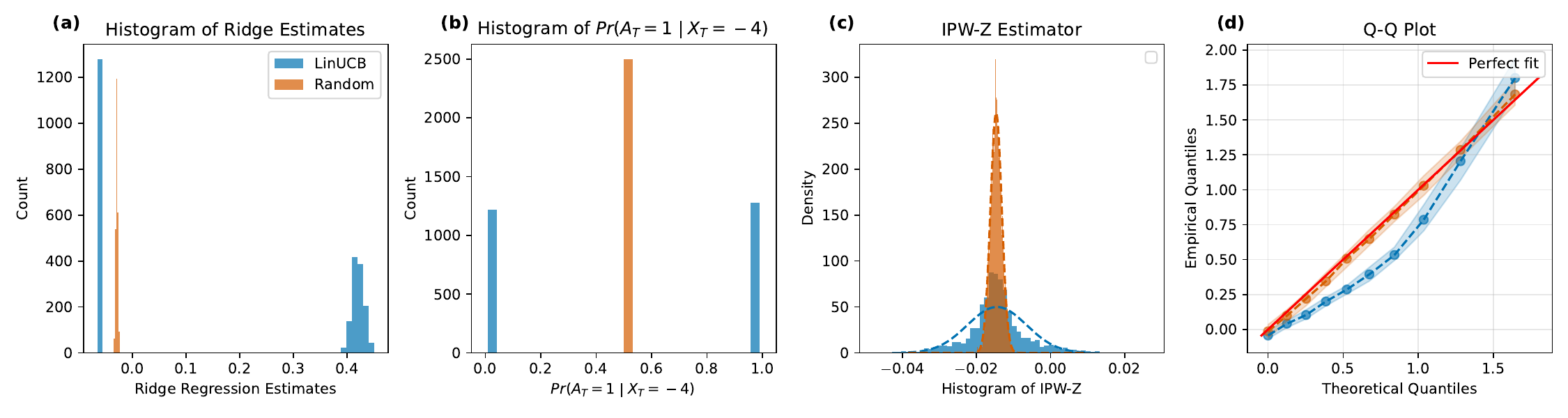}
    \caption{Example of policy non-convergence and non-normality of the IPW-Z estimator (\ref{eq::estimating-equation-general}). We independently run the LinUCB and Random algorithms for $10{,}000$ steps over $2{,}500$ replications. \textbf{(a)} Last-step ridge regression estimator of $\btheta_{1}$. \textbf{(b)} Last-step sampling probability at context $\bx = -4$. \textbf{(c)} The IPW-Z estimator (\ref{eq::estimating-equation-general}) for the inference target in (\ref{ex::misspecified-linear-bandits}). \textbf{(d)} QQ plot of the standardized empirical distribution of the IPW-Z estimator compared with the standard Gaussian distribution. Results shown in blue correspond to LinUCB, and results shown in orange correspond to Random.}
    \label{fig:non-convergence}
\end{figure}

\section{Sufficient Conditions for Policy Convergence}\label{sec::policy-convergence}

In this section, we first study sufficient conditions under which a policy satisfies the policy convergence condition in Definition~\ref{aspt:policy_convergence}. Then, in Sections \ref{sec::MAB-policy-convergence} to \ref{sec::boltzmann-policy-convergence}, we illustrate several classes of convergent policies in general environments without assuming a well-specified reward model. These results offer a principled foundation for constructing stable bandit policies in practice, particularly in noisy and complex settings.

We focus on behavior policies that belong to a parametric class of the form $\pi(\cdot|\bX_t, \bbeta)$, where $\bbeta\in\RR^{d_{\beta}}$. At each time $t$, the behavior policy $\pi_t(a|\bX_t, \cH_{t-1})$ takes the form 
\begin{equation}\label{eq::policy-statistics}
\pi(a|\bX_t, \hat\bbeta_{t-1})
\end{equation}
for a fixed function $\pi: \cX\times \mathbb R^{d_{\beta}}\mapsto \Delta(\cA)$, where $\hat\bbeta_{t-1}\in\mathbb R^{d_{\beta}}$ is a $\cH_{t-1}$-measurable random vector. This vector can be interpreted as a \emph{summary statistic} that aggregates information from the past $t-1$ rounds and, together with the current context $\bX_t$, determines the action selection probability at time $t$. This policy class is broad and encompasses many commonly used bandit algorithms---such as 
$\epsilon$-greedy \citep{sutton1998reinforcement}, UCB \citep{auer2002finite, chu2011contextual}, and Thompson sampling \citep{agrawal2012analysis, russo2018tutorial}---and is widely adopted in adaptive experimental designs in practice \citep{liao2020personalized, athey2022contextual, xu2025reinforcement}.

The following theorem provides a sufficient condition under which policies of the form (\ref{eq::policy-statistics}) satisfy the policy convergence condition in general environments, without requiring a well-specified reward model. The proof is given in Appendix \ref{apdx::proof-lem::statistics-converge-implies-policy-converge}.

\begin{theorem}\label{lem::statistics-converge-implies-policy-converge}
Suppose the behavior policy $\pi_t(a|\bX_t, \cH_{t-1})$ takes the form of (\ref{eq::policy-statistics}). Then for any $a\in\cA$, as long as 
\begin{itemize}
    \item[(i)] $\hat\bbeta_t\xrightarrow{p} \bbeta^*$ for some $\bbeta^*\in \RR^{d_{\beta}}$,
    \item[(ii)] $\pi(a|\bX_t, \cdot)$ is continuous at $\bbeta^*$ a.e. $\bX_t$,
\end{itemize}
the behavior policy converges at action $a$ as in Definition \ref{aspt:policy_convergence}, with the limit policy $\bar\pi(a|\bx) = \pi(a|\bx, \bbeta^*)$.
\end{theorem}

We offer some remarks on conditions (i) and (ii) in Theorem \ref{lem::statistics-converge-implies-policy-converge}. Condition (i) is notably general: it only requires $\hat\bbeta_t$ converges to a deterministic limit $\bbeta^*$, which needs not correspond to any parameter from a correctly specified model, nor be optimal in any sense. In Sections \ref{sec::MAB-policy-convergence} to \ref{sec::boltzmann-policy-convergence}, we present several examples where the summary statistic $\hat\bbeta_t$ converges to different limits through a variety of mechanisms. Condition (ii) imposes only local continuity of the policy function $\pi$ at $\bbeta^*$, rather than requiring global continuity. This requirement is important: even if $\hat\bbeta_t \xrightarrow{p} \bbeta^*$, discontinuity of $\pi$ at $\bbeta^*$ may result in irregular behavior of the policy and prevent convergence. 

In the following, we apply Theorem \ref{lem::statistics-converge-implies-policy-converge} and show several classes of behavior policies satisfy the policy convergence condition in the sense of Definition \ref{aspt:policy_convergence}. For each policy, we identify suitable summary statistics and verify the continuity of the policy function at the limiting value. Importantly, all these policies converge in general environments under mild regularity conditions, without requiring a well-specified reward model.

% \paragraph{Uniform PAC with unique optimal policy.} Online learning algorithm in the literature is typically measured by \textit{Probably Approximately Correct} (PAC) or \textit{Regret}. However, as discussed in \cite{dann2017unifying}, none of these measures guarantee convergence to the optimal policy (when these exists a unique one). They propose Uniform-PAC, a stronger performance measurement that guarantees convergence. 

% \begin{definition}[Uniform-PAC]
%     Let $N_{\epsilon} = \sum_{t = 1}^{\infty} \one\{\Delta_t \geq \epsilon\}$. An algorithm is Uniform-PAC, if for all $\delta > 0$, 
%     $$
%         \PP(\text{ exists } \epsilon > 0: N_{\epsilon}  > F(d, |\mathcal{A}|, 1/\epsilon, \log(1/\delta))) \leq \delta,
%     $$
%     for some function $F$ that is polynomial in all arguments.
% \end{definition}

\subsection{Multi-armed bandit ignoring context}\label{sec::MAB-policy-convergence} 
In real-world bandit deployments, simple algorithms like multi-armed bandits (MAB) are often preferred, especially in early trials when prior information is limited. By ignoring context, MAB avoids reward model misspecification, which helps to manage variance and prevent poor decisions in complex systems. In this section, we show that under mild conditions, common MAB algorithms additionally satisfy the policy convergence condition, enabling valid post-study inference.

For any action $a\in\cA$ and time $t$, denote
\begin{equation}\label{eq::count-running-avg-reward-for-each-arm}
\hat\mu_{a, t}: =
\frac
{\sum_{\tau = 1}^t1_{\{A_\tau = a\}}Y_{\tau}}
{N_{a, t}},\quad N_{a, t} : = \sum_{\tau = 1}^t1_{\{A_\tau = a\}}.
\end{equation}
Consider the following MAB algorithms with a minimum sampling probability: 

\begin{itemize}
    \item \textbf{The $\epsilon$-greedy algorithm:}
        \begin{align}\label{eq::policy-eps-greedy-mab}
        \pi_t^{\epsilon\text{-greedy}}(a|\mathcal H_{t-1})= 
        \begin{cases}
            1-\frac{K-1}{K}\epsilon,\quad &\text{if }i = \argmax_i \hat\mu_{i, t-1},\\
            \frac1K\epsilon, \quad &\text{otherwise.}
        \end{cases}
    \end{align}
    \item \textbf{The UCB algorithm:}
    \begin{align}\label{eq::policy-ucb-mab}
        \pi_t^{\text{UCB}}(a|\mathcal H_{t-1})= 
        \begin{cases}
            1-(K-1)\pi_{\min},\quad &\text{if }i \coloneq \argmax_i \left\{\hat\mu_{i, t-1} + \sqrt{\frac{C_t}{N_{i, t-1}}}\right\},\\
            \pi_{\min}, \quad &\text{otherwise.}
        \end{cases}
        \end{align}
        where $\{C_t\}_{t\geq 1}$ is any deterministic sequence such that $\lim_{t\rightarrow\infty}\frac{C_t}{t} = 0$, e.g., $C_t = 2\log t$ \citep{auer2010ucb}.
    \item \textbf{The Thompson Sampling algorithm:}   
    \begin{align}\label{eq::policy-ts-mab}
        \left(\pi_t^{\text{TS}}(a|\mathcal H_{t-1})\right)_{a\in\cA}=\mathrm{Clip}\left(\left(\bar\pi_t^{\text{TS}}(a|\mathcal H_{t-1})\right)_{a\in\cA}\right),
    \end{align}
    where    
    \begin{itemize}
        \item
    $
    \bar\pi_t^{\text{TS}}(a|\mathcal H_{t-1}) \coloneq  \EE_{(\mu_i')_{i\in\cA}\sim \mathcal N(\bmu_{t-1}^{\mathrm{post}}, \bSigma_{t-1}^{\mathrm{post}})}1_{\{\forall i\neq a, \mu_i'<\mu_a'\}}
    $
    is the posterior probability of action $a$ being optimal under a Gaussian prior. Here
    \begin{gather}
    \bmu_{t-1}^{\mathrm{post}} = \big(\mu_{a, t-1}^{\mathrm{post}}\big)_{a\in\cA},\quad \bSigma_{t-1}^{\mathrm{post}} = \mathrm{diag}\big((\sigma_{a, t-1}^{\mathrm{post}})^2\big),\label{eq::ts-mab-posterior-mean-var-joint}\\
    \mu_{a, t}^{\mathrm{post}}:=\bigg(\frac1{\sigma_0^2} +\frac{N_{a, t}}{\sigma^2}\bigg)^{-1}\bigg(\frac{\mu_0}{\sigma_0^2} +\frac{N_{a, t}\hat\mu_{a, t}}{\sigma^2}\bigg), \quad (\sigma_{a, t}^{\mathrm{post}})^2 := \bigg(\frac1{\sigma_0^2} +\frac{N_{a, t}}{\sigma^2}\bigg)^{-1},\label{eq::ts-mab-posterior-mean-var}
    \end{gather}
    where $\mu_0\in\RR$ and $\sigma^2, \sigma_0^2>0$ are 
    fixed algorithm parameters representing the prior mean, prior variance, and observation noise variance, respectively.
    \item The mapping $\mathrm{Clip}: \RR^K\rightarrow \RR^K$ adjusts a probability distribution over $K$ discrete actions to ensure that each coordinate is lower bounded by $\pi_{\min}$ (details in Appendix \ref{apdx::clipping}).
    \end{itemize} 
\end{itemize} 

Define the expected reward of arm $a$ by $\mu_a^* = \EE[Y_t \mid A_t = a]$, and let $a^* = \argmax_{a} \mu_a^*$. The following proposition shows that all three algorithms above satisfy the policy convergence condition under a nonzero suboptimality gap. Its proof is in Appendix \ref{apdx::proof-lem::convergence-mab}.

\begin{proposition}\label{lem::convergence-mab}
Suppose $\sup_{a\in\cA}\mathrm{Var}(Y(a))\leq \sigma_Y^2$ for a universal constant $\sigma_Y^2$. Then as long as the suboptimal gap $\Delta = \mu_{a^*}^* - \max_{a' \neq a^*}\mu_{a'}^* > 0$, the policies $\{\pi_t^{\epsilon\text{-greedy}}\}_{t\geq 1}$, $\{\pi_t^{\text{UCB}}\}_{t\geq 1}$ and $\{\pi_t^{\text{TS}}\}_{t\geq 1}$, defined in (\ref{eq::policy-eps-greedy-mab}), (\ref{eq::policy-ucb-mab}) and (\ref{eq::policy-ts-mab}), respectively, satisfy the policy convergence condition in Definition \ref{aspt:policy_convergence}.
\end{proposition}

\emph{Choice of summary statistics.} Proposition \ref{lem::convergence-mab} follows from Theorem \ref{lem::statistics-converge-implies-policy-converge} by identifying the key summary statistics used by each policy. While the most natural statistics vary across policies, natural choices for the three cases are suitable functions of the arm-specific sample mean  $\hat\mu_{a, t}$ and the inverse number of pulls $1/N_{a, t}$. Intuitively, under a minimum sampling probability, each action is chosen infinitely often, and these context-free statistics converge to their population-level values. This ensures that the policies stabilize over time and satisfy the policy convergence condition required for inference. A more rigorous justification for this argument can be found in Lemma \ref{lem::policy-with-no-contexts} in Appendix \ref{apdx::proof-lem::convergence-mab}.

\emph{Beyond minimum sampling probability.} Technically, a minimum sampling probability is not required for policy convergence in general; some algorithms exhibit stability even without this condition. See, for instance, \citep{khamaru2024inference,halder2025stable}, where the authors analyze the stability properties of the UCB algorithm without requiring a minimum sampling probability. Proposition \ref{lem::convergence-mab} focuses on policies with a minimum sampling probability, primarily because it is desired for inference in a misspecified environment, as in the setting of Theorem~\ref{thm::asymptotic-normality-joint-general}. 
%\ziping{I don't understand how this is about whether the reward model is well-specified or not}

%\begin{proposition}
%    Assume that $\Delta > 0$. Any algorithm whose behavior policy converges to the greedy policy w.r.t OLS estimator with clipping level $p_0$, i.e., $\pi_{t}(a) \xrightarrow{p} \one\{a = \argmax_{a'} \hat \mu_{t, a'}\} (1- p_0) + \one\{a \neq \argmax_{a'} \hat \mu_{t, a'}\} p_0 / (|\mathcal{A}| - 1)$ satisfies
%    $$
%        \pi_{t}(a) \xrightarrow{p} \one\{a = \argmax_{a'} \mu_{a'}^*\} (1- p_0) + \one\{a \neq \argmax_{a'} \mu_{a'}^*\} p_0 / (|\mathcal{A}| - 1)
%    $$
%\end{proposition}

\subsection{Policies based on the IPW-Z estimator} 
In certain settings, an ideal policy takes the form $\pi(\cdot \mid \bX_t, \{\btheta_a^*\}_{a \in \cA})$, where $\{\btheta_a^*\}_{a\in\cA}$ are our inferential target parameters defined in (\ref{eq::theta-a-*}). These parameters often inform optimal decision-making within a parametric class, either with respect to maximizing expected reward or achieving other objectives. For instance, in the case of misspecified linear bandits (Example \ref{ex::misspecified-linear-bandits}), the best linear policy that maximizes expected reward can be parameterized by $\{\btheta_a^*\}_{a \in \cA}$, which encode the best linear approximations of the reward functions \citep{chen2021statistical}. 

Motivated by this setup, it is natural to consider behavior policies of the form (\ref{eq::policy-statistics}), where the statistics $\hat\bbeta_{t-1}$ consist of the IPW-Z estimators defined by (\ref{eq::estimating-equation-general}), together with additional algorithm parameters that converge over time:
\begin{equation}\label{eq::IPWZ-policy-plus-parameter}
    \pi(\cdot \mid \bX_t, \{\hat\btheta_a^{(t-1)}\}_{a \in \cA}, \bgamma_{t-1}).
\end{equation}
Here, $\{\bgamma_t\}_{t \geq 1}$ is a sequence of parameters governing the policy. For example, in the case of an $\epsilon$-greedy policy, $\bgamma_t$ can represent the exploration probability at time $t$. These policies often serve as good approximations to an ideal policy within the specified model class. The following proposition shows that, under mild conditions, such behavior policies converge to a limiting policy parameterized by $\btheta^*=(\btheta_a^*)_{a\in\cA}$. The proof is provided in Appendix \ref{apdx::proof-thm::convergence-of-IPWZ-policy}.

\begin{proposition}\label{thm::convergence-of-IPWZ-policy}
Under the assumptions of Theorem \ref{thm::asymptotic-normality-joint-general}, suppose the behavior policy at each time $t$ takes the form (\ref{eq::IPWZ-policy-plus-parameter}), where $\{\hat\btheta_a^{(t)}\}_{a\in\cA, t\geq 1}$ is the IPW-Z estimator defined by (\ref{eq::estimating-equation-general}), and $\{\bgamma_{t}\}_{t\geq 1}$ is a deterministic parameter sequence such that $\bgamma_{t}\in \RR^{d_{\gamma}}$,  $\lim_{t\rightarrow \infty}\bgamma_{t} = \bgamma^*$ for some $\bgamma^*\in\RR^{d_{\gamma}}$. Assume 
$
\pi:\cX\times \RR^{Kd}\times\RR^{d_{\gamma}}\rightarrow \Delta(\cA)
$
is continuous in its second and third argument at $(\btheta^*, \bgamma^*)$ for almost every $\bx\in\cX$ under the distribution of $\bX_t$. Then the behavior policy satisfies the policy convergence condition in Definition \ref{aspt:policy_convergence}, with the limit policy $\bar\pi(a|\bx) = \pi(a|\bx, \btheta^*, \bgamma^*)$.
\end{proposition}

%\ziping{Can we have an example of $\gamma_t$ here? I don't quite understand the role of $\gamma$. Are you thinking of $\gamma_t$ being the temperature parameter of a Boltzmann exploration policy?} 

\begin{remark}
The behavior policy studied in \citep{chen2021statistical}---an $\epsilon$-greedy policy combined with a weighted online LS estimator in the setting of misspecified linear bandits (Example \ref{ex::misspecified-linear-bandits})---is a special case of the policies considered in Proposition \ref{thm::convergence-of-IPWZ-policy}, under their assumption of an almost surely nonzero margin. In this case, $\bgamma_t$ corresponds to the exploration rate at time $t$.
%(\ziping{With $\gamma_t$ being $\epsilon_t$?}), 
By Proposition~\ref{thm::convergence-of-IPWZ-policy}, the policy then satisfies the convergence requirement in Definition~\ref{aspt:policy_convergence}, and the asymptotic normality of the parameter estimates follows from Theorem~\ref{thm::asymptotic-normality-joint-general}. Further details are given in Appendix~\ref{apdx::convergence-epsilon-greedy-weighted-LS}.
\end{remark}

% \paragraph{Boltzmann exploration.} We show convergence of Newton's method and simple stochastic gradient descent, where the behavior policy is a Boltzmann exploration with temperature $\gamma$.

\subsection{Boltzmann exploration}\label{sec::boltzmann-policy-convergence} 
Boltzmann exploration---also known as softmax or Gibbs exploration---is a widespread method in bandit and reinforcement learning literature \citep{kaelbling1996reinforcement, sutton1998reinforcement, sutton1999policy, vermorel2005multi, cesa2017boltzmann}. 
%Given a temperature parameter, a Boltzmann exploration policy chooses actions by sampling from a probability distribution proportional to the exponential of estimated action values. 
Given a temperature parameter $\gamma$, a Boltzmann exploration policy maps the estimated action values into selection probabilities via the softmax transform and samples an action accordingly.
Unlike policies such as UCB and $\epsilon$-greedy, it provides a smooth, differentiable family of policies with a tunable temperature $\gamma$ controlling exploration-exploitation, while remaining simple and numerically stable to implement. 

In this section, we show that, with two common choices of action-value estimators, Boltzmann policy converges when the temperature is sufficiently large. Here, the action value means the conditional mean outcome $\EE[Y(a)|\bX]$. For simplicity, we focus on policies that maintain a \textit{linear working model} for the conditional mean outcomes: At each time $t$, suppose the agent \textit{estimates} that $\EE[Y_t(a)|\bX_t]\approx \bX_t^\top\hat\bbeta_{t, a}$, where $\hat\bbeta_{t, a}$ is an estimator for the linear coefficient vector derived from past history $\cH_{t-1}$. Then, given the estimators $\hat\bbeta_{t} = (\hat\bbeta_{t, a})_{a\in\cA}$ and a temperature parameter $\gamma>0$, Boltzmann policy selects action $a$ with probability 
\begin{align}\label{eq::Boltzmann-policy-unclipped}
\pi^{\gamma}(a \mid \bX_t, \hat\bbeta_{t}) = \frac{\exp\left(\big\langle \hat\bbeta_{t, a}, \bX_t\big\rangle / \gamma\right)}{\sum_{a'\in\cA}\exp\left( \big\langle \hat\bbeta_{t, a'}, \bX_t\big\rangle / \gamma\right)}.
\end{align}
To avoid vanishing probabilities, we further apply a clipping map so that the policy has a minimum sampling probability of $\pi_{\min} \in (0, 1)$ (see Appendix \ref{apdx::clipping} for details). We denote the resulting clipped policy by $\tilde{\pi}^{\gamma}$. 

In what follows, we will show that the policy $\tilde{\pi}^{\gamma}$ converges in the sense of Definition \ref{aspt:policy_convergence} with two common estimator sequences $\{\hat\bbeta_{t}\}_{t\geq 1}$---the ridge estimator and the stochastic gradient descent estimator---provided $\gamma$ is sufficiently large. The key idea is to first show the stochastic convergence of $\{\hat\bbeta_{t}\}_{t\geq 1}$ by writing its evolution as a stochastic approximation process, where a large temperature induces contraction in some underlying mappings that govern these updates. Then, we use Theorem \ref{lem::statistics-converge-implies-policy-converge} to establish policy convergence by viewing $\hat\bbeta_{t}$ as the summary statistics.  Importantly, the result holds in a general environment; the linear working model used by the Boltzmann policy need not be correctly specified.

A broader remark is that these arguments can extend beyond Boltzmann to a general behavior policy with a continuous action selection function: Under mild conditions, such policies converge as long as the policy mapping is not too steep. This offers a simple, practical guideline for designing adaptive experiments in practice.

\subsubsection{The ridge estimator}

The ridge regression estimator at each time $t$ is defined by $\hat \bbeta_{t}^{\Ridge} = (\hat \bbeta_{t, a}^{\Ridge})_{a\in\cA}$, where
\begin{equation}
    \hat \bbeta_{t, a}^{\Ridge} = \left(\lambda \bI + \sum_{\tau=1}^{t-1} 1_{\{A_\tau = a\}} \bX_\tau \bX_\tau^\top\right)^{-1} \left(\sum_{\tau=1}^{t-1} 1_{\{A_\tau = a\}} \bX_\tau Y_\tau\right). \label{eq:ridge-estimator}
\end{equation}

% Further, we let $\tilde{\pi}^{\gamma}(a \mid x, \bm{\theta})$ be the clipped Boltzmann exploration policy defined by 
% \begin{align}
% \tilde{\pi}^{\gamma}(a \mid x, \bm{\theta}) = \max\left(\pi^{\gamma}(a \mid x, \bm{\theta}) - \nu^*, p_{\min}\right), \label{equ:clipping}
% \end{align}
% where $\nu^*$ is chosen such that $\sum_{a\in\cA}\tilde{\pi}^{\gamma}(a \mid x, \bm{\theta}) = 1$.

% \ziping{Notation change in the proof: $f(X_t) \rightarrow \bX_t$, $p_{\min} \rightarrow \pi_{\min}$}

To prove the convergence of $\tilde\pi^\gamma$ under the ridge estimator, we impose the regularity conditions below.
\begin{assumption}
\label{aspt:Ridge_Conv}
There exist positive constants $M$, $\sigma_{\min}$, $R_y$, $\sigma_{\eta}$ such that (i) $\|\bX_t\|_2 \leq M$ a.s.; (ii) $\EE[\bX_t\bX_t^\top] \succeq \sigma_{\min}^2 \bI$; (iii) There exists a function $f(\bx, a)$ such that $Y_t = y(\bX_t, A_t) + \eta_t$ where $\eta_t$ are independent noise with $\EE[\eta_t] = 0$ and $\EE[\eta_t^2] \leq \sigma_{\eta}^2$. Further we assume that $y(\bx, a) \leq R_y < \infty$ for all $\|\bx\|_2 \leq M$ and $a \in \cA$.
% $\|\btheta_a^*\|_2 \leq R_{\theta}$ for all $a \in \cA$. Here, $\btheta_a^*$ is the solution to (\ref{eq::theta-a-*}) with $\bg$ being the score function defined in (\ref{eq::target-parameter-misspecified-linear-bandits}).\YG{Can possibly simplify this?}  

% \ziping{Tentative: $\EE[Y_t^2] \leq M_2^2$. Define $\eta_{t} = Y_t - \EE[Y_t \mid X_t, A_t]$. We immediately have $\EE[\eta_t]$}
%    \begin{itemize}
%        \item (A.1) $\|\bX_t\|_2 \leq M$ a.s. % \ziping{I am just repeating this assumption.}
%        \item (A.2) $\EE[\bX_t\bX_t^\top] \succeq \sigma_{\min}^2 I $
%        \item (A.3) $\|\btheta_a^*\|_2 \leq R_{\theta}$ for all $a \in \cA$ with $\btheta_a^*$ being the the solution to (\ref{eq::theta-a-*}) with $\bg$ being the score function defined in (\ref{eq::target-parameter-misspecified-linear-bandits}).
%    \end{itemize}
\end{assumption}

Under Assumption \ref{aspt:Ridge_Conv}, the next proposition shows that, with sufficiently large temperature $\gamma$, the Boltzmann policy paired with ridge estimator satisfies the policy convergence condition.

\begin{proposition}
    \label{theorem:convergence-of-Ridge}
Suppose Assumption \ref{aspt:Ridge_Conv} holds. Let the data be collected using the behavior policy $\{\tilde{\pi}^{\gamma}(\cdot \mid \bX_t, \hat \bbeta_{t}^{\Ridge})\}_{t\geq 1}$ derived from (\ref{eq::Boltzmann-policy-unclipped}), where $\hat \bbeta_{t}^{\Ridge}$ is the ridge estimator defined in (\ref{eq:ridge-estimator}). If the temperature parameter 
\begin{align}
    \gamma \geq 2M^3 \cdot \max\left\{\frac{2}{\pi_{\min}\sigma_{\min}^2}, \frac{8(M R_{y} + 4M \sigma_{\eta} \sqrt{dK})}{\pi_{\min}^2\sigma_{\min}^4} \right\} K(K+1),
\end{align}
% \YG{In the above expression, is $A$ actually $K$?}
then there exists a deterministic vector $\bar\bbeta\in\RR^{Kd}$ such that $\hat \bbeta_{t}^{\Ridge} \xrightarrow{p} \bar{\bbeta}$. Consequently, the policy $\{\tilde{\pi}^{\gamma}(\cdot \mid \bX_t, \hat \bbeta_{t}^{\Ridge})\}_{t\geq 1}$ satisfies the policy convergence condition in Definition \ref{aspt:policy_convergence}.
\end{proposition}

The proof Proposition \ref{theorem:convergence-of-Ridge} is in Appendix \ref{apdx::proof-thm::convergence-of-Ridge}. The key technical challenge is to identify a quantity tied to $\hat \bbeta_{t}^{\Ridge}$ whose evolution becomes a contraction when $\gamma$ is large. We then apply tools from stochastic approximation theory to show the convergence of $\hat \bbeta_{t}^{\Ridge}$ to a fixed point of the underlying dynamics. 

\emph{Convergence considerations for ridge under alternative policies.} Many online contextual bandit algorithms, such as LinUCB \citep{chu2011contextual} and Thompson Sampling \citep{russo2018tutorial}, rely on the ridge estimator as summary statistics for their behavior policy. Proposition \ref{theorem:convergence-of-Ridge} suggests that the ridge estimator converges with data collected under smooth, non-steep policies, such as Boltzmann exploration with large temperature. This conclusion need not hold for other policies. For example, if the action-selection mapping is too steep or noncontinuous, the induced dynamics may not be a contraction, and the sequence $\{\hat \bbeta_{t}^{\Ridge}\}_{t\geq 1}$ may not converge (see the LinUCB example in Section \ref{sec::example-policy-nonconvergence}). Appendix \ref{apdx::ridge-nonconvergence-with-no-fixed-point} provides a rigorous argument for one such mechanism: nonconvergence occurs if the policy mapping induces no fixed point in the dynamics of $\{\hat \bbeta_{t}^{\Ridge}\}_{t\geq 1}$, i.e., there exists an action $\bar a\in\cA$ so that no $\bbeta$ satisfies
\begin{align}
    \bm{\beta}_{\bar a}=\bSigma_{\bar a}^{-1}(\bm{\beta}) \bm{\varphi}_{\bar a}(\bm{\beta}).
\end{align}
Here $\bSigma_{a}(\bm{\beta}) = \EE_{\bX_t, A_t \sim \pi(\cdot \mid \bX_t, \bm{\beta})}[ 1_{\{A_t = a\}} \bX_t \bX_t^\top]$,  $\bm{\varphi}_a(\bm{\beta}) = \EE_{\bX_t, A_t\sim \pi(\cdot \mid \bX_t, \bm{\beta}), Y_t}[ 1_{\{A_t = a\}} \bX_t Y_t ]$. Guided by these considerations, in Section \ref{sec::simulation}, we present a number of challenging environments where common behavior policies that use $\{\hat \bbeta_{t}^{\Ridge}\}_{t\geq 1}$ as summary statistics fail to converge.

\subsubsection{Stochastic gradient descent}\label{sec::SGD}

In addition to the ridge estimator, we show that the convergence also holds for a family of stochastic gradient descent (SGD) estimators under the Boltzmann exploration policy. Specifically, let $\hat \bbeta_{t}^{\SGD} = (\hat \bbeta_{t, a}^{\SGD})_{a\in\cA}$ be the SGD estimator, where for each $a\in\cA$, $\hat \bbeta_{t, a}^{\SGD}$ admits the following update rule:
\begin{align}\label{eq::SGD-update}
    \hat \bbeta_{t, a}^{\SGD} = \hat \bbeta_{t-1, a}^{\SGD} + \eta_t 1_{\{A_t = a\}}\bh\left(\bX_t, Y_t; \hat \bbeta_{t-1, a}^{\SGD}\right),
\end{align}
where $\eta_t$ is the learning rate at time $t$ and $\bh: \RR^d \times \RR \to \RR^d$ is a function parameterized by $\hat \bbeta_{t, a}^{\SGD}$. The role of $\bh$ is to specify the stochastic update direction, and in practice it can take the form of the gradient of a loss function (e.g., squared loss in linear models, logistic loss in classification, or a general neural-network loss function). This SGD estimator is widely used in practice especially when complex models such as neural networks are involved in decision-making \citep{zhou2020neural}. %When $\bh$ is the score function in the three examples we introduced in Section \ref{sec::problem-setup}, the inference target $\btheta_a^*$ satisfies $\EE[\bh(\bx, y; \btheta_a^*)] = 0$.

% We provide three example $\bh$ functions corresponding to the three examples discussed in Section \ref{sec::problem-setup}:
% \begin{itemize}
%         \item $\bh(\bx, y; \bbeta_a) = \bx y - \bx \bx^\top \bbeta_a$ for the misspecified linear bandits in Example \ref{ex::misspecified-linear-bandits}. 
%         \item $\bh(\bx, y; \bbeta_a) = \bx y - (\bx\bx^\top - \bSigma_e)\bbeta_a$ for the noisy context model in Example \ref{ex::bandits-noisy-contexts}.
%         \item $\bh(\bx, y; \bbeta_a) = \pi^{e}(a|\bx)y - \bbeta_a$ for the OPE estimator in Example \ref{ex::ope}.
% \end{itemize}
% By definition, the above update functions satisfy that $\EE[\bh(\bx, y; \btheta_a^*)] = 0$ with $\btheta_a^*$ being the optimal inference target parameters defined in (\ref{eq::theta-a-*}) for each example. Therefore, the SGD will converge to the inference targets, which are likely to be the ideal choice for policy parameters.

%We further write $\hat \bbeta_{t}^{\SGD} = (\hat \bbeta_{t, 1}^{\SGD}, \ldots, \hat \bbeta_{t, K}^{\SGD})$. 

In this section, we study the stochastic convergence of the SGD estimator via stochastic approximation theory. Proposition \ref{theorem:convergence-of-SGD} shows that, under mild regularity conditions on $\bh$, the policy converges provided the sequence $\{\hat\bbeta_{t}^{\SGD}\}_{t\geq 1}$ remains bounded with probability one. The proof of Proposition \ref{theorem:convergence-of-SGD} is given in Appendix \ref{apdx::proof-thm::convergence-of-SGD}. Proposition \ref{prop::SGD-boundedness} further demonstrates that the requirements of Proposition \ref{theorem:convergence-of-SGD} are satisfied in a broad class of settings.

\begin{proposition}
    \label{theorem:convergence-of-SGD}
    Suppose the data are collected by the clipped Boltzmann exploration policy $\{\tilde\pi^{\gamma}(a \mid \bX_t, \hat \bbeta_{t}^{\SGD})\}_{t\geq 1}$, where $\{\hat \bbeta_{t}^{\SGD}\}_{t\geq 1}$ are the SGD estimators defined recursively from (\ref{eq::SGD-update}). Assume that
    \begin{itemize}
        \item $\sum_{t=1}^{\infty} \eta_t = \infty$ and $\sum_{t=1}^{\infty} \eta_t^2 < \infty$.
        \item $\hat \bbeta_{t, a}^{\SGD} \leq R_{\bbeta} < \infty$ is a.s. for all $a \in \cA$ and $t \geq 1$.
        \item $\bh$ satisfies that $\EE[\|\bh(\bX_t, Y_t(a); \bbeta)\|_2^2] \leq \rho_{\bh} (1+\|\bbeta\|_2^2) < \infty$ for all $\|\bbeta\|_2 \leq R_{\bbeta}$ and $a \in \cA$.
    \end{itemize}
    Then as long as %$\gamma$ is sufficiently large such that 
    \begin{align}
        \gamma > 2\sqrt{K}(K - 1)M \sqrt{\rho_{\bh}} R_{\bbeta},
    \end{align}
    we have that 
    \begin{align}
        \hat \bbeta_{t, a}^{\SGD} \xrightarrow{p} \bar{\bbeta}_a \quad \forall a \in \cA,
    \end{align}
    for some $\{\bar{\bbeta}_a\}_{a\in\cA}$. Consequently, the behavior policy satisfies the policy convergence condition in Definition \ref{aspt:policy_convergence}.
\end{proposition}

% The proof of Proposition \ref{theorem:convergence-of-SGD} is a direct application of Theorem 3.1 in \citep{borkar2008stochastic}, and the details are given by Appendix \ref{apdx::proof-thm::convergence-of-SGD}.

    % $$
    % for some $\{\bar{\bbeta}_a\}_{a\in\cA}$. Consequently, the behavior policy satisfies the policy convergence condition in Definition \ref{aspt:policy_convergence}.
% \end{proposition}

The above conditions hold in many practical settings. A common case arises when the policy statistics are intended to approximate target parameters $\{\btheta_a^*\}_{a\in\cA}$ satisfying (\ref{eq::theta-a-*}) for some score function $\bg$. In this scenario, it is natural to apply the update rule (\ref{eq::SGD-update}) with $\bh(\bx, y;\bbeta_a)=\bg(\bx, y;\bbeta_a)$. Proposition \ref{prop::SGD-boundedness} verifies that the above conditions are met for the score functions $\bg$ in Examples \ref{ex::misspecified-linear-bandits}, \ref{ex::bandits-noisy-contexts}, and \ref{ex::ope}. The central step is to establish almost sure boundedness of the SGD estimator via the standard Lyapunov argument for stability in stochastic approximation. The proof of this proposition, along with a generalization of the idea to a broader class of functions $\bh$, is provided in Appendix \ref{appx::prop::SGD-boundedness}.
% One may easily verify that the third condition on the coefficient $\rho_{\bh}$ is satisfied under regularity conditions in Assumption \ref{aspt:Ridge_Conv}.

\begin{proposition}\label{prop::SGD-boundedness}
    Suppose the stepsize satisfies $\sum_{t=1}^{\infty} \eta_t = \infty$ and $\sum_{t=1}^{\infty} \eta_t^2 < \infty$ and Assumption \ref{aspt:Ridge_Conv} holds. Then $\hat \bbeta_{t, a}^{\SGD}$ is bounded a.s. for the update function $\bh$ being the score functions $\bg$ defined in three examples \ref{ex::misspecified-linear-bandits}, \ref{ex::bandits-noisy-contexts}, and \ref{ex::ope} under the clipped Boltzmann exploration policy with any $\gamma > 0$. Further, the condition in Proposition \ref{theorem:convergence-of-SGD} on the coefficient $\rho_{\bh}$ is satisfied.
\end{proposition}

% \YG{Is Theorem 4.2 proving convergence of a non clipped boltzmann policy, different from the ridge case?}

%The proof of Proposition \ref{theorem:convergence-of-SGD} and \ref{prop::SGD-boundedness} is given by Appendix \ref{apdx::proof-thm::convergence-of-SGD} and \ref{appx::prop::SGD-boundedness}, respectively.
% The following proposition shows that Assumption \ref{aspt:globally-asymptotically-stable-equilibrium} holds for the two estimators corresponding to the two inference targets in Example \ref{ex::misspecified-linear-bandits} and \ref{ex::bandits-noisy-contexts}.

\section{Examples}\label{seq::examples}

In this section, we revisit the three examples from Section \ref{sec::problem-setup} and establish statistical inference guarantees by applying Theorem \ref{thm::asymptotic-normality-general} and \ref{thm::asymptotic-normality-joint-general} or suitably adapting them to relevant practical settings. 

\subsection{Misspecified Linear Bandits}\label{sec::ex1}

We first consider Example \ref{ex::misspecified-linear-bandits}, where the target parameter $\btheta_a^*$ is defined by (\ref{eq::theta-a-*}) with the score function $\bg$ in (\ref{eq::target-parameter-misspecified-linear-bandits}). A natural estimator for $\btheta_a^*$ that satisfies (\ref{eq::estimating-equation-general}) is 
\begin{equation}
\hat\btheta_a^{(T)} = \left(\frac1T\sum_{t=1}^T\frac{1_{\{A_t = a\}}}{\pi_t(A_t)}\bX_t\bX_t^\top\right)^{-1}\left(\frac1T\sum_{t=1}^T\frac{1_{\{A_t = a\}}}{\pi_t(A_t)}\bX_tY_t\right).
\end{equation}

In this setting, Assumptions \ref{aspt:identifiability-general}-\ref{aspt:smoothness-general}, which form part of the conditions for Theorem \ref{thm::asymptotic-normality-joint-general}, are implied by the following regularity assumption.
\begin{assumption}\label{aspt:ex1}
There exist constants $R_\theta, M>0$ such that $\sup_{a\in\cA}\|\btheta_a^*\|_2< R_\theta$, $\|\bX_t\|_2\leq M$. Moreover, $\bSigma_X:= \EE[\bX_t\bX_t^\top]$ is invertible, and $\sup_{a\in\cA}\EE [Y_t(a)^4]<\infty$.
\end{assumption}

With Assumption \ref{aspt:ex1} in place, together with the remaining conditions of Theorem \ref{thm::asymptotic-normality-joint-general}, the joint asymptotic normality of $\hat\btheta^{(T)} = \big((\hat\btheta_1^{(T)})^\top, \ldots, (\hat\btheta_K^{(T)})^\top\big)^\top$ and a consistent estimator for its asymptotic variance follow directly from Theorem \ref{thm::asymptotic-normality-joint-general} and Proposition \ref{thm::consistent-var-estimator-general}, as stated below.

\begin{corollary}\label{cor::ex1}
Suppose Assumptions \ref{aspt:unconfoundedness} and \ref{aspt:ex1} holds. Consider any behavior policy that converges to a limiting policy $\bar\pi$ as in Definition \ref{aspt:policy_convergence} and satisfies Assumption \ref{aspt:min-sampling-prob} for every action $a\in\cA$. Then, as $T\rightarrow\infty$, (\ref{eq::asymptotic-normality-joint-general}) holds with the joint asymptotic variance $\bSigma^* = \diag(\bSigma_1^*, \ldots, \bSigma_K^*)$, where for all $a\in\cA$,
\begin{equation}
\bSigma_a^* = \bSigma_X^{-1}\cdot \EE\left[\frac{(Y_t(a)-\bX_t^\top\btheta_a^*)^2}{\bar \pi(a|\bX_t)}\bX_t\bX_t^\top\right]\cdot \bSigma_X^{-1}.
\end{equation}
A consistent estimator for $\bSigma_a^*$ is given by (\ref{eq::estimator-asympt-var}), with 
\begin{align}\label{eq::estimator-asympt-var-ex1}
\hat{\dot{\bG}}_{a, T} = \frac1T\sum_{t=1}^T\frac{1_{\{A_t = a\}}}{\pi_t(A_t)}\bX_t\bX_t^\top,\quad 
\hat\bI_{a, T} = \frac1T\sum_{t=1}^T\frac{1_{\{A_t = a\}}}{\pi_t(A_t)^2}(Y_t-\bX_t^\top\hat\btheta_a^{(T)})^2\bX_t\bX_t^\top.
\end{align}
\end{corollary}

\begin{remark}\label{rmk::ex1-simple-var-est}
In Corollary \ref{cor::ex1}, a simpler consistent estimator for the asymptotic variance is obtained by replacing $\hat{\dot{\bG}}_{a, T}$ in (\ref{eq::estimator-asympt-var}) with $\hat\bSigma_X := \frac1T\sum_{t=1}^T\bX_t\bX_t^\top$. This alternative remains consistent since both matrices consistently estimate $\bSigma_X$. 
\end{remark}

In the special case where the behavior policy is an 
$\epsilon$-greedy policy combined with the weighted online LS estimator, Corollary \ref{cor::ex1} yields the same asymptotic normality result as \citep{chen2021statistical}; see also the discussion following Theorem \ref{thm::asymptotic-normality-general}. The variance estimator used in \citep{chen2021statistical} corresponds to the simplified plug-in form described in Remark \ref{rmk::ex1-simple-var-est}. In addition, Corollary \ref{cor::ex1} holds under generally weaker technical assumptions, such as not requiring the context $\bX_t$ to be a continuous random vector.

\subsection{Bandits with Noisy Contexts}\label{sec::ex2}

We now revisit Example \ref{ex::bandits-noisy-contexts}. Recall that the potential outcome satisfies $Y_t(a) = \bS_t^\top \btheta_a^* + \eta_t$ with unobserved state $\bS_t$, while the observed context is the noisy proxy $\bX_t = \bS_t + \bepsilon_t$. The noise $\bepsilon_t$ satisfies $\EE[\bepsilon_t|\bS_t] = \mathbf{0}$, and $\mathrm{Var}(\bepsilon_t|\bS_t) = \bSigma_e$ for a constant matrix $\bSigma_e\in\RR^{d\times d}$. No parametric assumptions are imposed on the distribution of $\bepsilon_t$. In addition, assume $\mathbb{E}[\eta_t\mid\bS_t,\bX_t]=0$ and define $\sigma_\eta^2=\mathrm{Var}(\eta_t\mid\bS_t,\bX_t)$.

We first consider the case where the contextual error variance $\bSigma_e$ is known. Then the score function reduces to (\ref{eq::target-parameter-bandits-noisy-contexts}). An estimator of $\btheta_a^*$ satisfying (\ref{eq::estimating-equation-general}) is
\begin{equation}
\hat\btheta_a^{(T)} = \left[\frac1T\sum_{t=1}^T\frac{1_{\{A_t = a\}}}{\pi_t(A_t)}(\bX_t\bX_t^\top-\bSigma_e)\right]^{-1}\left(\frac1T\sum_{t=1}^T\frac{1_{\{A_t = a\}}}{\pi_t(A_t)}\bX_tY_t\right).
\end{equation}

To infer the target parameter $\btheta_a^*$ via Theorem \ref{thm::asymptotic-normality-joint-general}, we impose the following regularity condition, which guarantees Assumptions \ref{aspt:identifiability-general}–\ref{aspt:smoothness-general} required for the theorem. 

\begin{assumption}\label{aspt:ex2}
There exist constants $R_\theta, M, M_\eta > 0$ such that $\sup_a\|\btheta_a^*\|_2< R_\theta$, \\$\max\{\|\bX_t\|_2, \|\bS_t\|_2\}\leq M$, $\EE[\eta_t^4|\bS_t, \bX_t]\leq M_\eta$. In addition, $\bSigma_S: = \EE[\bS_t\bS_t^\top]$ is invertible.
\end{assumption}

The following corollary, obtained from Theorem \ref{thm::asymptotic-normality-joint-general} and Proposition \ref{thm::consistent-var-estimator-general}, provides inference for $\{\btheta_a^*\}_{a\in\cA}$ in this setting. The proof establishing the variance expression in the corollary is provided in Appendix \ref{apdx::proof-cor::ex2}.
\begin{corollary}\label{cor::ex2}
Suppose Assumption \ref{aspt:unconfoundedness} and \ref{aspt:ex2} holds. Consider any behavior policy that converges to a limiting policy $\bar\pi$ as in Definition \ref{aspt:policy_convergence} and satisfies Assumption \ref{aspt:min-sampling-prob} for every $a\in\cA$. Then, as $T\rightarrow\infty$, (\ref{eq::asymptotic-normality-joint-general}) holds with $\bSigma^* = \diag(\bSigma_1^*, \ldots, \bSigma_K^*)$, where for all $a\in\cA$,
\begin{equation}\label{eq::var-ex2}
\bSigma_a^* = \bSigma_S^{-1}\bar \bI_a \bSigma_S^{-1},\quad \bar \bI_a:= \EE\frac{1}{\bar \pi(a|\bX_t)} \big[\bh_a(\bX_t, \bS_t)\bh_a(\bX_t, \bS_t)^\top + \sigma_\eta^2\bX_t\bX_t^\top\big],
\end{equation}
%$\bar \bI_a= \EE\frac{1}{\bar \pi(a|\bX_t)} \big[\bh_a(\bX_t, \bS_t)\bh_a(\bX_t, \bS_t)^\top + \sigma_\eta^2\bX_t\bX_t^\top\big]$, 
and $\bh_a(\bx, \bs):= (\bx \bx^\top - \bSigma_e )\btheta_a^* - \bx\bs^\top\btheta_a^*$. A consistent estimator for $\bSigma_a^*$ is in (\ref{eq::estimator-asympt-var}), with 
\begin{align}
\hat{\dot{\bG}}_{a, T} = \frac1T\sum_{t=1}^T\frac{1_{\{A_t = a\}}}{\pi_t(A_t)}(\bX_t\bX_t^\top\!-\!\bSigma_e),
\hat\bI_{a, T} = \frac1T\sum_{t=1}^T\frac{1_{\{A_t = a\}}}{\pi_t(A_t)^2}\!\left[(\bX_t\bX_t^\top \!-\!\bSigma_e)\hat\btheta_a^{(T)}\!-\!\bX_t Y_t\right]^{\otimes 2}.\label{eq::estimator-asympt-var-ex2}
\end{align}
\end{corollary}

In practice, the variance of the contextual error $\bSigma_e$ is often unknown and must be estimated from auxiliary data. In this case, $\hat\btheta_a^{(T)}$ is obtained from (\ref{eq::estimating-equation-general}) by replacing $\bSigma_e$ with an estimator $\hat\bSigma_e$. The derivation and corresponding inference results are provided in Appendix \ref{apdx::ex2-estimated-Sigma-e}.

\subsection{Off-policy Evaluation (OPE) with Adaptively Collected Data} 

We now consider Example \ref{ex::ope}, where the target parameter $V^* = \sum_{a\in\cA}\btheta_a^*$, and $\btheta_a^* = \EE [\pi^e(a|\bX_t) Y_t]$ solves (\ref{eq::theta-a-*}) with an arm-specific score function (\ref{eq::target-parameter-ope}). We estimate $V^*$ by 
$\hat V^{(T)} = \sum_{a\in\cA}\hat\btheta_a^{(T)}$,
where 
\begin{equation}
\hat\btheta_a^{(T)} = \left(\frac1T\sum_{t=1}^T\frac{1_{\{A_t = a\}}}{\pi_t(A_t)}\right)^{-1}\left(\frac1T\sum_{t=1}^T\frac{1_{\{A_t = a\}}}{\pi_t(A_t)}\pi^e(a|\bX_t) Y_t\right)
\end{equation}
satisfies (\ref{eq::estimating-equation-general}). 

The asymptotic distribution of $\hat V^{(T)}$ can be derived from the joint distribution of $\{\hat\btheta_a^{(T)}\}_{a\in\cA}$ via Theorem \ref{thm::asymptotic-normality-joint-general}, enabling statistical inference. Specifically, we impose the following condition, which encompasses Assumptions \ref{aspt:identifiability-general}–\ref{aspt:smoothness-general} required for the theorem. 

\begin{assumption}\label{aspt:ex3}
$\EE[Y_t^4]<\infty$, and there exist a constant $R_\theta > 0$ such that $\sup_a\|\btheta_a^*\|_2< R_\theta$. 
\end{assumption}

The corollary below informs valid inference of the target parameter $V^*$. It follows from Theorem \ref{thm::asymptotic-normality-joint-general} and Proposition \ref{thm::consistent-var-estimator-general}.

\begin{corollary}
Suppose Assumption \ref{aspt:unconfoundedness} and \ref{aspt:ex3} holds. Consider any behavior policy that converges to a limiting policy $\bar\pi$ as in Definition \ref{aspt:policy_convergence} and satisfies Assumption \ref{aspt:min-sampling-prob} for every $a\in\cA$. Then, as $T\rightarrow\infty$,
\begin{equation}
\sqrt{T}\big(\hat V^{T} - V^*\big)\xrightarrow{d}
\cN\left(0, \sum_{a\in\cA}\EE\frac{\big(\pi^e(a|\bX_t)Y_t - \btheta_a^*\big)^2}{\bar \pi(a|\bX_t)}\right).
\end{equation}
A consistent variance estimator is 
\begin{equation}\label{eq::ex3-var-estimator}
\sum_{a\in\cA} \left(\hat{\dot{\bG}}_{a, T}\right)^{-2}\hat \bI_{a, T},
\end{equation}
where for all $a\in\cA$,
\begin{align}
\hat{\dot{\bG}}_{a, T} = \frac1T\sum_{t=1}^T\frac{1_{\{A_t = a\}}}{\pi_t(A_t)},\quad 
\hat\bI_{a, T} = \frac1T\sum_{t=1}^T\frac{1_{\{A_t = a\}}}{\pi_t(A_t)^2}\big(\pi^e(a|\bX_t)Y_t - \hat\btheta_a^{(T)}\big)^2.
\end{align}
\end{corollary}
Similar to Example \ref{ex::misspecified-linear-bandits}, the variance estimator can be simplified by replacing $\hat{\dot{\bG}}_{a, T}$ in (\ref{eq::ex3-var-estimator}) with 1, since $\hat{\dot{\bG}}_{a, T}\xrightarrow{p} 1$ in this setting. See Lemma \ref{lem::Convergence-of-weighted-avg-action-selection} for details. 

Finally, we note that \citep{zhan2021off, bibaut2021post} proposed alternative estimators for the same target parameter by adding normalization weights to the classical doubly robust estimator \citep{jiang2016doubly}, where a regression model for the reward is adaptively fitted based on the observed history to reduce variance. Valid inference is possible if this regression function converges, and variance can be further reduced when the regression function is well aligned with the true reward. The regression model may be chosen in a flexible manner; However, ensuring convergence and alignment is substantially more difficult in adaptive data collection regimes than in the standard i.i.d. setting---a subtle challenge that has received relatively little attention in the literature. By contrast, our approach decomposes the target across actions and applies self-normalization weights to a simpler importance-weighted estimator for each action. In Section~\ref{sec::simulations-ope}, we demonstrate across five simulation environments that our method, though simple, is consistently competitive with the approaches in \citep{zhan2021off, bibaut2021post}---whether using a zero function, a linear model, or a flexible tree model as the regression model---and in several cases provides clear gains in both stability and efficiency. These findings highlight both the promise and the difficulty of variance reduction in adaptive settings, and underscore the importance of further methodological work.

\section{Simulation Studies}
\label{sec::simulation}

% \ziping{Update to add Ruohan}

In this section, we conduct simulation studies to evaluate the proposed inference methods under three previously introduced inference targets: Target \ref{ex::misspecified-linear-bandits} (misspecified linear bandits with score function (\ref{eq::target-parameter-misspecified-linear-bandits})), Target \ref{ex::bandits-noisy-contexts} (linear bandits with noisy contexts with score function (\ref{eq::target-parameter-bandits-noisy-contexts})), and Target \ref{ex::ope} (offline policy valuation with score function (\ref{eq::target-parameter-ope})). 

For Target \ref{ex::bandits-noisy-contexts}, we design three simulation environments: \texttt{NC-hard1}, \texttt{NC-hard2}, \texttt{NC-Gaussian}. In these environments, the reward function w.r.t. the underlying true context is linear, and the model misspecification arises from noisy context. Specifically, \texttt{NC-hard1}, \texttt{NC-hard2} are deliberately designed challenging environments so that the ridge estimator either oscillates or tends to multiple limits under common behavior policies such as LinUCB. In \texttt{NC-Gaussian}, all variables and parameters are jointly Gaussian. For Targets \ref{ex::misspecified-linear-bandits} and \ref{ex::ope}, in addition to the three environments above, we also include two environments \texttt{MS-Polynomial} and \texttt{MS-Neural}, where the potential outcomes depend directly on the contexts through complex non-linear relationships: a polynomial reward function in \texttt{MS-Polynomial} and a neural network reward function in \texttt{MS-Neural}. 
% \YG{It seems for Figure \ref{fig::results-bandits-noisy-contexts}, target 2, we only use NC-hard1, NC-hard2 and an additional 'random' environment? We need to explain what 'random' is} More details are provided in Appendix \ref{apdx::simulation-details::environment-settings}.

We evaluate the inference methods using datasets generated by four distinct algorithms: pure random selection (\texttt{Random}), Boltzmann exploration with the ridge estimator (\texttt{Boltzmann}), a greedy algorithm with the IPW-Z estimator in (\ref{eq::estimating-equation-general}) (\texttt{IPW-Z}), and Boltzmann exploration with the stochastic gradient descent estimator (\texttt{SGD}). Throughout, we use linear working models in the behavior policies, despite the true mean reward function $y(x,a) \coloneqq \EE[Y_t \mid A_t = a, \bX_t = \bx]$ being non-linear, allowing us to assess the robustness of the inference methods under misspecifications. These algorithms are shown to have policy convergence in Section \ref{sec::policy-convergence}. All the algorithm are clipped with a minimal sampling probability $\pi_{\min}$ of 0.05 to ensure bounded inverse probability weights. 
% All algorithms clip the sampling probabilities at a minimum value $p_0$ to ensure bounded inverse probability weights.

\subsection{Results for Targets \ref{ex::misspecified-linear-bandits} and \ref{ex::bandits-noisy-contexts}}

We first examine the empirical coverage of nominal confidence intervals (95\%, 90\%, 80\%, 70\%, 60\%, and 50\%) for inference target \ref{ex::misspecified-linear-bandits} across the five environments. Results are based on 2500 Monte Carlo simulations. Figure \ref{fig::results-misspecified-linear-bandits} shows that the proposed inference method consistently achieves the desired coverage across all algorithms and environments, with only a slight undercoverage (around 90\% for nominal 95\%) observed in the non-linear environments (\texttt{MS-Polynomial}, \texttt{MS-Neural}) under some behavior policies.
%when using IPW-Z and Boltzmann exploration with the ridge estimator. 
Similar trends appear for Target \ref{ex::bandits-noisy-contexts} in the three noisy context environments (Figure \ref{fig::results-bandits-noisy-contexts}, panels a and b).

We further examine how the temperature parameter $\gamma$ and the minimum sampling probability $\pi_{\min}$ affect inference performance, using Target \ref{ex::bandits-noisy-contexts} as an example. Figure \ref{fig::results-bandits-noisy-contexts}(c) shows that increasing $\gamma$ reduces the variance of the IPW-Z estimator. This is consistent with Section \ref{sec::boltzmann-policy-convergence}, which establishes that larger $\gamma$ stabilizes the behavior policy, and also prevents action selection probabilities from becoming too small---thereby reducing variance. Similarly, Figure \ref{fig::results-bandits-noisy-contexts}(d) shows that higher minimum sampling probabilities $\pi_{\min}$ substantially improve empirical coverage of the 95\% confidence interval. This provides empirical evidence that avoiding very small sampling probabilities %mitigates large inverse probability weights, making 
makes the inference procedure more robust in complex environments.

%We further analyze how the temperature parameter $\gamma$ and the minimum sampling probability $p_0$ affect inference performance, with Target \ref{ex::bandits-noisy-contexts} as an example. Figure \ref{fig::results-bandits-noisy-contexts}(c) shows that increasing temperature $\gamma$ reduces the variance of the IPW-Z estimator. Figure \ref{fig::results-bandits-noisy-contexts}(d) demonstrates that higher minimum sampling probabilities $p_0$ significantly improve empirical coverage of the 95\% confidence interval. Conversely, low values of $\gamma$ or $p_0$ cause excessively large inverse probability weights, hindering convergence of the asymptotic variance estimator.

\begin{figure}[tb]
    \centering
    \includegraphics[width=1\textwidth]{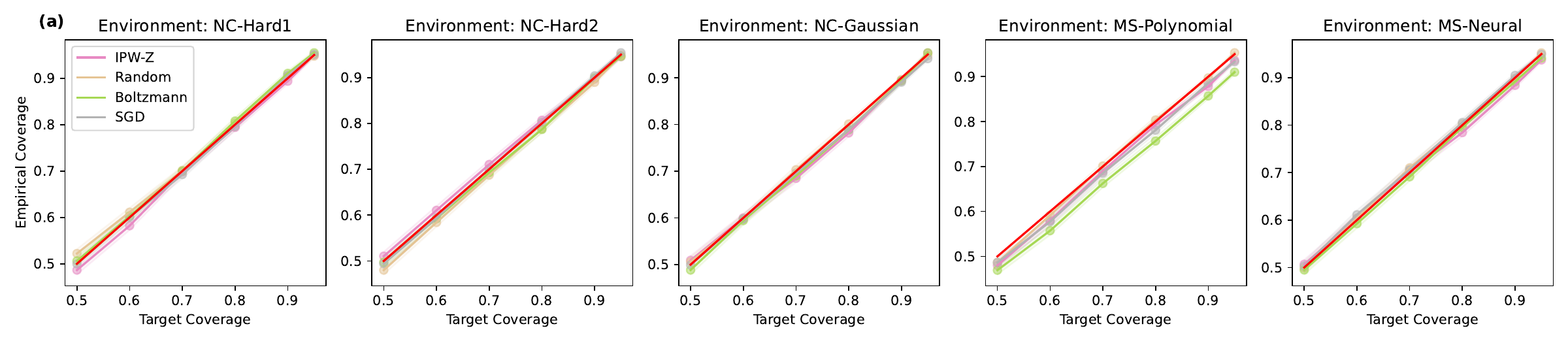}
    \includegraphics[width=1\textwidth]{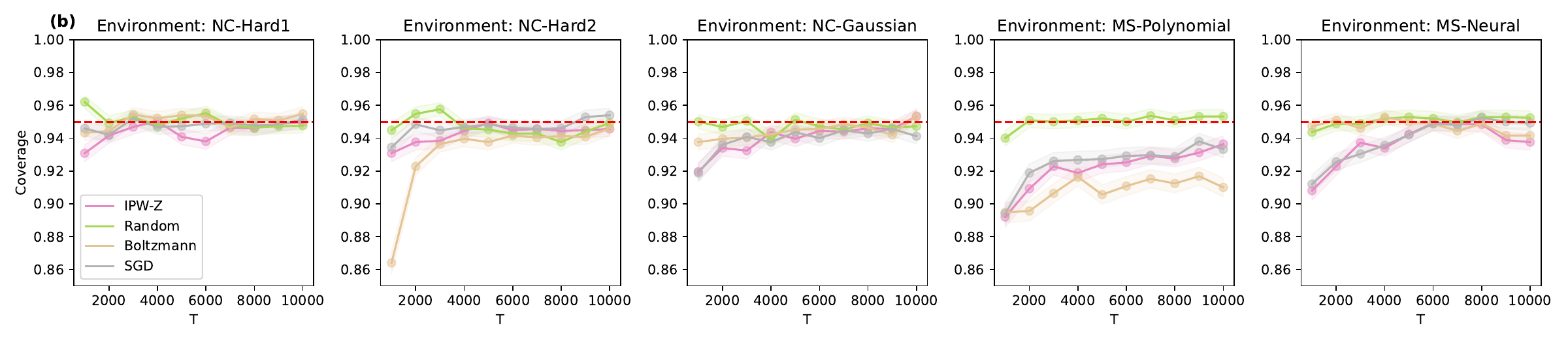}
    \caption{\textbf{(a)} Empirical coverages of 95\%, 90\%, 80\%, 70\%, 60\%, and 50\% confidence intervals vs. the target coverage for Target \ref{ex::misspecified-linear-bandits}. \textbf{(b)} Empirical coverages of 95\% confidence interval over 10,000 steps under three noisy context environments. Results averaged across 2,500 Monte Carlo simulations.}
    \label{fig::results-misspecified-linear-bandits}
\end{figure}

\begin{figure}[tb]
    \centering
    \subfigure{
        \includegraphics[width=0.65\textwidth]{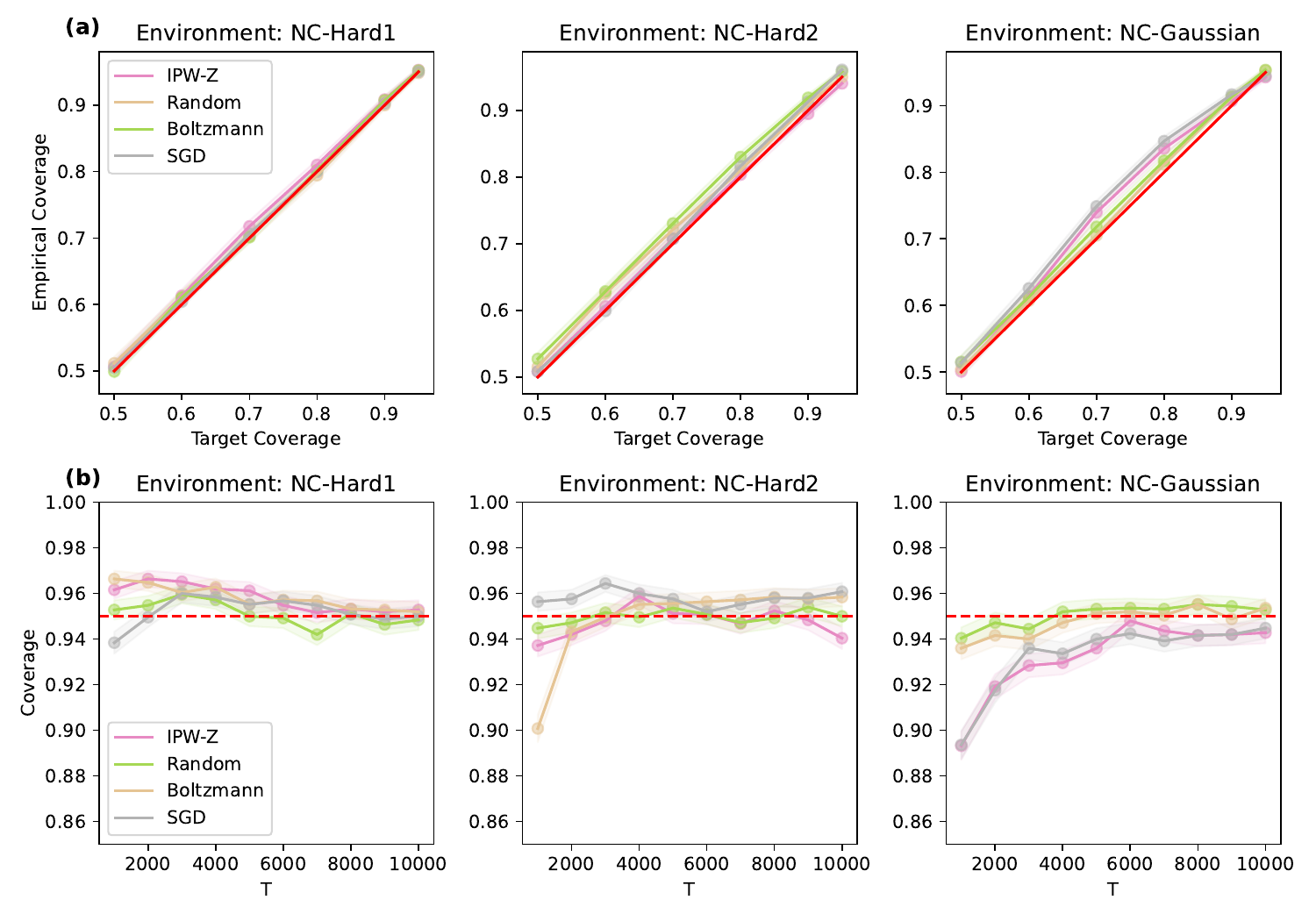}
    }
    \subfigure{
        \includegraphics[width=0.225\textwidth]{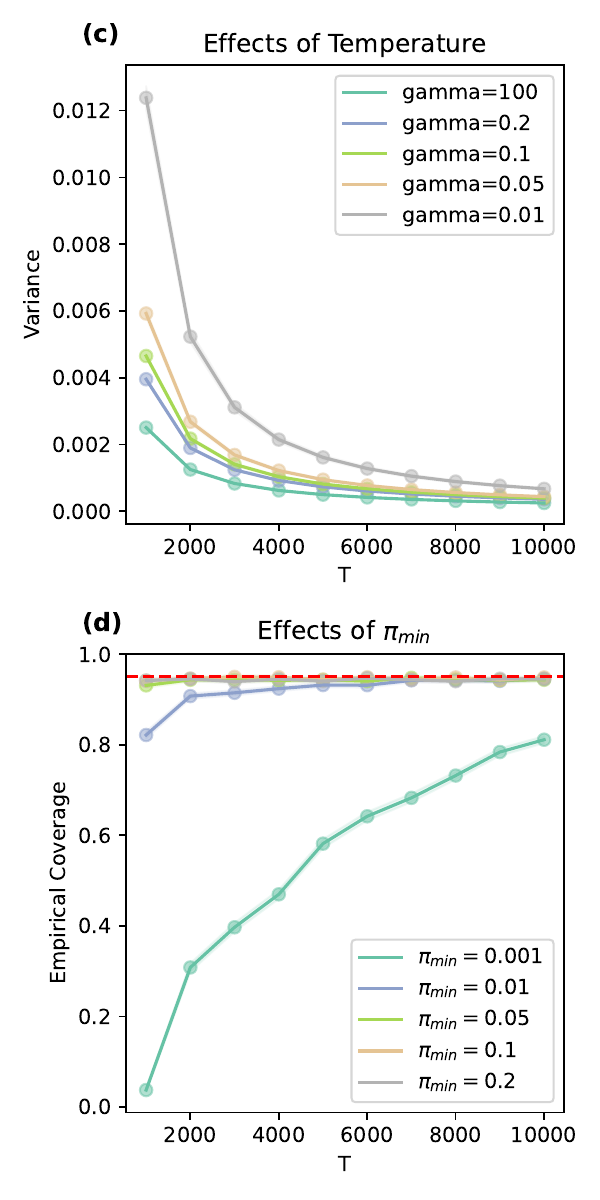}
    }
    \caption{\textbf{(a)} Empirical coverages of 95\%, 90\%, 80\%, 70\%, 60\%, and 50\% confidence intervals vs. the target coverage for Target \ref{ex::bandits-noisy-contexts}. \textbf{(b)} Empirical coverages of 95\% confidence interval over 10,000 steps under three noisy context environments. \textbf{(c)} Variance of the proposed IPW-Z estimator over 10,000 steps under Boltzmann exploration with different temperature $\gamma$. \textbf{(d)} Empirical coverages of 95\% confidence interval over 10,000 steps under greedy algorithm w.r.t. the IPW-Z estimator with different minimum sampling probabilities $\pi_{\min}$.}
    \label{fig::results-bandits-noisy-contexts}
\end{figure}

\subsection{Results for Target \ref{ex::ope}}\label{sec::simulations-ope}

In the OPE setting, we compare our proposed inference method with the CADR (Contextual Adaptive Doubly Robust) method \citep{bibaut2021post} and the Contextual Adaptive Weight (AW) method implemented via StableVar \citep{zhan2021off}. Both are stabilized doubly robust estimators, where we select the prediction model from a linear model, a tree-based model, or a dummy model that always outputs 0. We run all the inference method on the same datasets collected by Boltzmann exploration w.r.t. Ridge regression estimator in five environments introduced above. In Figure \ref{fig::results-ope} (a), we show that CADR, AW and our proposed inference method all achieve empirical coverage close to the target coverage. In Figure \ref{fig::results-ope} (b), we show that our proposed estimator in fact has lower variance under \texttt{NC-Hard2} and \texttt{MS-Polynomial} environments, while maintaining comparable variance in the others.%while the variance being close to for all three methods. %This improvement may be due to the unstable variance estimator of CADR and AW used for stabilization.
This gain may stem from the instability of the variance estimators used in CADR and AW for stabilization, together with the instability of regression functions fitted from adaptively collected data in complex settings.

\begin{figure}[tb]
    \centering
    \includegraphics[width=1\textwidth]{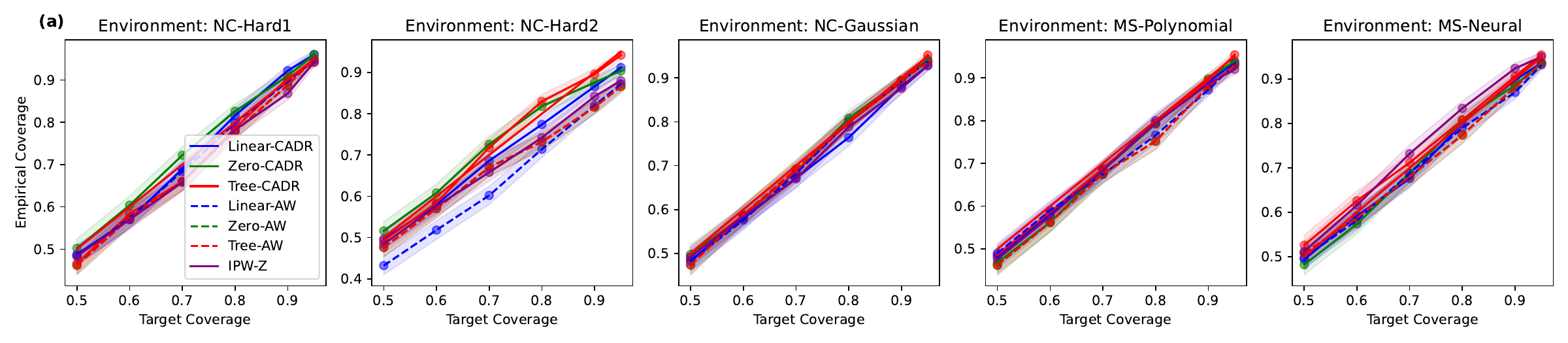}
    \includegraphics[width=1\textwidth]{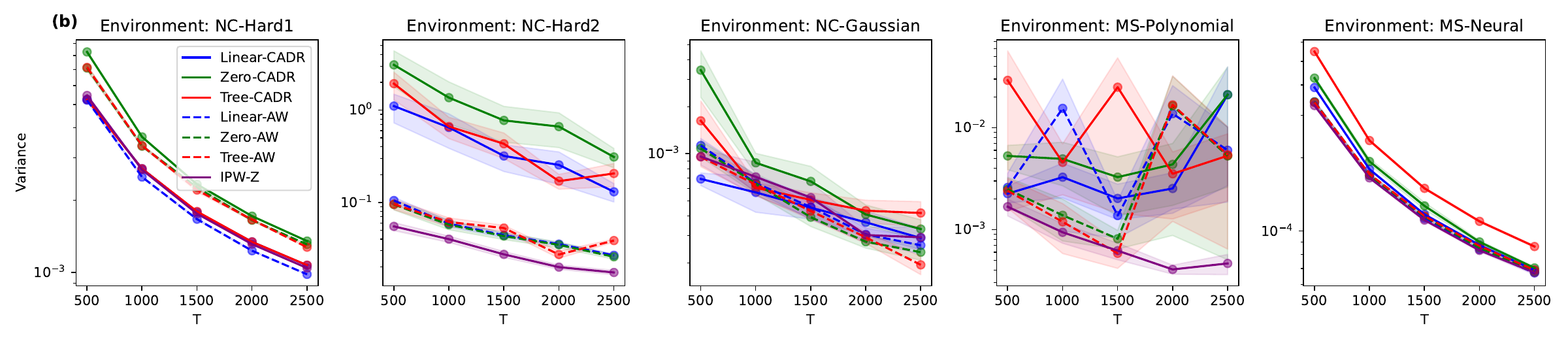}
    \caption{\textbf{(a)} Empirical coverage for confidence intervals based on CADR and AW through StableVar under different prediction models and our proposed inference method across five environments. \textbf{(b)} Monte Carlo estimates (based on 2,500 samples) of the variance of OPE target estimator based on CADR and AW through StableVar under different prediction models and our proposed inference method across five environments over 2,500 steps.}
    \label{fig::results-ope}
\end{figure}

\newpage

\subsection{Acknowledgment} 
Ziping Xu acknowledges support from the following NIH grants during his postdoctoral fellowship at Harvard University: NIH/NIDA P50DA054039, NIH/NIBIB and OD P41EB028242, NIH/NIDCR UH3DE028723, and NIH/NIA 5P30AG073107-03 GY3 Pilots.

\newpage

%%%%%%%%%%%%%%%%%%%%%%%%%%%%%%%%%%%%%%%%%%%%%%
%% Example with multiple Appendixes:        %%
%%%%%%%%%%%%%%%%%%%%%%%%%%%%%%%%%%%%%%%%%%%%%%
 \newpage
\begin{appendix}

\section{Additional Technical Details}\label{appAA}

\subsection{The Clipping Operator}\label{apdx::clipping}

In this section, we present in details a useful transformation, 
$$\mathrm{Clip}: \RR^K\rightarrow \RR^K,$$ which maps a probability distribution over $K$ discrete actions to a new distribution to ensure that each coordinate is lower bounded by $\pi_{\min}$. Specifically, for $\bm{\pi}\in [0, 1]^{K}$, define 
\begin{align}\label{eq::clip}
    \operatorname{Clip}(\bm{\pi}) = \max\{\bm{\pi} - \nu^*(\bm{\pi}), \pi_{\min}\},
\end{align}
where $\nu^*(\bm{\pi})$ is defined as the unique value such that 
$$
q(\nu; \bm{\pi}) \coloneqq \sum_{a\in\cA} \max\{\bm{\pi}_a - \nu, \pi_{\min}\} = 1.
$$
This transformation adjusts $\bpi$ to remain a valid probability distribution while ensuring that each component is at least $\pi_{\min}$. The following lemma shows that this operation is in fact the $L_2$ projection onto the constrained simplex, providing a principled justification for its use. The proof is in Appendix \ref{apdx::proof-lem::clip-l2-projection}.
\begin{lemma}\label{lem::clip-l2-projection}
The mapping $\operatorname{Clip}(\bm{\pi})$ defined in (\ref{eq::clip}) is the $L_2$ projection of $\bm{\pi}$ onto the set 
$$\left\{ \bm{\pi} \in [0, 1]^{|\mathcal{A}|} \bigg| \sum_{a\in\cA} \bm{\pi}_a = 1, \bm{\pi}_a \geq \pi_{\min}\right\}.$$
\end{lemma}

\subsection{Convergence of the Policy in \citep{chen2021statistical}}\label{apdx::convergence-epsilon-greedy-weighted-LS}

In this section, we analyze a behavior policy studied in \citep{chen2021statistical}, which combines an $\epsilon$-greedy algorithm with a weighted online LS estimator, in the setting of misspecified linear bandits. The policy can be written as
\begin{equation}\label{eq::policy-epsilon-greedy-weighted-LS}
    \pi_t(a|\bX_t, \cH_{t-1}) = (1-\epsilon_t)1_{\{(\hat\bbeta_a^{(t-1)}-\hat\bbeta_{1-a}^{(t-1)})^\top\bX_t>0\}} + \frac{\epsilon_t}{2}
\end{equation}
for $a\in\cA = \{0, 1\}$. Here, $\widehat{\bbeta}_a^{(t-1)}$ denotes the weighted online LS estimate of the reward parameter for arm $a$ using data up to time $t-1$, which is a special case of the IPW-Z estimator $\hat{\btheta}_a^{(t-1)}$ defined by (\ref{eq::estimating-equation-general}) with the score function $\bg$ given by (\ref{eq::target-parameter-misspecified-linear-bandits}). $\epsilon_t\in (0, 1)$ is a time-varying exploration parameter such that $\lim_{t\rightarrow \infty}\epsilon_t =\epsilon_{\infty}>0$. 

It is straightforward to verify the above policy is of the form (\ref{eq::IPWZ-policy-plus-parameter}) with $\hat{\btheta}_a^{(t-1)} = \hat{\bbeta}_a^{(t-1)}$,  $\bgamma_{t-1} = \epsilon_t$, and 
$$
\pi(a|\bx, \{\btheta_a\}_{a\in\cA}, \bgamma)= (1-\bgamma)1_{\{(\btheta_a-\btheta_{1-a})^\top\bx>0\}} + \frac{\bgamma}{2}.
$$
The convergence of $\{\bgamma_{t}\}_{t\geq 1}$ is guaranteed by the convergence of $\{\epsilon_{t}\}_{t\geq 1}$. In addition, the function $\pi$ is continuous in $(\{\btheta_a\}_{a\in\cA}, \bgamma)$ at the point $(\{\btheta_a^*\}_{a\in\cA}, \epsilon_{\infty})$ for any $\bx\in\cX_0:= \{\bx: (\btheta_1^* - \btheta_0^*)^\top\bx \neq 0\}$. Here for $a\in\cA$, $\btheta_a^*$ is defined in (\ref{eq::theta-a-*}) with the score function $\bg$ given by (\ref{eq::target-parameter-misspecified-linear-bandits}). Note that $\cX_0^c$ is a Lebesgue null set, and from Assumption 1 of \citep{chen2021statistical}, we deduce that
$$
\PP(\cX_0^c) = 0.
$$
Combining the above arguments, we have verified the conditions of Proposition \ref{thm::convergence-of-IPWZ-policy}, which implies that the policy (\ref{eq::policy-epsilon-greedy-weighted-LS}) satisfies the policy convergence condition in Definition \ref{aspt:policy_convergence}. 

\subsection{Nonconvergence of the Ridge Estimator When the Evolution Lacks a Fixed Point}\label{apdx::ridge-nonconvergence-with-no-fixed-point}

%The following theorem states a necessary condition for the ridge regression estimator $\hat \bbeta_{t}^{\Ridge}$ to converge under a given behavior policy $\pi(a \mid \bX_t, \bm{\beta})$. Equation (\ref{eq:necessary-condition-for-Ridge-convergence}) is a critical condition for non-convergence. Intuitively, it states that there is no fixed point of the estimating equation for the ridge regression estimator. We refer to the \texttt{NC-Hard1} environment in Section \ref{sec::simulation} as a simple example that satisfies this condition.

The following theorem states that, if the behavior policy is parameterized using the ridge estimator, and the policy mapping induces no fixed point in the evolution dynamics of the ridge statistics, then the ridge estimator will not converge. 

\begin{theorem}\label{theorem:necessary-condition-for-Ridge-convergence}
Assume the conditions of Proposition \ref{theorem:convergence-of-Ridge}. Suppose that the behavior policy takes the form $\pi_t(a|\bX_t, \cH_{t-1}) = \pi(a|\bX_t, \hat \bbeta_{t}^{\Ridge})$, where $\hat \bbeta_{t}^{\Ridge}$ is the ridge estimator defined in (\ref{eq:ridge-estimator}), and the policy mapping $\pi: \cX\times \RR^{Kd}\rightarrow\Delta(\cA)$ satisfies:
\begin{itemize}
\item $\pi(a \mid \bx, \bm{\beta})$ is a continuous function of $\bm{\beta}$ for any $\bx \in \cX$,

\item $\pi(a \mid \bx, \bm{\beta})>0$ for any $\bx \in \cX, \bm{\beta}\in\RR^{Kd}$,

\item There exists $\bar a\in\cA$, such that for any $\bm{\beta} = (\bbeta_1^\top, \ldots, \bbeta_K^\top)^\top\in \mathbb{R}^{Kd}$,  
    \begin{align}
    \bm{\beta}_{\bar a} \neq \bSigma_{\bar a}^{-1}(\bm{\beta}) \bm{\varphi}_{\bar a}(\bm{\beta}). \label{eq:necessary-condition-for-Ridge-convergence}
    \end{align}
Here $\bSigma_{a}(\bm{\beta}) = \EE_{\bX_t, A_t \sim \pi(\cdot \mid \bX_t, \bm{\beta})}[ 1_{\{A_t = a\}} \bX_t \bX_t^\top]$, and \\  $\bm{\varphi}_a(\bm{\beta}) = \EE_{\bX_t, A_t\sim \pi(\cdot \mid \bX_t, \bm{\beta}), Y_t}[ 1_{\{A_t = a\}} \bX_t Y_t ]$.
\end{itemize}
Then the ridge estimator $\hat \bbeta_{t}^{\Ridge}$ does not converge in probability as $t\rightarrow\infty$.
\end{theorem}

The proof is given in Appendix \ref{apdx::proof-thm::necessary-condition-for-Ridge-convergence}.

\subsection{Asymptotic Stability}\label{apdx::asymptotic-stability}

\begin{definition}[Asymptotic Stability]
    \label{def::asymptotic-stability}
    The equilibrium point $x=0$ of $\dot{x} = f(x)$ is 
    \begin{itemize}
        \item stable if, for each $\varepsilon>0$, there is $\delta=\delta(\varepsilon)>0$ such that
        $$
            \|x(0)\|<\delta \Rightarrow\|x(t)\|<\varepsilon, \quad \forall t \geq 0
        $$
        \item unstable if it is not stable.
        \item asymptotically stable if it is stable and $\delta$ can be chosen such that
        $$
            \|x(0)\|<\delta \Rightarrow \lim _{t \rightarrow \infty} x(t)=0
        $$
    \end{itemize}
\end{definition}

\subsection{Inference for Bandits with Noisy Contexts under Estimated Contextual Error Variance}\label{apdx::ex2-estimated-Sigma-e}

In this section, we discuss the statistical inference task in Example \ref{ex::bandits-noisy-contexts} where the contextual error variance $\bSigma_e$ is unknown. In addition to the adaptively collected dataset $\cD$, suppose we have an auxiliary offline dataset
$$
\tilde \cD = \{\tilde{\bX}_i,  \tilde\bS_i\}_{i=1}^n
$$
containing paired noisy and true contexts. Each $(\tilde{\bX}_i, \tilde\bS_i)$ is independent and equal in distribution to $(\bX_t, \bS_t)$ in $\cD$. Such auxiliary data arise naturally when gold-standard measurements are available for a subset of observations, or when the observed context is generated by a complex prediction algorithm, and its error variance can be estimated using gold-standard labels. In healthcare applications, $\tilde\cD$ is often obtained from prior studies, such as pilot trials. Because $\tilde\cD$ is collected offline without intervention or adaptive sampling, it is typically accessible in practice.

Under this setting, notice that 
$$
\bSigma_e = \mathrm{Var}(\bX_t|\bS_t) = \EE\big[(\bX_t-\bS_t)(\bX_t-\bS_t)^\top\big|\bS_t\big] = \EE\big[(\bX_t-\bS_t)(\bX_t-\bS_t)^\top\big],
$$
we can estimate $\bSigma_e$ using 
$$
\hat\bSigma_e = \frac1n\sum_{i=1}^n\tilde\bV_i,
$$
where $\tilde\bV_i:= (\tilde\bX_i-\tilde\bS_i)(\tilde \bX_i-\tilde\bS_i)^\top$. A natural estimator for $\btheta_a^*$ would be $\tilde \btheta_a^{(T)}$ which solves $\tilde \bG_T(\btheta) = \mathbf{0}$, where 
$$
\tilde \bG_T(\btheta):= \frac1T\sum_{t\in[T]}W_t1_{\{A_t = a\}}\tilde \bg(\bX_t, Y_t;\btheta),
$$
where $W_t = \frac{1}{\pi_t(A_t)}$, $\tilde \bg(\bx, y;\btheta) = (\bx \bx^\top - \hat\bSigma_e)\btheta - \bx y$. 

The following theorem characterize the asymptotic distribution of $\tilde \btheta_a^{(T)}$, which can be used to conduct inference on $\btheta_a^*$. The proof is in Appendix \ref{apdx::proof-thm::inference-with-Sigma_e_hat}. 

\begin{theorem}\label{thm::inference-with-Sigma_e_hat}
    Under Assumptions \ref{aspt:unconfoundedness},  \ref{aspt:min-sampling-prob} and \ref{aspt:ex2}, define $\bSigma_{S}$ and $\bar\bI_a$ the same as in Corollary \ref{cor::ex2}. Then if the behavior policy converges in the sense of Definition \ref{aspt:policy_convergence}, as $n, T\rightarrow \infty$, 
    \begin{itemize}
        \item If $n/T\rightarrow \infty$, then $\sqrt{T}(\tilde \btheta_a^{(T)} - \btheta_a^*)\xrightarrow{d}\cN\left(\mathbf{0}, \bSigma_{S}^{-1}\bar\bI_a\bSigma_{S}^{-1}\right)$.
        \item If $n/T\rightarrow \kappa$ for some positive constant $\kappa$, then $\sqrt{T}(\tilde \btheta_a^{(T)} - \btheta_a^*)\xrightarrow{d}\cN\left(0, \bSigma_{S}^{-1}\big[{\bar\bI_a}+\frac{\bar \bH_a}{\kappa}\big]\bSigma_{S}^{-1}\right)$, where $\bar \bH_a: = \EE(\tilde\bV_i-\bSigma_e)\btheta_a^*\btheta_a^{*, \top}(\tilde\bV_i-\bSigma_e)$.
        \item If $n/T\rightarrow 0$, then $\sqrt{n}(\tilde \btheta_a^{(T)} - \btheta_a^*)\xrightarrow{d}\cN\left(\mathbf{0}, \bSigma_{S}^{-1}\bar\bH_a\bSigma_{S}^{-1}\right)$, where 
        $$\bar \bH_a: = \EE(\tilde\bV_i-\bSigma_e)\btheta_a^*\btheta_a^{*, \top}(\tilde\bV_i-\bSigma_e).$$
    \end{itemize}
\end{theorem}

By Theorem \ref{thm::inference-with-Sigma_e_hat}, if the sample size of the auxiliary data $\tilde\cD$ is sufficiently larger than that of the adaptively collected data $\cD$, the distribution of $\tilde\btheta_a^{(T)}$ coincides with that of $\hat\btheta_a^{(T)}$. This scenario is common in practice since, unlike $\cD$, the auxiliary dataset $\tilde\cD$ involves no intervention and is typically easier to obtain. In case $\tilde\cD$ has comparable or smaller sample size compared to $\cD$, valid inference remains possible, but the asymptotic variance of $\tilde\btheta_a^{(T)}$ exceeds that of $\hat\btheta_a^{(T)}$ due to the additional uncertainty from estimating $\bSigma_e$. The joint asymptotic distribution of $\{\tilde\btheta_a^{(T)}\}_{a\in\cA}$ can be derived similar to Theorem \ref{thm::asymptotic-normality-joint-general}, and is omitted for brevity.

%Finally, the following proposition shows that all the asymptotic variance in Theorem \ref{thm::inference-with-Sigma_e_hat} can be consistently estimated. Its proof is in Appendix \ref{???}.

%\YG{Add consistent estimator of the variance later}

\section{Proofs of Main Theorems}\label{appA}
We organize the proofs of the main theorems by sections in the main text.

    % \subsection{Main Theorems in Section \ref{sec::inference-guarantee}}\label{apdx::missing-proofs-inference-guarantee}

    \subsection{Proof of Theorem \ref{thm::asymptotic-normality-general}}\label{apdx::proof-thm::asymptotic-normality-general}
    
    We first prove the following lemma. Its proof is in Appendix \ref{apdx::proof-lem::consistency-general}.
    
    \begin{lemma}\label{lem::consistency-general}
        Under the assumptions of Theorem \ref{thm::asymptotic-normality-general}, there exists a sequence of estimators $\{\hat\btheta_a^{(T)}\}_{T\geq 1}$ such that (\ref{eq::estimating-equation-general}) holds, and $\|\hat\btheta_a^{(T)}\|_2\leq R_{\btheta}$, $\forall T$. In addition, for any such sequence, as $T\rightarrow \infty$, $\hat\btheta_a^{(T)}\xrightarrow{p}\btheta_a^*$.
    \end{lemma}
    
    For the remainder of this proof, for notational convenience, we omit the dependence of $\hat\btheta_a^{(T)}$ on $T$ and write it as $\hat\btheta_a$. Let $\bG_T^{(i)}(\btheta)$ denote the $i$-th entry of $\bG_T(\btheta)$. By Taylor expansion, we have that for any $i\in\{1, \ldots, d\}$, there exists some $\tilde\btheta_{a, i}$ on the line segment between $\btheta_a^*$ and $\hat\btheta_a$ such that  
    \begin{align*}
    -\bG_T^{(i)}(\btheta_a^*)&= \bG_T^{(i)}(\hat\btheta_a) - \bG_T^{(i)}(\btheta_a^*)+o_p(1/\sqrt{T})\\ 
    &= \langle\nabla\bG_T^{(i)}(\btheta_a^*), \hat\btheta_a-\btheta_a^*\rangle + \frac12(\hat\btheta_a-\btheta_a^*)^\top\nabla^2\bG_T^{(i)}(\tilde\btheta_{a, i})(\hat\btheta_a-\btheta_a^*)+o_p(1/\sqrt{T}).
    \end{align*}
    Stacking the above expansions over the entries $i=1, \ldots, d$, we have 
    $$
    -\bG_T(\btheta_a^*)= \nabla\bG_T(\btheta_a^*)(\hat\btheta_a-\btheta_a^*) + \frac12\tilde{\bm{\delta}}_a(\hat\btheta_a-\btheta_a^*)+o_p(1/\sqrt{T}),
    $$
    where 
    $$
    \tilde{\bm{\delta}}_a = 
    \begin{pmatrix}
    (\hat\btheta_a-\btheta_a^*)^\top\nabla^2\bG_T^{(1)}(\tilde\btheta_{a, 1})\\
    \vdots\\
    (\hat\btheta_a-\btheta_a^*)^\top\nabla^2\bG_T^{(d)}(\tilde\btheta_{a, d})
    \end{pmatrix}.
    $$
    By rearranging, we obtain 
    \begin{equation}\label{eq::taylor-expansion-general}
        \big[\nabla\bG_T(\btheta_a^*) + \frac12\tilde{\bm{\delta}}_a \big]\cdot \sqrt{T}(\hat\btheta_a - \btheta_a^*) = -\sqrt{T}\bG_T(\btheta_a^*) + o_p(1).
    \end{equation}
    Below we state the following lemmas, the proof of these lemmas are in Appendix \ref{apdx::proof-lem::convergence-true-derivative-general}, \ref{apdx::proof-lem::2nd-derivative-boundness-general}, and \ref{apdx::proof-lem::asymptotic-normality-true-G-general}, respectively.
    
    \begin{lemma}\label{lem::convergence-true-derivative-general}
        Under the assumptions of Theorem \ref{thm::asymptotic-normality-general}, as $T\rightarrow \infty$, $\nabla\bG_T(\btheta_a^*)\xrightarrow{p}\EE\nabla\bg(\bX_t, Y_t(a);\btheta_a^*)$.
    \end{lemma}
    
    \begin{lemma}\label{lem::2nd-derivative-boundness-general}
        Under the assumptions of Theorem \ref{thm::asymptotic-normality-general}, $\sup_{\|\btheta - \btheta_a^*\|_2\leq \epsilon_0}\|\nabla^2\bG_T(\btheta)\|_1 = \cO_p(1)$. Here, for a tensor $\bB\in\RR^{d_1\times d_2\times d_3}$, we define $\|\bB\|_1 = \sum_{i\in[d_1], j\in[d_2], k\in[d_3]}|\bB_{i, j, k}|$.
    \end{lemma}
    
    \begin{lemma}\label{lem::asymptotic-normality-true-G-general}
        Under the assumptions of Theorem \ref{thm::asymptotic-normality-general}, as $T\rightarrow \infty$, $\sqrt{T}\bG_T(\btheta_a^*)\xrightarrow{d} \cN(\mathbf{0}, \bar\bI_a)$, where $\bar\bI_a = \EE\big[\frac{1}{\bar\pi(a|\bX_t)}\bg(\bX_t, Y_t(a);\btheta_a^*)\bg(\bX_t, Y_t(a);\btheta_a^*)^\top\big]$.
    \end{lemma}
    
    We now derive the asymptotic distribution of $\hat\btheta_a$ from (\ref{eq::taylor-expansion-general}). First, since $\tilde\btheta_{a, i}$ is on the line segment between $\btheta_a^*$ and $\hat\btheta_a$, from Lemma \ref{lem::consistency-general} and Lemma \ref{lem::2nd-derivative-boundness-general}, we have $\forall i$,
    \begin{align*}
        \|\nabla^2\bG_T^{(i)}(\tilde\btheta_{a, i})\|_{1, 1} &\leq  \|\nabla^2\bG_T(\tilde\btheta_{a, i})\|_1\\
        &= \|\nabla^2\bG_T(\tilde\btheta_{a, i})\|_1\cdot1_{\{\|\tilde\btheta_{a, i}-\btheta_a^*\|_2\leq \epsilon_0\}} + \|\nabla^2\bG_T(\tilde\btheta_{a, i})\|_1\cdot1_{\{\|\tilde\btheta_{a, i}-\btheta_a^*\|_2> \epsilon_0\}}\\
        &\leq \sup_{\|\btheta - \btheta_a^*\|_2\leq \epsilon_0}\|\nabla^2\bG_T(\btheta)\|_1 + \|\nabla^2\bG_T(\tilde\btheta_{a, i})\|_1\cdot 1_{\{\|\hat\btheta_a-\btheta_a^*\|_2> \epsilon_0\}}= \cO_p(1).
    \end{align*}
    Here for a matrix $\bB\in\RR^{d_1\times d_2}$, we define $\|\bB\|_{1, 1} = \sum_{i\in[d_1], j\in[d_2]}|\bB_{i, j}|$. 
    
    Combine the above with Lemma \ref{lem::consistency-general} which implies $\hat\btheta_a - \btheta_a^* = o_p(1)$ and Lemma \ref{lem::convergence-true-derivative-general} which ensures convergence of $\nabla\bG_T(\btheta_a^*)$, we deduce that 
    \begin{equation}\label{eq::gradient-plus-convergence-general}
        \nabla\bG_T(\btheta_a^*) + \frac12\tilde{\bm{\delta}}_a\xrightarrow{p}\EE\nabla\bg(\bX_t, Y_t(a);\btheta_a^*).
    \end{equation}
    We further combine the above expression with Lemma \ref{lem::asymptotic-normality-true-G-general} and use Slutsky's theorem to obtain (\ref{eq::asymptotic-normality-general}).
    
\subsection{Proof of Theorem \ref{thm::asymptotic-normality-joint-general}}\label{apdx::proof-thm::asymptotic-normality-joint-general}

We prove a more general version of Theorem \ref{thm::asymptotic-normality-joint-general} where the score function $\bg$ in (\ref{eq::theta-a-*}) for each arm can be different. Suppose that for any $a\in\cA$, $\btheta_a^*$ satisfies 
\begin{equation}\label{eq::theta-a-*-a}
    \EE\bg_a(\bX, Y(a);\btheta_a^*) = \mathbf{0}
\end{equation}
for a score function $\bg_a$. 

The following conditions extend Assumptions \ref{aspt:identifiability-general}, \ref{aspt:boundedness-general}, \ref{aspt:smoothness-general}, and \ref{aspt:min-sampling-prob} to the more general setting where the conditions apply not only to a single arm but simultaneously across all arms $a\in\cA$.

\begin{assumption}[Well-separated solution for all arms]\label{aspt:identifiability-general-joint}
    $\forall a\in\cA, \forall \epsilon>0$, \\ $\inf_{\|\btheta - \btheta_a^*\|_2>\epsilon}\|\EE\bg_a(\bX_t, Y_t(a);\btheta)\|_2>0$.
\end{assumption}

\begin{assumption}[Boundedness for all arms]\label{aspt:boundedness-general-joint}
    There exist constants $R_{\btheta}$, $M_2$ such that \\
    (i) $\forall a\in\cA$, $\|\EE[\bg_a(\bX_t, Y_t(a);\btheta_a^*)\bg_a(\bX_t, Y_t(a);\btheta_a^*)^\top|\bX_t]\|_2\leq M_2$, a.e. $\bX_t$; \\
    (ii) $\forall a\in\cA$, $\|\btheta_a^*\|_2<R_{\btheta}$, $\sup_{\|\btheta\|_2\leq R_{\btheta}}\EE\|\bg_a(\bX_t, Y_t(a);\btheta)\|_2^2<\infty$; \\
    (iii) $\forall a\in\cA$, $\EE \|\bg_a(\bX_t, Y_t(a);\btheta_a^*)\|_2^4<\infty$.
\end{assumption}

\begin{assumption}[Smoothness for all arms]\label{aspt:smoothness-general-joint}
(i)$\forall a\in\cA$, the function $\bg_a(\bx, y; \btheta)$ is twice differentiable with respect to $\btheta$, with $\EE\nabla\bg_a(\bX_t, Y_t(a);\btheta_a^*)$ nonsingular; \\
(ii) There exists a function $\phi$ such that $\forall a\in\cA$, $\forall \bx, y$, $\sup_{\|\btheta\|_2\leq R_{\btheta}}\|\nabla\bg_a(\bx, y; \btheta)\|_2\leq \phi(\bx, y)$, and $\EE\phi(\bX_t, Y_t(a))^2<\infty$;\\
(iii)There exists a constant $\epsilon_0>0$ and a function $\Phi$ such that $\forall a\in\cA$, \\
$\sup_{\|\btheta - \btheta_a^*\|_2\leq \epsilon_0, i\in[d]}\|\nabla^2\bg_a^{(i)}(\bx, y; \btheta)\|_2\leq\Phi(\bx, y)$ and $\EE\Phi(\bX_t, Y_t(a))<\infty$. Here $\bg_a^{(i)}(\bx, y; \btheta)$ denotes the $i$-th entry of $\bg_a(\bx, y; \btheta)$.
\end{assumption}

\begin{assumption}[Minimum sampling probability for all arms] \label{aspt:min-sampling-prob-joint}
$\forall a\in\cA$, $\pi_t(a)\geq \pi_{\min}$ almost surely for some constant $\pi_{\min} \in(0, 1)$.
\end{assumption}

%\begin{assumption}[Policy convergence for all arms] 
%\label{aspt:policy-convergence-joint} There exists a policy $\bar\pi: \cX \mapsto \Delta(\cA)$ such that $\forall a\in\cA$,  $\pi_t(a|\bX_t, \cH_{t-1}) - \bar\pi(a|\bX_t)\xrightarrow{p} 0$ as $t \rightarrow \infty$. 
%\end{assumption}

Define $\bG_{a, T} := \frac1T\sum_{t=1}^T\frac1{\pi_t(A_t)}1_{\{A_t = a\}}\bg_a(\bX_t, Y_t;\btheta)$, and similar to (\ref{eq::estimating-equation-general}), we look for $\hat\btheta_a^{(T)}$ such that $\forall a\in\cA$,
\begin{equation}\label{eq::estimating-equation-general-A}
    \bG_{a, T}(\hat\btheta_a) = o_p(1/\sqrt{T})
\end{equation}
as $T\rightarrow \infty$. We prove the following theorem, which is a generalization of Theorem \ref{thm::asymptotic-normality-joint-general}.

\begin{theorem}\label{thm::asymptotic-normality-general-joint-A}
    Under Assumptions \ref{aspt:unconfoundedness}, \ref{aspt:identifiability-general-joint}, \ref{aspt:boundedness-general-joint}, \ref{aspt:smoothness-general-joint}, and \ref{aspt:min-sampling-prob-joint}, if the behavior policy $\pi$ satisfies policy convergence to a policy $\bar\pi$ in the sense of Definition \ref{aspt:policy_convergence}, then there exist estimators $\{\hat\btheta_a^{(T)}\}_{a\in\cA, T\geq 1}$ such that (\ref{eq::estimating-equation-general-A}) holds for any $a\in\cA$, and $\|\hat\btheta_a^{(T)}\|_2\leq R_{\btheta}$, $\forall a\in\cA, T\geq 1$. In addition, any such estimators satisfy
\begin{equation}\label{eq::asymptotic-normality-joint-general-A}
    \sqrt{T}(\hat\btheta^{(T)} - \btheta^*)\xrightarrow{d}\cN
    \left(
    \mathbf{0},\bSigma^*
    \right)
\end{equation}
as $T\rightarrow \infty$. Here $\hat\btheta^{(T)} = \big((\hat\btheta_1^{(T)})^\top, \ldots, (\hat\btheta_K^{(T)})^\top\big)^\top$, $\btheta^* = \big(\btheta_1^{*, \top}, \ldots, \btheta_K^{*, \top}\big)^\top$, and $\bSigma^* = \diag(\bSigma_1^*, \ldots, \bSigma_K^*)$, where $\bSigma_{a}^*:= \bJ_a^{-1}\bar \bI_a \bJ_a^{-1, \top}$, and $$\bar\bI_a := \EE\big[\frac{1}{\bar\pi(a|\bX_t)}\bg_a(\bX_t, Y_t(a);\btheta_a^*)\bg_a(\bX_t, Y_t(a);\btheta_a^*)^\top\big], \bJ_a:= \EE\nabla\bg_a(\bX_t, Y_t(a);\btheta_a^*).$$
\end{theorem}

\begin{proof}[Proof of Theorem \ref{thm::asymptotic-normality-general-joint-A}]
First, the existence of $\{\hat\btheta_a^{(T)}\}_{a\in\cA, T\geq 1}$ is guaranteed by applying Theorem \ref{thm::asymptotic-normality-general} to each individual arm $a\in\cA$. Now suppose $\{\hat\btheta_a^{(T)}\}_{a\in\cA, T\geq 1}$ satisfies (\ref{eq::estimating-equation-general-A}) $\forall a\in\cA$, and $\|\hat\btheta_a^{(T)}\|_2\leq R_{\btheta}$ $\forall a\in\cA, T\geq 1$. From the Cramer-Wold theorem, in order to prove (\ref{eq::asymptotic-normality-joint-general-A}), we only need to show that for any nonrandom vectors $\bbeta_1, \ldots, \bbeta_K\in\RR^d$, 
\begin{equation}\label{eq::asymptotic-normality-general-linear-combination}
    \sqrt{T}\sum_{a\in\cA}\bbeta_a^\top (\hat\btheta_a - \btheta_a^*)\xrightarrow{d} \cN\left(0, \sum_{a\in\cA}\bbeta_a^\top \bJ_a^{-1}\bar\bI_a\bJ_a^{-1, \top}\bbeta_a\right).
\end{equation}
First, for any fixed arm $a$, we let $\bg = \bg_a$ and $\bG_T = \bG_{a, T}$ in (\ref{eq::taylor-expansion-general}), and considering Lemma \ref{lem::asymptotic-normality-true-G-general} as well as (\ref{eq::gradient-plus-convergence-general}), we obtain that 
\begin{equation*}
    \sqrt{T}(\hat\btheta_a - \btheta_a^*) = \big[\bJ_a + \bm{\delta}_{a, T}\big]^{-1}\sqrt{T}\bG_{a,T}(\btheta_a^*) + o_p(1)
\end{equation*}
where $\bm{\delta}_{a, T} = o_p(1)$. By further analysis, we have 
\begin{align}
\sqrt{T}(\hat\btheta_a - \btheta_a^*)& = \bJ_a^{-1}\cdot \sqrt{T}\bG_{a,T}(\btheta_a^*) + \left(\big[\bJ_a + \bm{\delta}_{a, T}\big]^{-1} - \bJ_a^{-1}\right)\cdot \sqrt{T}\bG_{a,T}(\btheta_a^*) + o_p(1)\nonumber\\
& = \bJ_a^{-1}\cdot \sqrt{T}\bG_{a,T}(\btheta_a^*) + \left([\bI + \bJ_a^{-1}\bm{\delta}_{a, T}]^{-1} - \bI\right)\bJ_a^{-1}\cdot \sqrt{T}\bG_{a,T}(\btheta_a^*) + o_p(1)\nonumber\\
& = \bJ_a^{-1}\cdot \sqrt{T}\bG_{a,T}(\btheta_a^*)+ o_p(1).\label{eq::asymptotic-expansion-a}
\end{align}
Here in the last equation, we have used the fact that $\bJ_a$ is invertible, $\bm{\delta}_{a, T} = o_p(1)$, and $\sqrt{T}\bG_{a,T}(\btheta_a^*) = \cO_p(1)$ from Lemma \ref{lem::asymptotic-normality-true-G-general}.
Applying (\ref{eq::asymptotic-expansion-a}) to all $a\in\cA$, we deduce that
\begin{align}
\sqrt{T}\sum_{a\in\cA}\bbeta_a^\top(\hat\btheta_a - \btheta_a^*)
& = \sum_{a\in\cA}\bbeta_a^\top\bJ_a^{-1}\cdot \sqrt{T}\bG_{a,T}(\btheta_a^*)+ o_p(1)\nonumber\\
& = \sum_{a\in\cA}\bbeta_a^\top\bJ_a^{-1}\cdot \frac1{\sqrt{T}}\sum_{t=1}^T\frac1{\pi_t(A_t)}1_{\{A_t = a\}}\bg_a(\bX_t, Y_t;\btheta_a^*)+ o_p(1)\nonumber\\
& = \frac1{\sqrt{T}}\sum_{t=1}^T\bar Z_t+ o_p(1),
\end{align}
where $\bar Z_t := \sum_{a\in\cA}\frac1{\pi_t(A_t)}1_{\{A_t = a\}}\bbeta_a^\top\bJ_a^{-1}\bg_a(\bX_t, Y_t;\btheta_a^*)$. Given the above expression, from Theorem 2.2 in \citep{dvoretzky1972asymptotic}, (\ref{eq::asymptotic-normality-general-linear-combination}) can be obtained by ensuring 
\begin{gather}
   \EE [\bar Z_t|\cH_{t-1}] = 0\quad  \forall t\in[T],\label{eq::cond-exp-general-linear-combination}\\
   \frac1T\sum_{t\in[T]}\mathrm{Var}(\bar Z_t|\cH_{t-1}) \xrightarrow{p} \sum_{a\in\cA}\bbeta_a^\top \bJ_a^{-1}\bar\bI_a\bJ_a^{-1, \top}\bbeta_a,\label{eq::cond-var-general-linear-combination}\\
   \frac1T\sum_{t\in[T]}\EE\left[\bar Z_t^21_{\{|\bar Z_t|>\sqrt{T}\delta\}}\Big|\cH_{t-1}\right]\xrightarrow{p} 0\quad \forall \delta>0.\label{eq::cond-lindeberg-general-linear-combination}
\end{gather}
Below we check these facts one by one.

\textbf{Check (\ref{eq::cond-exp-general-linear-combination})}: We have 
\begin{align*}
&\EE [\bar Z_t|\cH_{t-1}] \\
&=
\EE_{\bX_t}\left[\EE_{A_t, \{Y_t(a)\}_{a\in\cA}}[\bar Z_t|\cH_{t-1}, \bX_t]\big|\cH_{t-1}\right]\\
& = \EE_{\bX_t}\left[\EE_{A_t, \{Y_t(a)\}_{a\in\cA}}\bigg[\sum_{a\in\cA}\frac1{\pi_t(A_t)}1_{\{A_t = a\}}\bbeta_a^\top\bJ_a^{-1}\bg_a(\bX_t, Y_t(a);\btheta_a^*)\Big|\cH_{t-1}, \bX_t\bigg]\bigg|\cH_{t-1}\right]\\
& = \EE_{\bX_t}\left[\sum_{a\in\cA}\EE_{A_t,Y_t(a)}\bigg[\frac1{\pi_t(A_t)}1_{\{A_t = a\}}\bbeta_a^\top\bJ_a^{-1}\bg_a(\bX_t, Y_t(a);\btheta_a^*)\Big|\cH_{t-1}, \bX_t\bigg]\bigg|\cH_{t-1}\right]\\
& = \EE_{\bX_t}\bigg[\sum_{a\in\cA}\EE_{A_t}\bigg[\frac1{\pi_t(A_t)}1_{\{A_t = a\}}\Big|\cH_{t-1}, \bX_t\bigg]\cdot \\
& \quad\quad\quad\quad\quad\quad
\EE_{Y_t(a)}\bigg[\bbeta_a^\top\bJ_a^{-1}\bg_a(\bX_t, Y_t(a);\btheta_a^*)\Big|\cH_{t-1}, \bX_t\bigg]\bigg|\cH_{t-1}\bigg]\\
& = \EE_{\bX_t}\left[\sum_{a\in\cA} 
\EE_{Y_t(a)}\bigg[\bbeta_a^\top\bJ_a^{-1}\bg_a(\bX_t, Y_t(a);\btheta_a^*)\Big|\cH_{t-1}, \bX_t\bigg]\bigg|\cH_{t-1}\right]\\
& = \sum_{a\in\cA}\bbeta_a^\top\bJ_a^{-1}\EE\bg_a(\bX_t, Y_t(a);\btheta_a^*) = 0.
\end{align*}
Here in the fourth equality we use Assumption \ref{aspt:unconfoundedness}, and the last equality is due to (\ref{eq::theta-a-*-a}).

\textbf{Check (\ref{eq::cond-var-general-linear-combination})}: Due to (\ref{eq::cond-exp-general-linear-combination}), we have 
$\frac1T\sum_{t\in[T]}\mathrm{Var}(\bar Z_t|\cH_{t-1}) = \frac1T\sum_{t\in[T]}\EE[\bar Z_t^2|\cH_{t-1}]$, where
\begin{align}
\EE[\bar Z_t^2|\cH_{t-1}]
& = \EE_{\bX_t}[\EE_{A_t, \{Y_t(a)\}_{a\in\cA}}[\bar Z_t^2|\cH_{t-1}, \bX_t]|\cH_{t-1}]\nonumber\\
& = \EE_{\bX_t}\bigg[\EE_{A_t, \{Y_t(a)\}_{a\in\cA}}\bigg[\Big(\sum_{a\in\cA}\frac1{\pi_t(A_t)}1_{\{A_t = a\}}\cdot\nonumber\\
&\quad\quad\quad\quad\quad\quad\quad\quad\quad\quad\quad\quad\bbeta_a^\top\bJ_a^{-1}\bg_a(\bX_t, Y_t(a);\btheta_a^*)\Big)^2\Big|\cH_{t-1}, \bX_t\bigg]\bigg|\cH_{t-1}\bigg]\nonumber\\
& = \EE_{\bX_t}\bigg[\EE_{A_t, \{Y_t(a)\}_{a\in\cA}}\bigg[\sum_{a, a'\in\cA}\frac{1_{\{A_t = a\}}\cdot 1_{\{A_t = a'\}}}{\pi_t(A_t)^2}\cdot\nonumber\\
&\quad \quad\bbeta_a^\top\bJ_a^{-1}\bg_a(\bX_t, Y_t(a);\btheta_a^*)\bg_{a'}(\bX_t, Y_t(a');\btheta_{a'}^*)^\top\bJ_{a'}^{-1, \top}\bbeta_{a'}\Big|\cH_{t-1}, \bX_t\bigg]\bigg|\cH_{t-1}\bigg]\nonumber\\
& = \EE_{\bX_t}\bigg[\EE_{A_t, \{Y_t(a)\}_{a\in\cA}}\bigg[\sum_{a\in\cA}\frac{1_{\{A_t = a\}}}{\pi_t(A_t)^2}\cdot \nonumber\\
&\quad\quad\quad \bbeta_a^\top\bJ_a^{-1}\bg_a(\bX_t, Y_t(a);\btheta_a^*)\bg_{a}(\bX_t, Y_t(a);\btheta_{a}^*)^\top\bJ_{a}^{-1, \top}\bbeta_{a}\Big|\cH_{t-1}, \bX_t\bigg]\bigg|\cH_{t-1}\bigg]\nonumber\\
& = \EE_{\bX_t}\bigg[\sum_{a\in\cA}\EE_{A_t, Y_t(a)}\bigg[\frac{1_{\{A_t = a\}}}{\pi_t(A_t)^2}\cdot \nonumber\\
&\quad\quad\quad \bbeta_a^\top\bJ_a^{-1}\bg_a(\bX_t, Y_t(a);\btheta_a^*)\bg_{a}(\bX_t, Y_t(a);\btheta_{a}^*)^\top\bJ_{a}^{-1, \top}\bbeta_{a}\Big|\cH_{t-1}, \bX_t\bigg]\bigg|\cH_{t-1}\bigg]\nonumber\\
& = \EE_{\bX_t}\bigg[\sum_{a\in\cA}\EE_{A_t}\bigg[\frac{1_{\{A_t = a\}}}{\pi_t(A_t)^2}\Big|\cH_{t-1}, \bX_t\bigg]\cdot \nonumber\\
& \thickspace
\EE_{Y_t(a)} \bigg[\bbeta_a^\top\bJ_a^{-1}\bg_a(\bX_t, Y_t(a);\btheta_a^*)\bg_{a}(\bX_t, Y_t(a);\btheta_{a}^*)^\top\bJ_{a}^{-1, \top}\bbeta_{a}\Big|\cH_{t-1}, \bX_t\bigg]\bigg|\cH_{t-1}\bigg]\nonumber\\
& = \EE_{\bX_t}\bigg[\sum_{a\in\cA}\frac1{\pi_t(a)}\cdot \bbeta_a^\top\bJ_a^{-1}\cdot \nonumber\\
& \quad\quad\quad\EE_{Y_t(a)} \bigg[\bg_a(\bX_t, Y_t(a);\btheta_a^*)\bg_{a}(\bX_t, Y_t(a);\btheta_{a}^*)^\top\Big|\cH_{t-1}, \bX_t\bigg]\bJ_{a}^{-1, \top}\bbeta_{a}\bigg|\cH_{t-1}\bigg]\nonumber\\
& = \EE_{\bX_t}\bigg[\sum_{a\in\cA}\bbeta_a^\top\bJ_a^{-1}\bI_{a, t}\bJ_{a}^{-1, \top}\bbeta_{a}\bigg|\cH_{t-1}\bigg]\nonumber\\
& = \sum_{a\in\cA}\bbeta_a^\top\bJ_a^{-1}\EE_{\bX_t}\big[\bI_{a, t}\big|\cH_{t-1}\big]\bJ_{a}^{-1, \top}\bbeta_{a}.\label{eq::conditional-var-linear-combination-1}
\end{align}
Here $\bI_{a, t} := \frac{1}{\pi_t(a)}\cdot \EE_{Y_t(a)}[\bg_a(\bX_t, Y_t(a);\btheta_a^*)\bg_a(\bX_t, Y_t(a);\btheta_a^*)^\top|\bX_t]$. This is consistent with the definition in Appendix \ref{apdx::proof-lem::asymptotic-normality-true-G-general}, where $\bI_{a, t}$ is defined for a single fixed arm $a$, by taking $\bg = \bg_a$ in the more general setting considered here. In the above, the sixth equality uses Assumption \ref{aspt:unconfoundedness}. 

From (\ref{eq::l1-convergence-1-general}) and (\ref{eq::l1-convergence-2-general}) with $\bc = \bJ_{a}^{-1, \top}\bbeta_{a}$, we obtain that $\forall a\in\cA$, 
\begin{equation}\label{eq::conditional-var-linear-combination-2}
\frac1T\sum_{t=1}^T\bbeta_a^\top\bJ_a^{-1}\EE_{\bX_t}\big[\bI_{a, t}\big|\cH_{t-1}\big]\bJ_{a}^{-1, \top}\bbeta_{a}\xrightarrow{p}\bbeta_a^\top\bJ_a^{-1}\bar\bI_a\bJ_{a}^{-1, \top}\bbeta_{a},
\end{equation}
where $\bar\bI_a$ is defined in Theorem \ref{thm::asymptotic-normality-general-joint-A}. This is consistent with the definition in Appendix \ref{apdx::proof-lem::asymptotic-normality-true-G-general}, where $\bar\bI_{a}$ is defined for a single fixed arm $a$, by taking $\bg = \bg_a$ in the more general setting considered here.

Finally, combining (\ref{eq::conditional-var-linear-combination-1}) and (\ref{eq::conditional-var-linear-combination-2}), we obtain (\ref{eq::cond-var-general-linear-combination}).

\textbf{Check (\ref{eq::cond-lindeberg-general-linear-combination})}: We have 
\begin{align*}
&\frac1T\!\sum_{t\in[T]}\!\EE\!\left[\bar Z_t^21_{\{|\bar Z_t|>\sqrt{T}\delta\}}\Big|\cH_{t-1}\right]\\
&\leq \frac1T\sum_{t=1}^T\frac1{T\delta^2}\EE[\bar Z_t^4|\cH_{t-1}]\\
& = \frac1{T^2\delta^2}\sum_{t=1}^T\EE\bigg[\Big(\sum_{a\in\cA}\frac1{\pi_t(A_t)}1_{\{A_t = a\}}\bbeta_a^\top\bJ_a^{-1}\bg_a(\bX_t, Y_t;\btheta_a^*)\Big)^4\bigg|\cH_{t-1}\bigg]\\
& = \frac1{T^2\delta^2}\sum_{t=1}^T\EE\bigg[\sum_{a\in\cA}\frac1{\pi_t(A_t)^4}1_{\{A_t = a\}}\Big(\bbeta_a^\top\bJ_a^{-1}\bg_a(\bX_t, Y_t;\btheta_a^*)\Big)^4\bigg|\cH_{t-1}\bigg]\\
& \leq \frac1{T^2\delta^2}\sum_{t=1}^T\frac1{\pi_{\min}^4}\cdot\sum_{a\in\cA}\|\bJ_a^{-1, \top}\bbeta_a\|_2^4\cdot \EE\|\bg_a(\bX_t, Y_t;\btheta_a^*)\|_2^4\\
&\leq \frac1{T\delta^2\pi_{\min}^4}\max_a\|\bJ_a^{-1, \top}\bbeta_a\|_2^4\cdot \sum_{a\in\cA}\EE\|\bg_a(\bX_t, Y_t;\btheta_a^*)\|_2^4\rightarrow 0.
\end{align*}
Here the first inequality uses Chebyshev's Inequality. The second equality holds because all cross-product terms vanish when expanding the fourth power of the sum over $a\in\cA$. The last convergence uses Assumption \ref{aspt:boundedness-general-joint}, and that $\bJ_a$ is invertible for all $a\in\cA$.
\end{proof}

\subsection{Proof of Proposition \ref{thm::consistent-var-estimator-general}}\label{apdx::proof-thm::consistent-var-estimator-general}

We first prove that as $T\rightarrow \infty$,
\begin{equation}\label{eq::var-estimator-component-1-general}
\hat{\dot{\bG}}_{a, T}\xrightarrow{p}\EE\nabla\bg(\bX_t, Y_t(a);\btheta_a^*).
\end{equation}
From Lemma \ref{lem::convergence-true-derivative-general}, we deduce that 
$$
\frac1T\sum_{t=1}^T\frac1{\pi_t(A_t)}1_{\{A_t = a\}}\nabla\bg(\bX_t, Y_t(a);\btheta_a^*)\xrightarrow{p} \EE\nabla \bg(\bX_t, Y_t(a);\btheta_a^*).
$$
In addition, $\forall i$,
\begin{align*}
&\left\|\frac1T\sum_{t=1}^T\frac1{\pi_t(A_t)}1_{\{A_t = a\}}\nabla\bg^{(i)}(\bX_t, Y_t(a);\hat\btheta_a^{(T)}) - \frac1T\sum_{t=1}^T\frac1{\pi_t(A_t)}1_{\{A_t = a\}}\nabla\bg^{(i)}(\bX_t, Y_t(a);\btheta_a^*)\right\|_2\\
=&\frac1T\sum_{t=1}^T\frac1{\pi_t(A_t)}1_{\{A_t = a\}}\|\nabla\bg^{(i)}(\bX_t, Y_t(a);\hat\btheta_a^{(T)})-\nabla\bg^{(i)}(\bX_t, Y_t(a);\btheta_a^*)\|_2\\
\leq & \frac1{\pi_{\min}T}\sum_{t=1}^T \left\|\int_{0}^1\nabla^2\bg^{(i)}(\bX_t, Y_t(a);\btheta_a^* + u(\hat\btheta_a^{(T)} - \btheta_a^*))\mathrm{d}u\cdot (\hat\btheta_a^{(T)} - \btheta_a^*)\right\|_2\\
\leq & \frac1{\pi_{\min}T}\bigg[\sum_{t=1}^T \Phi(\bX_t, Y_t(a))\bigg]\cdot \|\hat\btheta_a^{(T)} - \btheta_a^*\|_2 +\\
&\quad 1_{\{\|\hat\btheta_a^{(T)}-\btheta_a^*\|_2>\epsilon_0\}}\cdot \frac1{\pi_{\min}T}\sum_{t=1}^T \left\|\int_{0}^1\nabla^2\bg^{(i)}(\bX_t, Y_t(a);\btheta_a^* + u(\hat\btheta_a^{(T)} - \btheta_a^*))\mathrm{d}u\cdot (\hat\btheta_a^{(T)} - \btheta_a^*)\right\|_2 \\
=& o_p(1).
\end{align*}
Here we have used Assumption \ref{aspt:min-sampling-prob} and \ref{aspt:smoothness-general}. Thus, 
$$
\frac1T\sum_{t=1}^T\frac1{\pi_t(A_t)}1_{\{A_t = a\}}\nabla\bg(\bX_t, Y_t(a);\hat\btheta_a^{(T)}) - \frac1T\sum_{t=1}^T\frac1{\pi_t(A_t)}1_{\{A_t = a\}}\nabla\bg(\bX_t, Y_t(a);\btheta_a^*) = o_p(1).
$$

Combining the above, we obtain (\ref{eq::var-estimator-component-1-general}).

Next we prove that as $T\rightarrow \infty$, 
\begin{equation}\label{eq::var-estimator-component-2-general}
    \hat\bI_{a, T}\xrightarrow{p}\bar\bI_a.
\end{equation}
Denote $\bZ_t = \frac{1}{\pi_t(A_t)}1_{\{A_t = a\}}\bg(\bX_t, Y_t; \btheta_a^*)$, and $\hat\bZ_t^{(T)} = \frac{1}{\pi_t(A_t)}1_{\{A_t = a\}}\bg(\bX_t, Y_t; \hat \btheta_a^{(T)})$. Then (\ref{eq::var-estimator-component-2-general}) is equivalent to 
$$
\frac1T\sum_{t=1}^T\hat\bZ_t^{(T)}\hat\bZ_t^{(T), \top}\xrightarrow{p} \bar \bI_a,
$$
and in order to show this, it suffices to prove 
\begin{align}
    \frac1T\sum_{t=1}^T \hat\bZ_t^{(T)}\hat\bZ_t^{(T), \top}  - \frac1T\sum_{t=1}^T \bZ_t\bZ_t^{\top} &= o_p(1),\label{eq::var-estimator-component-2-1-general}\\
    \frac1T\sum_{t=1}^T \Big(\bZ_t\bZ_t^{\top} - \EE[\bZ_t\bZ_t^{\top}|\cH_{t-1}]\Big) &= o_p(1),\label{eq::var-estimator-component-2-2-general} \\
    \frac1T\sum_{t=1}^T \EE[\bZ_t\bZ_t^{\top}|\cH_{t-1}]&\xrightarrow{p}\bar\bI_a.\label{eq::var-estimator-component-2-3-general}
\end{align}
Below we check these facts one by one.

\textbf{Check (\ref{eq::var-estimator-component-2-1-general}):} For convenience, we define $\bg_t^* = \bg(\bX_t, Y_t(a);\btheta_a^*)$, and $\hat\bg_t^{(T)} = \bg(\bX_t, Y_t(a);\hat\btheta_a^{(T)})$. Then 
\begin{align}
&\left\|\frac1T\sum_{t=1}^T\big(\hat\bZ_t^{(T)}\hat\bZ_t^{(T), \top} - \bZ_t\bZ_t^{\top}\big)\right\|_2\nonumber\\
=& \left\|\frac1T\sum_{t=1}^T\frac{1}{\pi_t(A_t)^2}1_{\{A_t = a\}}\big(\hat\bg_t^{(T)}\hat\bg_t^{(T), \top} - \bg_t^*\bg_t^{*, \top}\big)\right\|_2\nonumber\\
\leq &\frac1T\cdot\frac1{\pi_{\min}^2}\sum_{t=1}^T\|\hat\bg_t^{(T)}\hat\bg_t^{(T), \top} - \bg_t^*\bg_t^{*, \top}\|_2\nonumber\\
\leq& \frac1{\pi_{\min}^2T}\sum_{t=1}^T \big(\|\hat\bg_t^{(T)}\hat\bg_t^{(T), \top} - \hat\bg_t^{(T)}\bg_t^{*, \top}\|_2 + \|\hat\bg_t^{(T)}\bg_t^{*, \top} - \bg_t^*\bg_t^{*, \top}\|_2\big)\nonumber\\
\leq& \frac1{\pi_{\min}^2T}\sum_{t=1}^T \|\hat\bg_t^{(T)} - \bg_t^*\|_2\cdot \big(\|\hat\bg_t^{(T)}\|_2 + \|\bg_t^*\|_2\big)\nonumber\\
\leq& \frac1{\pi_{\min}^2T}\sum_{t=1}^T \|\hat\bg_t^{(T)} - \bg_t^*\|_2\cdot \big(\|\hat\bg_t^{(T)} - \bg_t^*\|_2 + 2\|\bg_t^*\|_2\big)\nonumber\\
\leq& \frac1{\pi_{\min}^2T}\sum_{t=1}^T \phi(\bX_t, Y_t(a))\|\hat\btheta_a^{(T)} - \btheta_a^*\|_2\Big[\phi(\bX_t, Y_t(a))\|\hat\btheta_a^{(T)} - \btheta_a^*\|_2 + 2\|\bg_t^*\|_2\Big]\nonumber\\
= &\frac1{\pi_{\min}^2T}\left[\sum_{t=1}^T\phi(\bX_t, Y_t(a))^2\right]\cdot\|\hat\btheta_a^{(T)} - \btheta_a^*\|_2^2 +  \nonumber\\
&\quad\quad \frac2{\pi_{\min}^2T}\left[\sum_{t=1}^T\phi(\bX_t, Y_t(a))\|\bg(\bX_t, Y_t(a);\btheta_a^*)\|_2\right]\cdot\|\hat\btheta_a^{(T)} - \btheta_a^*\|_2\nonumber\\
\leq &\frac1{\pi_{\min}^2T}\left[\sum_{t=1}^T\phi(\bX_t, Y_t(a))^2\right]\cdot\|\hat\btheta_a^{(T)} - \btheta_a^*\|_2^2 +  \nonumber\\
&\quad \frac2{\pi_{\min}^2}\sqrt{\frac1T\sum_{t=1}^T\phi(\bX_t, Y_t(a))^2\cdot \frac1T\sum_{t=1}^T\|\bg(\bX_t, Y_t(a);\btheta_a^*)\|_2^2}\cdot \|\hat\btheta_a^{(T)} - \btheta_a^*\|_2\nonumber\\
 =& o_p(1).\label{eq::var-estimator-component-2-1-1-general}
\end{align}
Here, the fifth inequality is due to Assumption \ref{aspt:smoothness-general}. The last inequality uses Cauchy–Schwarz inequality. The final equality is because of $\frac1T\sum_{t=1}^T\phi(\bX_t, Y_t(a))^2 = \cO_p(1)$ and \\
$\frac1T\sum_{t=1}^T\|\bg(\bX_t, Y_t(a);\btheta_a^*)\|_2^2= \cO_p(1)$ (from the law of large numbers), and $\|\hat\btheta_a^{(T)} - \btheta_a^*\|_2 = o_p(1)$ (from Lemma \ref{lem::consistency-general}).

\textbf{Check (\ref{eq::var-estimator-component-2-2-general}):} $\forall \bc\in\RR^d$, $\forall \delta>0$,
\begin{align}
& \PP\left(\bigg|\frac1T\sum_{t=1}^T[\bc^\top\bZ_t\bZ_t^\top\bc - \EE[\bc^\top\bZ_t\bZ_t^\top\bc|\cH_{t-1}]]\bigg|>\delta\right)\nonumber\\
\leq & \frac1{\delta^2T^2}\EE\left(\sum_{t=1}^T[\bc^\top\bZ_t\bZ_t^\top\bc - \EE[\bc^\top\bZ_t\bZ_t^\top\bc|\cH_{t-1}]]\right)^2\nonumber\\
 = & \frac1{\delta^2T^2}\sum_{t=1}^T\EE\left(\bc^\top\bZ_t\bZ_t^\top\bc - \EE[\bc^\top\bZ_t\bZ_t^\top\bc|\cH_{t-1}]\right)^2\nonumber\\
 \leq & \frac1{\delta^2T^2}\sum_{t=1}^T\EE\left(\bc^\top\bZ_t\bZ_t^\top\bc\right)^2\nonumber\\
 \leq & \frac1{\delta^2T^2}\sum_{t=1}^T\frac{1}{\pi_{\min}^2}\EE\left(\bc^\top\bg(\bX_t, Y_t(a);\btheta_a^*) \right)^4\rightarrow 0.\nonumber%\label{eq::var-estimator-component-2-2-1-general}
\end{align}
Here the first inequality is because of Chebyshev's Inequality. The second equality is due to the following fact: Let $v_{t}' = \bc^\top\bZ_t\bZ_t^\top\bc$. Then for $t_1<t_2$,
\begin{align*}
   & \EE (v_{t_1}' - \EE[v_{t_1}'|\cH_{t_1-1}])(v_{t_2}' - \EE[v_{t_2}'|\cH_{t_2-1}]) \\
   = & \EE\big[ \EE[(v_{t_1}' - \EE[v_{t_1}'|\cH_{t_1-1}])(v_{t_2}' - \EE[v_{t_2}'|\cH_{t_2-1}])\big|\cH_{t_1-1}]\big]\\
   = & \EE\big[(v_{t_1}' - \EE[v_{t_1}'|\cH_{t_1-1}])\cdot \EE[v_{t_2}' - \EE[v_{t_2}'|\cH_{t_2-1}]\big|\cH_{t_1-1}]\big]\\
   = & \EE\big[ (v_{t_1}' - \EE[v_{t_1}'|\cH_{t_1-1}])\cdot 0\big] = 0.
\end{align*}
The last convergence uses Assumption \ref{aspt:boundedness-general}. 

\textbf{Check (\ref{eq::var-estimator-component-2-3-general}):} In fact, (\ref{eq::var-estimator-component-2-3-general}) is a direct consequence of (\ref{eq::cond-var-2-general}), (\ref{eq::l1-convergence-1-general}), and (\ref{eq::l1-convergence-2-general}).

    \subsection{Proof of Theorem \ref{lem::statistics-converge-implies-policy-converge}}\label{apdx::proof-lem::statistics-converge-implies-policy-converge}
    
    Fix any $a\in\cA$. Denote $f_{\bx}(\bbeta):=\pi(a|\bx, \bbeta)$. Then $\forall \epsilon, \delta>0$, 
\begin{align}
    &\PP(|\pi_t(a|\bX_t, \cH_{t-1}) - \bar\pi(a|\bX_t)|>\epsilon)\\
& = \PP(|\pi_t(a|\bX_t, \bbeta_{t-1}) - \pi(a|\bX_t, \bbeta^*)|>\epsilon)\nonumber\\
& = \PP(|f_{\bX_t}(\hat\bbeta_{t-1}) - f_{\bX_t}(\bbeta^*)|>\epsilon)\nonumber\\
& \leq \PP(\|\hat\bbeta_{t-1}-\bbeta^*\|_2\geq \delta) + \PP(\bbeta^*\in \cD(f_{\bX_t})) + \PP(\bbeta^*\in\cB_{\epsilon, \delta}(\bX_t)),\label{eq::policy-converge-decomposition}
\end{align}
Here $\forall \bx$, we define 
\begin{align*}
\cD(f_{\bx})&:= \{\bbeta: f_{\bx}\text{ is discontinuous at }\bbeta\},\\
\cB_{\epsilon, \delta}(\bx)& := \{\bbeta\notin \cD(f_{\bx}): \exists \bbeta',\text{ s.t.} \|\bbeta' - \bbeta\|_2<\delta, |f_{\bx}(\bbeta) - f_{\bx}(\bbeta')|>\epsilon\}.
\end{align*}
Notice that:
\begin{itemize}
    \item From condition (i), for the first term on the RHS of (\ref{eq::policy-converge-decomposition}), we have $\limsup_{t\rightarrow \infty} \PP(\|\hat\bbeta_{t-1}-\bbeta^*\|_2\geq \delta) = 0$,
    \item From condition (ii), for the second term on the RHS of (\ref{eq::policy-converge-decomposition}), we have $\PP(\bbeta^*\in \cD(f_{\bX_t})) = 0$ $\forall t$,
    \item Due to $\bX_t$ are i.i.d., the last two terms on the RHS of (\ref{eq::policy-converge-decomposition}) does not depend on time $t$.
\end{itemize}
Plugging in the above into (\ref{eq::policy-converge-decomposition}), we obtain that $\forall \epsilon, \delta>0$,
\begin{equation}\label{eq::policy-converge-decomposition-limit}
    \limsup_{t\rightarrow \infty}\PP(|\pi_t(a|\bX_t, \cH_{t-1}) - \bar\pi(a|\bX_t)|>\epsilon)\leq \PP(\bbeta^*\in\cB_{\epsilon, \delta}(\bX_t)).
\end{equation}
Finally, by definition of continuity, we have that $\forall \bx$, 
$$
\lim_{\delta\rightarrow 0} 1_{\{\bbeta^*\in\cB_{\epsilon, \delta}(\bx)\}} = 0.
$$
From the dominated convergence theorem, 
$$
\lim_{\delta\rightarrow 0}\PP(\bbeta^*\in \cB_{\epsilon, \delta}(\bX_t)) = \lim_{\delta\rightarrow 0}\EE1_{\{\bbeta^*\in \cB_{\epsilon, \delta}(\bX_t)\}} = 0.
$$
We obtain the desired result by plugging the above into (\ref{eq::policy-converge-decomposition-limit}).

\subsection{Proof of Proposition \ref{lem::convergence-mab}}\label{apdx::proof-lem::convergence-mab}

We first present the following lemma. Its proof is in Appendix \ref{apdx::proof-lem::policy-with-no-contexts}.

\begin{lemma}\label{lem::policy-with-no-contexts}
Consider a setting where the behavior policy does not use the contexts, i.e.,
$$
\pi = \{
\pi_t(\cdot|\cH_{t-1}^0)\}_{t\geq 1},
$$ 
where $\cH_t^0:= \sigma(\{A_\tau, Y_\tau\}_{\tau = 1}^t)$ for $t\geq 1$. Assume that the potential outcomes have uniformly bounded variance, i.e., $\sup_{a\in\cA}\mathrm{Var}(Y(a))\leq \sigma_Y^2$ for some constant $\sigma_Y^2$. Suppose Assumption \ref{aspt:min-sampling-prob} is satisfied for all arm $a\in\cA$, and we assume unconfoundedness: $\forall t\in[T]$,
\begin{equation}\label{eq::unconfoundedness-no-context}
A_t\perp \{Y_t(a)\}_{a\in\cA}|\cH_{t-1}^0.
\end{equation}
Then as $t\rightarrow \infty$,
\begin{itemize}
    \item[(i)] For any deterministic sequence $\{C_t\}_{t\geq 1}$ such that $\lim_{t\rightarrow\infty}\frac{C_t}{t} = 0$,
    $\frac{C_t}{N_{a, t}}\xrightarrow{p} 0$ for any $a\in\cA$;
    \item[(ii)] $\hat \mu_{a, t}\xrightarrow{p} \mu_a^*$  for any $a\in\cA$.
\end{itemize}
\end{lemma}

Below we analyze the three common multi-arm bandits algorithms that ensure a minimum sampling probability.

\textbf{The $\epsilon$-greedy algorithm.} Consider the $\epsilon$-greedy policy defined by (\ref{eq::policy-eps-greedy-mab}). Define 
\begin{align*}
\pi^{(1)}_t(a|\{\hat\mu_{i, t-1}\}_{i\in\cA}) := \pi_t^{\epsilon\text{-greedy}}(a|\mathcal H_{t-1})=
\begin{cases}
    1-\frac{K-1}{K}\epsilon,\quad &\text{if }i = \argmax_i \hat\mu_{i, t-1},\\
    \frac1K\epsilon, \quad &\text{otherwise.}
\end{cases}
\end{align*}
Then $\pi^{(1)}$ satisfies Assumption \ref{aspt:min-sampling-prob} with $\pi_{\min} = \epsilon/K$. Under the setting of Proposition \ref{lem::convergence-mab}, the bounded variance condition and the unconfoundedness condition in Lemma \ref{lem::policy-with-no-contexts} is also satisfied. Thus, according to part (i) of Lemma \ref{lem::policy-with-no-contexts}, condition (i) of Theorem \ref{lem::statistics-converge-implies-policy-converge} is satisfied with $\bbeta_t = (\hat\mu_{a, t})_{a\in\cA}$, $\bbeta^* = (\mu^*_{a})_{a\in\cA}$. In addition, condition (ii) of Theorem \ref{lem::statistics-converge-implies-policy-converge} holds for $\pi^{(1)}$ as long as $\Delta>0$. From Theorem \ref{lem::statistics-converge-implies-policy-converge}, we deduce that the policy convergence condition as in Definition \ref{aspt:policy_convergence} holds.

\textbf{The UCB algorithm.} Consider the clipped UCB algorithm defined by (\ref{eq::policy-ucb-mab}). Let 
\begin{align*}
\pi^{(2)}(a|\{\hat\mu_{i, t-1}, C_t/N_{i, t-1}\}_{i\in\cA}) &:= \pi_t^{\text{UCB}}(a|\mathcal H_{t-1})\\
&=
\begin{cases}
    1-(K-1)\pi_{\min},\quad &\text{if }i = \argmax_i \left\{\hat\mu_{i, t-1} + \sqrt{\frac{C_t}{N_{i, t-1}}}\right\},\\
    \pi_{\min}, \quad &\text{otherwise.}
\end{cases}
\end{align*}
Then $\pi^{(2)}$ satisfies Assumption \ref{aspt:min-sampling-prob}, and similar to the previous case, the bounded variance condition and the unconfoundedness condition in Lemma \ref{lem::policy-with-no-contexts} also hold. According to part (i) and (ii) of Lemma \ref{lem::policy-with-no-contexts}, condition (i) of Theorem \ref{lem::statistics-converge-implies-policy-converge} is satisfied with $\bbeta_t = (\hat\mu_{a, t}, C_{t+1}/N_{a, t})_{a\in\cA}$, $\bbeta^* = (\mu^*_{a}, 0)_{a\in\cA}$. Also, condition (ii) of Theorem \ref{lem::statistics-converge-implies-policy-converge} holds for $\pi^{(2)}$ as long as $\Delta>0$. From Theorem \ref{lem::statistics-converge-implies-policy-converge}, we deduce that the policy convergence condition as in Definition \ref{aspt:policy_convergence} holds.

\textbf{The TS algorithm.} Define 
\begin{align*}
\bar \pi^{(3)}(a|\bmu_{t-1}^{\mathrm{post}}, \bSigma_{t-1}^{\mathrm{post}}) =\bar\pi_t^{\text{TS}}(a|\mathcal H_{t-1}),
\end{align*}
where the dependence of $\bar \pi^{(3)}$ on $\bmu_{t-1}^{\mathrm{post}}$ and $\bSigma_{t-1}^{\mathrm{post}}$ is shown in Section \ref{sec::MAB-policy-convergence}. In addition, define 
$$
\pi^{(3)}(a|\bmu_{t-1}^{\mathrm{post}}, \bSigma_{t-1}^{\mathrm{post}})  = \left[\mathrm{Clip} \big(\bar \pi^{(3)}(i|\bmu_{t-1}^{\mathrm{post}}, \bSigma_{t-1}^{\mathrm{post}})_{i\in\cA}\big)\right]^{(a)},
$$
Here for a vector $\bv\in\RR^K$, $[\bv]^{(a)}$ denotes its $a$-th entry. From (\ref{eq::policy-ts-mab}), we deduce that 
$$
\pi^{(3)}(a|\bmu_{t-1}^{\mathrm{post}}, \bSigma_{t-1}^{\mathrm{post}}) = \pi_t^{\mathrm{TS}}(a|\cH_{t-1}).
$$
We now show that conditions (i) and (ii) of Theorem \ref{lem::statistics-converge-implies-policy-converge} holds with $\pi^{(3)}$ as the behavior policy and with statistics $\bbeta_t = (\bmu_{t}^{\mathrm{post}}, \bSigma_{t}^{\mathrm{post}})$. First, $\pi^{(3)}$ satisfies Assumption \ref{aspt:min-sampling-prob}, the bounded variance condition and the unconfoundedness condition in Lemma \ref{lem::policy-with-no-contexts}. From Lemma \ref{lem::policy-with-no-contexts}, we obtain that $\forall a\in\cA$, $1/N_{a, t}\xrightarrow{p} 0$ and $\hat\mu_{a, t}\xrightarrow{p} \mu_a^*$. Combining these results together with (\ref{eq::ts-mab-posterior-mean-var}) and the continuous mapping theorem, we deduce that $\forall a\in\cA$, 
$$
\mu_{a, t}^{\mathrm{post}}\xrightarrow{p} \mu_a^*,\quad \sigma_{a, t}^{\mathrm{post}}\xrightarrow{p} 0.
$$
Note that entrywise convergence in probability implies joint convergence in probability in a finite dimensional Euclidean space, and thus from (\ref{eq::ts-mab-posterior-mean-var-joint}) we have 
$$
\bmu_{t}^{\mathrm{post}}\xrightarrow{p} \bmu^*,\quad \bSigma_{ t}^{\mathrm{post}}\xrightarrow{p} \mathbf{0}_{K\times K},
$$
where $\bmu^* = (\mu_a^*)_{a\in\cA}$. This implies condition (i) of Theorem \ref{lem::statistics-converge-implies-policy-converge} holds with the statistics $\bbeta_t$ and limit $\bbeta^* = (\bmu^*, \mathbf{0}_{K\times K})$.

In addition, it is not difficult to verify that $\bar\pi^{(3)}$ is continuous at $\bbeta^*$ as long as the suboptimality gap $\Delta>0$. Also,  $\mathrm{Clip}$ is a continuous mapping (see Lemma \ref{lem::clipping-Lipschitz}). Therefore, the composite mapping $\pi^{(3)}$ is continuous at $\bbeta^*$, and condition (ii) of Theorem \ref{lem::statistics-converge-implies-policy-converge} holds.

In summary, we have verified both conditions (i) and (ii) of Theorem \ref{lem::statistics-converge-implies-policy-converge}. We deduce from the theorem that the policy convergence condition in Definition \ref{aspt:policy_convergence} holds.

\subsection{Proof of Proposition \ref{thm::convergence-of-IPWZ-policy}}\label{apdx::proof-thm::convergence-of-IPWZ-policy}
Let $\hat\bbeta_t = (\hat\btheta^{(t)}, \bgamma_t)$, where $\hat\btheta^{(t)} = (\hat\btheta_a^{(t)})_{a\in\cA}$. Using the same arguments in the proof of Lemma \ref{lem::consistency-general}, we deduce that $\hat\bbeta_{t}\xrightarrow{p} \bbeta^*:= (\btheta^*, \bgamma^*)$ (note that the proof of Lemma \ref{lem::consistency-general} does not require the behavior policy to converge), which implies that condition (i) of Theorem \ref{lem::statistics-converge-implies-policy-converge} holds in this setting. Under an additional condition of policy continuity, Theorem \ref{lem::statistics-converge-implies-policy-converge} can then be invoked to establish policy convergence.
%\ziping{Theorem 3.3 requires policy convergence, which is what we are proving here...}  

\subsection{Proof of Proposition \ref{theorem:convergence-of-Ridge}}\label{apdx::proof-thm::convergence-of-Ridge}

We adopt the convergence analysis of the recursive least square fitting based on the stochastic approximation theory. We first rewrite the ridge regression estimator in (\ref{eq:ridge-estimator}) in a recursive form. Define $\bX_{t, a} = \bX_t 1_{\{A_t = a\}}$ and $Y_{t, a} = Y_t 1_{\{A_t = a\}}$. Let the sample covariance matrix and the sample cross product matrix be 
\begin{equation}
    \bPhi_{t, a} = \frac{1}{t-1} \left(\lambda I + \sum_{i=1}^{t-1} \bX_{i, a} \bX_{i, a}^\top\right), \quad \text{and} \quad \bvarphi_{t, a} = \frac{1}{t-1} \left(\sum_{i=1}^{t-1} \bX_{i, a} Y_{i, a}\right). \label{eq:sample-cov}
\end{equation}

% \begin{lemma}[Matrix Inverse Lemma]
%     \begin{equation}
%         [A + BCD]^{-1} = A^{-1} - A^{-1}B[C^{-1} + DA^{-1}B]^{-1}DA^{-1}.
%     \end{equation}
% \end{lemma}

% Then, based on the matrix inverse lemma, the OLS estimator can be rewritten as 
Based on the definition of ridge regression estimator, we directly have that 
\begin{align}
    \hat \bbeta_{t, a}^{\Ridge} 
    &= \left(\lambda I + \sum_{i=1}^{t-1} \bX_{i, a} \bX_{i, a}^\top\right)^{-1} \left(\sum_{i=1}^{t-1} \bX_{i, a} Y_i\right) = \bPhi_{t, a}^{-1} \bvarphi_{t, a}.
\end{align}
Thus, it is sufficient to show that both $\bPhi_{t, a}$ and $\bvarphi_{t, a}$ converge in probability as $t\to\infty$. We show this by writing the recursive form for $\bPhi_{t, a}$ and $\bvarphi_{t, a}$.
Denote by $\Vec(M)$ the vectorized form of a matrix $M$. We have that 
\begin{align}
    \Vec(\bPhi_{t+1,a}) &= \Vec(\bPhi_{t,a}) + \frac{1}{t} \Vec\left(\bX_{t,a}\bX_{t,a}^\top - \bPhi_{t,a}\right)\\
    \bvarphi_{t+1,a} &= \bvarphi_{t,a} + \frac{1}{t} (\bX_{t,a} Y_{t,a} - \bvarphi_{t,a}).
\end{align}

We may write the recursive form for the joint vector of $\Vec(\Phi_{t,a})$ and $\varphi_{t,a}$ as 
\begin{align}
    \begin{pmatrix}
        \bvarphi_{t+1,a} \\ \Vec(\bPhi_{t+1,a})
    \end{pmatrix} = \begin{pmatrix}
        \bvarphi_{t,a} \\ \Vec(\bPhi_{t,a})
    \end{pmatrix} + &\frac{1}{t} 
    \begin{pmatrix}
        \bX_{t, a} Y_{t,a} - \bvarphi_{t,a}\\
        \Vec\left(\bX_{t, a}\bX_{t, a}^\top - \bPhi_{t,a}\right)
    \end{pmatrix}  
\end{align}

We denote by $\bm{\psi}_{t, a} = \begin{pmatrix}
    \bvarphi_{t, a} \\ \Vec(\bPhi_{t, a})
\end{pmatrix}$ and $\bm{\psi}_{t} = \begin{pmatrix}
    \bm{\psi}_{t, 1}^\top, \bm{\psi}_{t, 2}^\top, \ldots, \bm{\psi}_{t, |\cA|}^\top
\end{pmatrix}^\top$. Concatenating the above recursive form for all $a \in \cA$ into a single expression, we have
\begin{align}
    \bm{\psi}_{t+1} = \bm{\psi}_{t} + w_t \left(\bm{g}(\bm{\psi}_{t}) - \bm{\psi}_{t} + \bm{M}_t \right),
\end{align}
where for any $\bm{\psi} = (\bvarphi_{1}^{\top}, \Vec(\bPhi_{1})^{\top}, \ldots, \bvarphi_{|\cA|}^{\top}, \Vec(\bPhi_{|\cA|})^{\top})^\top \in \mathbb{R}^{d+d^2|\cA|}$, we have
\begin{align}
    w_t &= \frac{1}{t}, \\
    \bm{g}\left(\bm{\psi}\right) &= \begin{pmatrix}
        g_1(\bm{\psi})^\top, g_2(\bm{\psi})^\top, \ldots, g_{|\cA|}(\bm{\psi})^\top
    \end{pmatrix}^\top, \\
    \bg_a(\bm{\psi}) &= \begin{pmatrix}
        \EE\left[ \bX_{t, a} Y_{t,a} \mid \bvarphi_{t,a} = \bvarphi_{a}, \Vec(\bPhi_{t,a}) = \Vec(\bPhi_{a}), \forall a \in \cA\right] \\ 
        \EE\left[ \bX_{t, a} \bX_{t, a}^\top \mid \bvarphi_{t,a} = \bvarphi_{a}, \Vec(\bPhi_{t,a}) = \Vec(\bPhi_{a}), \forall a \in \cA\right]
    \end{pmatrix}, \\
    \bm{M}_t &= \begin{pmatrix}
        \bM_{t, 1}^\top, \bM_{t, 2}^\top, \ldots, \bM_{t, |\cA|}^\top
    \end{pmatrix}^\top, \\
    \bM_{t, a} &= 
    \begin{pmatrix}
        \bX_{t, a} Y_{t,a} \\ 
        \Vec(\bX_{t, a}\bX_{t, a}^\top)
    \end{pmatrix} - \bg_a(\bm{\psi}_{t})
\end{align}

Typical convergence analysis requires the following three conditions:
\begin{itemize}
    \item (C.1) $\sup_t \EE[\|\bm{\psi}_t\|^2_2] < \infty$
    \item (C.2) Contraction mapping condition: $\|\bg(\bm{\psi}_1) - \bg(\bm{\psi}_2) \| \leq \kappa \|\bm{\psi}_1 - \bm{\psi}_2 \|$ for some $\kappa < 1$
    \item (C.3) $\{\bm{M}_t\}$ is a martingale difference sequence with respect to the increasing family of the $\sigma$-field
    $$
        \mathcal{F}_t = \sigma(A_1, \bX_1, Y_1, \ldots, A_{t}, \bX_t, Y_{t}).
    $$
    That is $\EE[\bm{M}_t \mid \mathcal{F}_{t-1}] = \mathbf{0}$. Furthermore, $\{\bm{M}_t\}$ are square integrable, i.e., $\EE[\|\bm{M}_t\|^2_2 \mid \mathcal{F}_{t-1}] < K (1+ \|\bm{\psi}_t\|^2_2)$ a.s., $t \geq 1$ for some constant $K$
    % \item (C.4) $\sum_{t=1}^{\infty} w_t \|\bm{\beta}_t\|_2 < \infty$ a.s.
\end{itemize}

For sufficiently large $t$, we show that the above three conditions hold for all $\tau = t, t + 1, \ldots$ under some good event $\cE_{t}$ with a high probability $1-2/(t-1)$. The probability is chosen so when $t$ goes infinity, the failure probability of $\cE_{t}$ goes to 0.

% The key is to show (C.2) about the contraction property of $g$ and (C.4) that the bias is asymptotically negligible.

\subsubsection{Good events}

The convergence analysis is based on the following good event $\mathcal{E}_t$ being satisfied with a high probability $1-2/(t-1)$. We define $\mathcal{E}_t = \mathcal{E}_{1, t} \cap \mathcal{E}_{2, t}$ where
\begin{align}
    \mathcal{E}_{1, t} = &\left\{\|\bPhi_{\tau, a}\|_2 \geq \pi_{\min} \sigma_{\min}^2 - M \sqrt{\frac{16\log(\tau Ad)}{\tau-1}}, \quad \text{ for all } \tau \geq t \text{ and } a \in \cA \right\} \label{eq:good-event-1}
\end{align}
% \begin{align}
%     \text{ and } \quad \mathcal{E}_{2, t} &= \left\{ \left\|\hat \btheta_{i, a}^{\OLS}\right\|_2 \leq \frac{R_{\btheta}M^2 + \sigma_{\eta} M \sqrt{d\log(10 At \log(t)) / (i-1)}}{\max\{0, \frac{i-1}{i}\pi_{\min}\sigma_{\min}^2 - M \sqrt{8\log(Adt\log(t)) / (i-1)}\}}, \quad \forall a \in \cA, \forall i \in [t] \right\}.
% \end{align}

\begin{align}
    \mathcal{E}_{2, t} = &\left\{\|\bm{\varphi}_{\tau, a}\|_2 \leq M R_{f} + 4M\sigma_{\eta}\sqrt{dA}, \quad \text{ for all } \tau \geq t \text{ and } a \in \cA\right\} \label{eq:good-event-2}
\end{align}
% \ziping{M}

Here we let $1/0 = \infty$.

\begin{lemma}
    \label{lem::good-event}
    For any $t = 2, \ldots$, we have
    \begin{align}
        \PP\left(\{\mathcal{E}_{t} \text{ holds}\}\right) \geq 1 - 2/(t-1).
    \end{align}
\end{lemma}

\begin{proof}

    Consider the martingale difference sequence $\{\bX_{t, a}\bX_{t, a}^{\top} - \EE[\bX_{t, a}\bX_{t, a}^{\top} \mid \cF_{t-1}]\}_{t}$, where $\cF_{t-1}$ is the $\sigma$-field generated by $\left\{\left\{\bX_{\tau, a}, A_{\tau, a}, Y_{\tau, a}\right\}_{\tau = 1}^{t-1}\right\}_{a \in \cA}$. Based on Assumption \ref{aspt:Ridge_Conv}, the martingale difference sequence satisfies
    \begin{align}
        \bX_{t, a}\bX_{t, a}^{\top} - \EE[\bX_{t, a}\bX_{t, a}^{\top} \mid \cF_{t-1}] \preceq M^2 \bI.
    \end{align}
    
    We first apply Corollary \ref{cor::matrix-azuma} for all $a \in \cA$ with a probability of at least $1 - 1/(A(t-1))$ to get the following concentration bound simultaneously for all $a \in \cA$ and $\tau \geq t$:
    \begin{align}
        \frac{1}{\tau-1} \left\|\sum_{k = 1}^{\tau-1} \left(\bX_{k,a}\bX_{k,a}^{\top} - \EE[\bX_{k,a}\bX_{k,a}^{\top} \mid \cF_{k-1}]\right)\right\|_2 \leq M\sqrt{\frac{16\log(A\tau d)}{\tau-1}}.
    \end{align}
    Because the sampling probability is clipped below by $\pi_{\min}$, we have that for all $k \in \mathbb{N}$ and $a \in \cA$,
    \begin{align}
        \pi_{\min} \sigma_{\min}^2 \bI \preceq \EE[\bX_{k,a}\bX_{k,a}^{\top} \mid \cF_{k-1}].
    \end{align}

    Therefore, with a probability of at least $1 - 1/(t-1)$, for all $a \in \cA$ and $\tau \geq t$,
    \begin{align}
        \pi_{\min} \sigma_{\min}^2 - M\sqrt{\frac{16\log(A\tau d)}{\tau-1}} \leq \frac{1}{\tau-1} \left\|\sum_{k = 1}^{\tau-1} \bX_{k,a}\bX_{k,a}^{\top}\right\|_2.
        %  \leq \sigma_{\max}^2 + M\sqrt{\frac{16\log(A\tau d)}{\tau-1}}.
    \end{align}

    Thus, we have with a probability of at least $1 - 1/(t-1)$, for all $a \in \cA$ and $\tau \geq t$,
    \begin{align}
        \pi_{\min} \sigma_{\min}^2 - M\sqrt{\frac{16\log(A\tau d)}{\tau-1}} \leq \|\bPhi_{\tau,a}\|_2. 
    \end{align}

    % This implies that with a probability of at least $1 - 1/(t-1)$, for all $a \in \cA$ and $\tau \geq t$,
    % \begin{align}
    %     \frac{\tau-1}{\tau}\pi_{\min} \sigma_{\min}^2 - M\sqrt{\frac{8\log(2A\tau d\log(t))}{\tau-1}} \leq \|\Phi_{\tau,a}\|_2 \leq \lambda / \tau + p_{\max} \sigma_{\max}^2 + M\sqrt{\frac{8\log(2A\tau d\log(t))}{\tau-1}}.
    % \end{align}

    This is our definition of the event $\mathcal{E}_{1, t}$.
    
    We further show that $\mathcal{E}_{2, t}$ holds with a probability of at least $1 - 1/(t-1)$. For any $a \in \cA$, and $\tau \geq t$, we have that 

    \begin{align}
        \left\|\varphi_{\tau, a}\right\|_2 
        &= \frac{1}{\tau-1} \left\|\sum_{k = 1}^{\tau-1} \bX_{k,a} Y_{k,a}\right\|_2 \\
        &= \frac{1}{\tau-1} \left\|\sum_{k = 1}^{\tau-1} \bX_{k,a} (f(\bX_k, a) + \eta_k)\right\|_2 \\
        &= \frac{1}{\tau-1} \left\|\sum_{k = 1}^{\tau-1} \bX_{k,a} f(\bX_k, a) + \sum_{k = 1}^{\tau-1} \bX_{k,a} \eta_k\right\|_2 \\
        &\leq \frac{1}{\tau-1} \left\|\sum_{k = 1}^{\tau-1} \bX_{k,a}f(\bX_k, a)\right\|_2 + \frac{1}{\tau-1} \left\|\sum_{k = 1}^{\tau-1} \bX_{k,a} \eta_k\right\|_2 \\
        &\leq M R_{f} + \frac{1}{\tau-1} \left\|\sum_{k = 1}^{\tau-1} \bX_{k,a} \eta_k\right\|_2, \label{eq::varphi-bound}
    \end{align}
    where $\eta_k \coloneqq Y_{t, a} - f(\bX_t, a)$ and $f(\bX_t, a) \coloneqq \EE[Y_{t, a} \mid \bX_t, A_t = a]$.

    By applying Kolmogorov's inequality to each coordinate of $X_{t, a} \eta_t$, we have that 
    \begin{align}
        \PP\left(\max_{1 \leq m \leq n} \left|\sum_{k = 1}^m (\bX_{t, a} \eta_t)_i\right| \geq z \right) \leq \Var\left(\sum_{k = 1}^n (\bX_{t, a} \eta_t)_i \right) / z^2 \leq n M^2 \sigma_{\eta}^2 / z^2.
    \end{align}

    Now we use a peeling argument to obtain a uniform bound for all $\tau \geq t$. For a block index $m = 0, 1, 2, \ldots$, and consider the block of indices:
    \begin{align}
        \sB_m = \{n: 2^m t-1 \leq n \leq 2^{m+1} t-1\}.
    \end{align}
    Define the bad event on this block for coordinate $i$ as:
    \begin{align}
        A_{m, i}(\varepsilon) \coloneqq \left\{\exists n \in \sB_m: \left|\frac{1}{n-1} \sum_{k = 1}^{n-1} (\bX_{k, a} \eta_k)_i\right|_2 > \varepsilon\right\}.
    \end{align}
    If $A_{m, i}(\varepsilon)$ occurs, then there exists $n^* \in \sB_m$ such that
    \begin{align}
        \left|\sum_{k = 1}^{n^*-1} (\bX_{k, a} \eta_{k})_i\right| > (n^*-1)\varepsilon \geq (2^m t - 1) \varepsilon.
    \end{align}
    Therefore, we have that
    \begin{align}
        A_{m, i}(\varepsilon) \subseteq \left\{ \max_{1 \leq n \leq 2^{m+1} t-1} \left|\sum_{k = 1}^{n-1} (\bX_{k, a} \eta_{k})_i\right| > (2^m t - 1) \varepsilon\right\}.
    \end{align}

    By applying Kolmogorov's inequality on coordinate $i$, we have that
    \begin{align}
        \PP\left( \max_{1 \leq n \leq N} \left|\sum_{k = 1}^{n} (\bX_{k, a} \eta_k)_i \right| \geq z\right) \leq \Var\left(\sum_{k = 1}^{N} (\bX_{k, a} \eta_k)_i \right) / z^2 \leq N M^2 \sigma_{\eta}^2 / z^2.
    \end{align}

    With $N = 2^{m+1} t - 1$, and $a = (2^m t - 1) \varepsilon$, we have that
    \begin{align}
        \PP\left(A_{m, i}(\varepsilon)\right) \leq \frac{(2^{m+1} t-1) M^2 \sigma_{\eta}^2}{((2^m t - 1) \varepsilon)^2} \leq   \frac{2^{m+1} t M^2 \sigma_{\eta}^2}{(2^m t / 2) \varepsilon^2} = \frac{8M^2 \sigma_{\eta}^2}{2^m \varepsilon^2}.
    \end{align}

    By union bound over all blocks $m = 0, 1, 2, \ldots$, we have that 
    \begin{align}
        \PP\left( \cup_{m \geq 0} A_{m, i}(\varepsilon)\right) \leq \sum_{m \geq 0} \PP\left(A_{m, i}(\varepsilon)\right) \leq \sum_{m \geq 0} \frac{8M^2 \sigma_{\eta}^2}{2^m \varepsilon^2} = 16M^2 \sigma_{\eta}^2.
    \end{align}

    Thus, for a single coordinate $i$, we have that
    \begin{align}
        \PP\left( \sup_{\tau \geq t} \left|\sum_{k = 1}^{\tau-1} (\bX_{k, a} \eta_k)_i\right| > \varepsilon\right) \leq 16M^2 \sigma_{\eta}^2 / (t\varepsilon^2).
    \end{align}

    To bound a union bound over $i = 1, 2, \ldots, d$, we notice that if 
    \begin{align}
        \sup_{\tau \geq t} \left\|\sum_{k = 1}^{\tau-1} \bX_{k, a} \eta_k\right\|_2 > \varepsilon,
    \end{align}
    then for some $i$, it holds that 
    \begin{align}
        \sup_{\tau \geq t} \left|\sum_{k = 1}^{\tau-1} (\bX_{k, a} \eta_k)\right| > \varepsilon.
    \end{align}

    Therefore, by a union bound over $i = 1, 2, \ldots, d$, we have that
    \begin{align}
        \PP\left( \sup_{\tau \geq t} \left\|\sum_{k = 1}^{\tau-1} \bX_{k, a} \eta_k\right\|_2 > \varepsilon\right) \leq d \PP\left( \sup_{\tau \geq t} \left|\sum_{k = 1}^{\tau-1} (\bX_{k, a} \eta_k)_i\right| > \varepsilon\right) \leq 16dM^2 \sigma_{\eta}^2 / (t\varepsilon^2).
    \end{align}

    Choose $16dM^2 \sigma_{\eta}^2 / (t\varepsilon^2) = 1/(t-1)$, then with probability at least $1 - 1/(t-1)$, we have that
    \begin{align}
        \sup_{\tau \geq t} \left\|\sum_{k = 1}^{\tau-1} \bX_{k, a} \eta_k\right\|_2 \leq \sqrt{16dM^2\sigma_{\eta}^2}.
    \end{align}

    A further union bound over $a \in \cA$ gives that $\mathcal{E}_{2, t}$ holds with probability at least $1 - 1/(t-1)$.

\end{proof}

\subsubsection{Contraction property of g}

The following contraction property is shown under the good event $\cE_{t}$.

% We first define $\bm{g}(\bm{\psi}^*) = \bm{0}$, the root of function $\bm{g}$. \ziping{need to show the existence of the root, and we need a boundedness condition.}

\begin{lemma} On event $\cE_{t}$, let $\bm{\psi}, \bm{\psi}'$ be two elements in $\{\bm{\psi}_{\tau}\}_{\tau \geq t} \subseteq \mathbb{R}^{(d+d^2)|\cA|}$. Consider sufficiently large $t$ such that for all $\tau \geq t$,
\begin{align}
    \pi_{\min}\sigma_{\min}^2 - M \sqrt{16\log(Ad\tau) / (\tau-1)} &\geq \frac{1}{2}\pi_{\min}\sigma_{\min}^2.
    % \lambda / (\tau - 1) + \sigma_{\max}^2 + M \sqrt{16\log(Ad\tau) / (\tau-1)} &\leq 2\sigma_{\max}^2.
\end{align}
    
Then, for sufficiently large $\gamma$, for any $\bm{\psi}, \bm{\psi}' \in \{\bm{\psi}_{\tau}\}_{\tau \geq t}$, we have
    \begin{equation}
        \|\bm{g}(\bm{\psi}) - \bm{g}(\bm{\psi}') \|_2 \leq \kappa \|\bm{\psi} - \bm{\psi}' \|_2 \text{ for some } \kappa < 1.
    \end{equation}
    
\end{lemma}
\begin{proof}
    Define $m_a(\bm{\psi}) = \EE[  \bX_{t, a} Y_{t, a}\mid \bm{\psi}_t = \bm{\psi}]$, and $h_a(\bm{\psi}) = \EE[ \bX_{t, a} \bX_{t, a}^\top \mid \bm{\psi}_t = \bm{\psi}']$.

    We have that 
    \begin{align}
        \|m_{a}(\bm{\psi}) - m_{a}(\bm{\psi}')\|_2 &= \left\|\EE[ (\tilde{\pi}^{\gamma}(a \mid \bX_t, \bm{\psi}) -\tilde{\pi}^{\gamma}(a \mid \bX_t, \bm{\psi}')) \bX_t f(\bX_t, a) ] \right\|_2 \\
        &\leq \left|\tilde{\pi}^{\gamma}(a \mid \bX_t, \bm{\psi}) -\tilde{\pi}^{\gamma}(a \mid \bX_t, \bm{\psi}') \right| \left\|\bX_t f(\bX_t, a)\right\|_2 \\
        &\leq \left|\tilde{\pi}^{\gamma}(a \mid \bX_t, \bm{\psi}) -\tilde{\pi}^{\gamma}(a \mid \bX_t, \bm{\psi}') \right| (M R_{f} + 4M \sigma_{\eta} \sqrt{dA}). \label{eq::m-bound}
    \end{align}
    Similarly, we have that 
    \begin{align}
        \|h_{a}(\bm{\psi}) - h_{a}(\bm{\psi}')\|_2 &= \left\|\EE[ (\tilde{\pi}^{\gamma}(a \mid \bX_t, \bm{\psi}) -\tilde{\pi}^{\gamma}(a \mid \bX_t, \bm{\psi}')) \bX_t \bX_t^{\top} ] \right\| \\
        &\leq \left|\tilde{\pi}^{\gamma}(a \mid \bX_t, \bm{\psi}) -\tilde{\pi}^{\gamma}(a \mid \bX_t, \bm{\psi}') \right| M^2. \label{eq::h-bound}
    \end{align}

    It suffices to upper bound the Lipschitz constant of $\tilde{{\pi}}^{\gamma}(\cdot \mid \bX_t, \bm{\psi})$. We denote by $\tilde{\bm{\pi}} = \tilde{\bm{\pi}}^{\gamma}(\cdot \mid \bX_t, \bm{\psi}) \in [0, 1]^{|\mathcal{A}|}$ and $\tilde{\bm{\pi}'} = \tilde{\bm{\pi}}^{\gamma}(\cdot \mid \bX_t, \bm{\psi}')$, and similarly define $\bm{\pi} = \bm{\pi}^{\gamma}(\cdot \mid \bX_t, \bm{\psi})$ and $\bm{\pi}' = \bm{\pi}^{\gamma}(\cdot \mid \bX_t, \bm{\psi}')$.

\begin{lemma}[L2 projection is Lipschitz continuous]
    \label{lem::clipping-Lipschitz} 
    Let $\operatorname{Clip}(\bm{\pi})$ be defined as in (\ref{eq::clip}). 
    %{\color{red} (Move to Appendix \ref{apdx::clipping})
    %Let $\bm{\pi}\in [0, 1]^{|\mathcal{A}|}$. Define clipping function $\operatorname{Clip}: [0, 1]^{|\mathcal{A}|} \to [0, 1]^{|\mathcal{A}|}$ as 
    %\begin{align}
    %    \operatorname{Clip}(\bm{\pi}) = \max\{\bm{\pi} - \nu^*(\bm{\pi}), \pi_{\min}\},
    %\end{align}
    %where $\nu^*(\bm{\pi})$ is the unique value such that function $q(\nu; \bm{\pi}) = 1$ with function $q(\nu; \bm{\pi}) \coloneqq \sum_{a} \max\{\bm{\pi}_a - \nu, \pi_{\min}\}$.
    %$\operatorname{Clip}(\bm{\pi})$ is the L2 projection of $\bm{\pi}$ onto the set $\{ \bm{\pi} \in [0, 1]^{|\mathcal{A}|} \mid \sum_{a} \bm{\pi}_a = 1, \bm{\pi}_a \geq \pi_{\min}\}$.} 
    Then, $\operatorname{Clip}(\bm{\pi})$ is $(|\mathcal{A}|+1)$-Lipschitz continuous in $\bm{\pi}$.
\end{lemma}

    Lemma \ref{lem::clipping-Lipschitz} implies that the clipping operator (defined in \ref{eq::clip}) is $(|\mathcal{A}|+1)$-Lipschitz, i.e., $\|\tilde{\bm{\pi}} - \tilde{\bm{\pi}}'\|_2 \leq (|\mathcal{A}|+1) \|\bm{\pi} - \bm{\pi}'\|_2$. 
    By the fact that the softmax function is $2/\gamma$-Lipschitz, we have that 
    \begin{align}
        \|\tilde{\bm{\pi}} - \tilde{\bm{\pi}}'\|_2 \leq (|\mathcal{A}|+1) \|\bm{\pi} - \bm{\pi}'\|_2 \leq \frac{2(|\mathcal{A}|+1)M}{\gamma} \|\bm{\beta} - \bm{\beta}'\|_2, \label{eq::first-Lipschitz}
    \end{align} 
    where $\bm{\beta} = (\bbeta_1, \ldots, \bbeta_{|\cA|})$ and each $\bbeta_a = \bPhi_{a}^{-1} \bvarphi_a$.

    To bound $\|\bm{\beta} - \bm{\beta}'\|_2$, we have that 
    \begin{align}
        \|\bbeta_a - \bbeta'_a\|_2 
        &= \|\bPhi_{a}^{-1} \bvarphi_a - \bPhi_{a}^{'-1} \bvarphi_a'\|_2 \\
        &\leq \|\bPhi_{a}^{-1} \bvarphi_a - \bPhi_{a}^{-1} \bvarphi_a'\|_2 + \|\bPhi_{a}^{-1} \bvarphi'_a - \bPhi_{a}^{'-1} \bvarphi_a'\|_2\\
        &\leq \|\bPhi_{a}^{-1}\|_2 \|\bvarphi_a - \bvarphi_a'\|_2 + \|\bPhi_{a}^{-1} - \bPhi_{a}^{'-1}\|_2 \|\bvarphi_a'\|_2 \\
        &\leq \|\bPhi_{a}^{-1}\|_2 \|\bvarphi_a - \bvarphi_a'\|_2 + \|\bPhi_{a}^{-1}\|_2\| \bPhi_a - \bPhi'_{a}\|_2 \|\bPhi_{a}^{'-1}\| \|\bvarphi_a'\|_2
    \end{align}
    Under the good event $\cE_{t}$, we have that both $\|\bPhi_a^{'-1}\|_2, \|\bPhi_a^{-1}\|_2 \leq 2/(\pi_{\min}\sigma_{\min}^2)$, and $\|\bvarphi_a\|_2 \leq 2(M R_{f}+4M\sigma_{\eta} \sqrt{dA})$.

    Thus, we have 
    \begin{align}
        \|\bm{\beta} - \bm{\beta}'\|_2 \leq |\cA|\left(\frac{2}{\pi_{\min}\sigma_{\min}^2} \|\bvarphi_a - \bvarphi_a'\|_2 + \frac{8(M R_{f} + 4M \sigma_{\eta} \sqrt{dA})}{\pi^2_{\min}\sigma_{\min}^4}  \|\bPhi_a - \bPhi'_{a}\|_2 \right)
    \end{align}
    
    Combined with (\ref{eq::first-Lipschitz}), we have
    \begin{align}
        \|\tilde{\bm{\pi}} - \tilde{\bm{\pi}}'\|_2
        \leq& \frac{2(|\mathcal{A}|+1)M}{\gamma} \|\bm{\beta} - \bm{\beta}'\|_2 \\
        \leq& \frac{2(|\mathcal{A}|+1)M}{\gamma} |\cA|\left(\frac{2}{\pi_{\min}\sigma_{\min}^2} \|\bvarphi_a - \bvarphi_a'\|_2 + \frac{8(M R_{f} + 4M \sigma_{\eta} \sqrt{dA})}{\pi^2_{\min}\sigma_{\min}^4}  \|\bPhi_a - \bPhi'_{a}\|_2 \right)
    \end{align}

    The Lipschitz constant of $\tilde{\pi}^{\gamma}(a \mid \bX_t, \bm{\psi})$ w.r.t. the parameter $\bm{\psi}$ is at most
    \begin{align}
        \frac{2|\cA|(|\mathcal{A}|+1)M}{\gamma} \max\left\{\frac{2}{\pi_{\min}\sigma_{\min}^2}, \frac{8(M R_{f} + 4M \sigma_{\eta} \sqrt{dA})}{\pi^2_{\min}\sigma_{\min}^4} \right\}.
    \end{align}

    Combined with (\ref{eq::m-bound}) and (\ref{eq::h-bound}), the Lipschitz constant of function $\bm{g}$ is 
    \begin{align}
        \frac{2|\cA|(|\mathcal{A}|+1)M^3}{\gamma} \max\left\{\frac{2}{\pi_{\min}\sigma_{\min}^2}, \frac{8(M R_{f} + 4M \sigma_{\eta} \sqrt{dA})}{\pi^2_{\min}\sigma_{\min}^4} \right\}
    \end{align}
    For sufficiently large $\gamma$, this constant is less than $1$.
\end{proof}

\subsubsection{Boundedness of moments of $\psi_t$ and $M_t$}

Under the good event $\cE_{t}$, for sufficiently large $t$, we have that for all $a \in \cA$, and $\tau \geq t$,
\begin{align}
    \|\bPhi_{\tau, a}\|_2 \leq M^2, \text{ and } \|\bvarphi_a\|_2 \leq 2(M R_{f} + 4M \sigma_{\eta} \sqrt{dA}).
\end{align}

We also have that for all $a \in \cA$, and $\tau \geq t$,
\begin{align}
    \EE[\|\bm{M}_{\tau, a}\|_2 \mid \cF_{t-1}] s
    &\leq \EE[\|\bX_{t, a} Y_{t,a}\|_2 + \|\Vec(\bX_{t, a}\bX_{t, a}^\top)\|_2]\\
    &= \EE[\|\bX_{t, a} (f(\bX_t, a) + \eta_t)\|_2 + \|\Vec(\bX_{t, a}\bX_{t, a}^\top)\|_2]\\
    &\leq (M R_{f} + 4M \sigma_{\eta} \sqrt{dA}) + M\sigma_{\eta} + M^2.
\end{align}

Thus conditions (C.1) and (C.3) are satisfied.

\subsubsection{Finishing the proof}

Now we are ready to show the convergence of $\bm{\psi}_t$. The above contraction property guarantees that on event $\cE_{t}$, the sequence $\{\bm{\psi}_{\tau}\}_{\tau \geq t}$ converges due to the Robbins-Monro theorem \citep{borkar2008stochastic}. To show the general convergence, we follow the following argument. 

Our goal is to show that for any $\epsilon > 0, \delta > 0$, there exists $T_0$ such that for all $t \geq T_0$, 
\begin{align}
    \PP( \|\bm{\psi}_t - \bm{\psi}^*\|_2 \geq \epsilon ) \leq \delta.
\end{align}
Now for any $\delta > 0$, let $t_0 = \lceil 4/\delta +1 \rceil$. Then we have $\PP(\cE_{t_0}^c) \leq \delta/2$ due to Lemma \ref{lem::good-event}. Therefore, we have that 
\begin{align}
    \PP( \|\bm{\psi}_t - \bm{\psi}^*\|_2 \geq \epsilon ) \leq \PP(\{ \|\bm{\psi}_t - \bm{\psi}^*\|_2 \geq \epsilon \} \cap \cE_{t_0}) + \PP(\cE_{t_0}^c).
\end{align}
The convergence of $\{\bm{\psi}_{\tau}\}_{\tau \geq t_0}$ to $\bm{\psi}^*$ on event $\cE_{t_0}$ implies that there exists a $T_0 \geq t_0$ such that for all $t \geq T_0$, the first term on the right-hand side is less than $\delta/2$. Therefore, we have that 
\begin{align}
    \PP( \|\bm{\psi}_t - \bm{\psi}^*\|_2 \geq \epsilon ) \leq \delta \text{ for all } t \geq T_0.
\end{align}
This shows that $\bm{\psi}_t$ converges in probability to $\bm{\psi}^*$ as $t \rightarrow \infty$.

\subsection{Proof of Proposition \ref{theorem:convergence-of-SGD}}\label{apdx::proof-thm::convergence-of-SGD}

The proof of Proposition \ref{theorem:convergence-of-SGD} follows a direct application of Theorem 2 about the convergence of stochastic approximation process \citep{borkar2008stochastic}. The Theorem is restated in Theorem \ref{thm:sac_converge}.

To check the assumptions are satisfied, we first rewrite the stochastic approximation process in Equation (\ref{eq::SGD-update}) as
\begin{align}
    \hat \bbeta_{t, a}^{\SGD} = \hat \bbeta_{t-1, a}^{\SGD} - \eta_t \EE[1_{\{A_t = a\}}\bh(\bX_t, Y_t(a); \hat \bbeta_{t-1, a}^{\SGD}) \mid \cF_{t-1}] + \bM_{t,a},
\end{align}
We stack the above equation for all $a \in \cA$ into a single expression, we have
\begin{align}
    \hat \bbeta_{t}^{\SGD} = \hat \bbeta_{t-1}^{\SGD} - \eta_t \left(\begin{array}{c}
        \EE[1_{\{A_t = 1\}}\bh(\bX_t, Y_t(1); \hat \bbeta_{t-1, 1}^{\SGD}) \mid \cF_{t-1}] + \bM_{t,1}\\
        \ldots\\
        \EE[1_{\{A_t = |\cA|\}}\bh(\bX_t, Y_t(|\cA|); \hat \bbeta_{t-1, |\cA|}^{\SGD}) \mid \cF_{t-1}] + \bM_{t,|\cA|}
    \end{array}\right),
\end{align}

We define 
\begin{align}
\tilde{\bh}(\bbeta_t) \coloneqq \left(\begin{array}{c}
    \tilde{\bh}_1(\bbeta_t) \\
    \ldots\\
    \tilde{\bh}_{|\cA|}(\bbeta_t)
\end{array}\right),
\text{ and } \bM_t \coloneqq \left(\begin{array}{c}
    \bM_{t,1}\\
    \ldots\\
    \bM_{t,|\cA|}
\end{array}\right),
\end{align}
where $\tilde{\bh}_a(\bbeta_t) \coloneqq \EE[1_{\{A_t = a\}}\bh(\bX_t, Y_t(a); \bbeta_t) \mid \cF_{t-1}]$.

Now we check the assumptions in Theorem \ref{thm:sac_converge}. Firstly, we show that the map $\tilde{\bh}$ is a contraction mapping. Based on Assumption in Proposition \ref{theorem:convergence-of-SGD} that $\bh$ is Lipschitz continuous and the action selection is based on the clipped Boltzmann exploration policy. We have that 
\begin{align}
    \|\tilde{\bh}(\bbeta_1) - \tilde{\bh}(\bbeta_2)\|_2 \leq \sqrt{|\cA|} \max_{a \in \cA} \|\tilde{\bh}_a(\bbeta_1) - \tilde{\bh}_a(\bbeta_2)\|_2.
\end{align}
For each $a \in \cA$, we have
\begin{align}
    &\|\tilde{\bh}_a(\bbeta_1) - \tilde{\bh}_a(\bbeta_2)\|_2 \\
    \leq& \| \EE[\tilde{\pi}^{\gamma}(a \mid \bX_t, \bbeta_1)\bh(\bX_t, Y_t(a); \bbeta_1) - \tilde{\pi}^{\gamma}(a \mid \bX_t, \bbeta_2)\bh(\bX_t, Y_t(a); \bbeta_2)]\|_2\\
    \leq& \|\EE[\tilde{\pi}^{\gamma}(a \mid \bX_t, \bbeta_1) - \tilde{\pi}^{\gamma}(a \mid \bX_t, \bbeta_2)]\|_2 \sqrt{\rho_{\bh}} R_{\bbeta}.
\end{align}
% where $\rho_{\bh} \coloneq \max_{a \in \cA} \sup_{\bbeta \in \RR^d} \|\EE[\bh(\bX, Y(a); \bbeta)]\|_2 / \|\bbeta\|_2$.

Following the similar argument in the proof of Proposition \ref{theorem:convergence-of-Ridge}, we have that 
\begin{align}
    |\EE[\tilde{\pi}^{\gamma}(a \mid \bX_t, \bbeta_1) - \tilde{\pi}^{\gamma}(a \mid \bX_t, \bbeta_2)]| \leq \frac{2(|\cA| - 1)M}{\gamma} \|\bbeta_1 - \bbeta_2\|_2.
\end{align}
Therefore, we have that 
\begin{align}
    \|\tilde{\bh}(\bbeta_1) - \tilde{\bh}(\bbeta_2)\|_2 \leq \sqrt{|\cA|} \frac{2(|\cA| - 1)M}{\gamma} \sqrt{\rho_{\bh}} R_{\bbeta} \|\bbeta_1 - \bbeta_2\|_2.
\end{align}

To ensure contraction, we need to have that 
\begin{align}
    \sqrt{|\cA|} \frac{2(|\cA| - 1)M}{\gamma} \sqrt{\rho_{\bh}} R_{\bbeta} < 1, \text{ i.e., } \gamma > 2\sqrt{|\cA|}(|\cA| - 1)M \sqrt{\rho_{\bh}} R_{\bbeta}.
\end{align}

Secondly, we show that the sequence $\bM_t$ is a martingale difference sequence based on the definition of $\bM_t$. To show that $\bM_t$ is square integrable, we have that 
\begin{align}
    \EE[\|\bM_t\|_2^2 \mid \cF_{t-1}] &\leq |\cA| \max_{a \in \cA} \EE[\|\bM_{t,a}\|_2^2 \mid \cF_{t-1}]\\
    &\leq |\cA| \max_{a \in \cA} \EE[\|\bh(\bX_t, Y_t(a); \bbeta_t)\|_2^2 \mid \cF_{t-1}] \\
    &\leq |\cA| \rho_{\bh} \|\bbeta_t\|_2^2.
\end{align}

\subsection{Proof of Proposition \ref{prop::SGD-boundedness}}
    \label{appx::prop::SGD-boundedness}

The key for stability is to show that $\bar\bbeta_a$ is a globally asymptotically stable equilibrium.

\begin{assumption}[Globally Asymptotically Stable Equilibrium]
    \label{aspt:equilibrium} We first define some notation. Let $\overline{\boldsymbol{\phi}}(\boldsymbol{\beta})=\mathbb{E}\left[\boldsymbol{\phi}\left(\boldsymbol{X}_t, Y_t ; \boldsymbol{\beta}\right) \mid\right. \left.\widehat{\boldsymbol{\beta}}_t^{\mathrm{SGD}}=\boldsymbol{\beta}\right]$.
    We assume that the function $\overline{\boldsymbol{\phi}}_c(\boldsymbol{\beta}):=\boldsymbol{\phi}(c \boldsymbol{\beta}) / c, c \geq 1, \boldsymbol{\beta} \in \mathbb{R}^d$, satisfies that $\overline{\boldsymbol{\phi}}_c(\boldsymbol{\beta}) \rightarrow \boldsymbol{\phi}_{\infty}(\boldsymbol{\beta})$ as $c \rightarrow \infty$, uniformly on compacts for some $\boldsymbol{\phi}_{\infty} \in C\left(\mathbb{R}^d\right)$. Furthermore, the o.d.e.
    $$
    \dot{\boldsymbol{\beta}}(t)=\boldsymbol{\phi}_{\infty}(\boldsymbol{\beta}(t))
    $$
    has the origin as its unique globally asymptotically stable equilibrium. $\boldsymbol{h}(x, y ; \boldsymbol{\beta})$ has Lipschitz constant $L_{\boldsymbol{h}}<\infty$ w.r.t. $\boldsymbol{\beta}$ for all $(\bx, y) \in \mathcal{X} \times \mathbb{R}$.
    \end{assumption} 

    \begin{proof}
    The key is to show that our $\hat \bbeta_t^{\SGD}$ satisfies Assumption \ref{aspt:equilibrium}.
        In example one, we have 
        \begin{align}
            \bh(\bx, y; \bbeta_a) = (y - \bbeta_a^\top \bx)\bx.
        \end{align}
        Thus, we have 
        \begin{align}
            \bar \bphi(\bbeta) 
            &= \EE[(Y_t - \bbeta_{A_t}^\top \bX_t)\bX_t \mid \hat \bbeta_{t}^{\SGD} = \bbeta]\\
            &= \sum_{a} \EE[\pi^{\gamma}(A_t = a \mid \bX_t, \bbeta) (Y_t - \btheta_{A_t}^\top \bX_t)\bX_t]\\
            &= \sum_{a} \EE[\pi^{\gamma}(A_t = a \mid \bX_t, \bbeta) (f(\bX_t, a) - \btheta_{A_t}^\top \bX_t)\bX_t]
        \end{align}
    
        For any $c \geq 1$, we have 
        \begin{align}
            \bar \bphi_c(\bbeta) /c 
            &= \sum_{a} \EE[\pi^{\gamma}(A_t = a \mid \bX_t, c\bbeta) (f(\bX_t, a) - c\bbeta_{a}^\top \bX_t)\bX_t]/c \\
            &\rightarrow -\sum_{a} \EE[1\{a = \argmax_{a'} \bX_t^{\top} \bbeta_{a'}\} \bX_t^{\top}\bX_t \bbeta_{a}] \text{ as } c \rightarrow \infty.
        \end{align}
        
        Define $\bar\Sigma_a(\bbeta) \coloneqq \EE[\bX_t \bX_t^\top 1\{a = \argmax_{a'} \bX_t^{\top} \bbeta_{a'}\}]$ for all $a \in \cA$, and let $\bar\Sigma(\bbeta) \coloneqq \operatorname{diag}(\bar\Sigma_1(\bbeta), \ldots, \bar\Sigma_{|\cA|}(\bbeta))$. 
        
        \begin{lemma}
            \label{lemma:convergence-of-SGD}
            If $\bSigma_X \coloneqq \EE[\bX_t \bX_t^\top] \succ 0$, then $\bar\Sigma_a(\bbeta) \succ 0$ for all $a \in \cA$, and therefore, $\bar\Sigma(\bbeta) \succ 0$.
        \end{lemma}
    
        Then, we have 
        \begin{align}
            \bar \bphi_c(\bbeta) /c \rightarrow -\bar\Sigma(\bbeta) \bbeta \text{ as } c \rightarrow \infty.
        \end{align}
        
        It suffices to show that the o.d.e.
        $$
            \dot{\bbeta}(t) = -\bar\Sigma(\bbeta(t)) \bbeta(t)
        $$
        has the origin as its unique globally asymptotically stable equilibrium. 
    
        We first show that the o.d.e. has a unique equilibrium solution. Setting $\dot{\bbeta}(t) = 0$, we have 
        \begin{align}
            -\bar\Sigma(\bbeta) \bbeta = 0.
        \end{align}
        Since $\bar\Sigma(\bbeta) \succ 0$, we have that $\bbeta = 0$ is the unique equilibrium solution.
    
        Next, we show that the stability of the equilibrium solution. Consider the Lyapunov function:
        \begin{align}
            V(\bbeta) = \frac{1}{2}\|\bbeta\|_2^2.
        \end{align}
        We have 
        \begin{align}
            \dot{V}(\bbeta) = \bbeta^\top \dot{\bbeta}(t) = -\bbeta^\top \bar\Sigma(\bbeta) \bbeta \leq 0.
        \end{align}
        Since $\dot{V}(\bbeta) \leq 0$, the Lyapunov function strictly decreases along trajectories except at the origin. Because $V(\bbeta)$ is positive and radially unbounded, and $\dot{V}(\bbeta)$ is non-positive, by the standard Lyapunov stability theory (Theorem 4.1 \citep{khalil2002nonlinear}), the origin $\bbeta = 0$ is globally asymptotically stable.

        The third condition in Assumption \ref{prop::SGD-boundedness} is
        \begin{align}
            \EE[\|\bh(\bX_t, Y_t(a); \bbeta)\|_2^2] \leq \rho_{\bh}(1+\|\bbeta\|_2^2).
        \end{align}
        This assumption is immediately satisfied from the almost surely boundedness of $\|\bX_t\|_2$ and the boundedness of $\EE[Y_t^2]$ given by Assumption \ref{aspt:Ridge_Conv} for all three cases.

        \textbf{Example \ref{ex::bandits-noisy-contexts}.} Write 
        \begin{align}
            \|\bh(\bX_t, Y_t(a); \bbeta)\| \leq |Y_t(a)| \|\bX_t\| + (\|\bX_t\|^2 + \|\Sigma_e\|) \|\bbeta\|.
        \end{align}
        Squaring and taking expectation gives
        \begin{align}
&\EE\|\bh(\bX_t,Y_t(a);\bbeta)\|^2\\
\le& 2\EE\big[Y_t(a)^2\|\bX_t\|^2\big]+2\EE\big[(\|\bX_t\|^2+\|\Sigma_e\|)^2\big]\|\bbeta\|^2\\
\le& 2M^2(R_y^2 + \sigma_{\eta}^2)+2(M^2 + \|\Sigma_e\|^2)\|\bbeta\|^2 \\
\le& \max\{2M^2(R_y^2 + \sigma_{\eta}^2) + 2(M^2 + \|\Sigma_e\|^2\} (1+\|\bbeta\|^2).
        \end{align}
Example \ref{ex::misspecified-linear-bandits} is a special case of \ref{ex::bandits-noisy-contexts} with $\Sigma_e = 0$.

\textbf{Example \ref{ex::ope}: $\bh(\bx, y; \bbeta) = \pi^e(a \mid \bx)y-\bbeta$.} Here $\|\bh(\bX_t, Y_t(a); \bbeta)\| \leq |\pi^e(a \mid \bX_t)| |Y_t(a)| + \|\bbeta\|$. Since $0\leq\pi^e(\cdot \mid \bx) \leq 1$ for all $\bx$, 
\begin{align}
    \EE \|\bh(\bX_t, Y_t(a);\bbeta)\| \leq 2 \EE[Y_t(a)^2] + 2\|\bbeta\|^2 \leq 2\max\{R_y^2 + \sigma_{\eta}^2, 1\} (1+\|\bbeta\|_2^2).
\end{align}

        Further we apply Theorem 7 \citep{borkar2008stochastic} gives that $\sup_{t\geq 0} \|\hat \bbeta_{t, a}^{\SGD}\|_2$ is bounded almost surely for all $a \in \cA$.

        % The same proof holds for Example \ref{ex::bandits-noisy-contexts} and \ref{ex::ope}.
\end{proof}

\subsection{Proof of Corollary \ref{cor::ex2}}\label{apdx::proof-cor::ex2}

It is straightforward to verify the Assumptions of Theorem \ref{thm::asymptotic-normality-joint-general} given Assumption  \ref{aspt:unconfoundedness}, \ref{aspt:min-sampling-prob} and \ref{aspt:ex2}. We only need to verify the form of the asymptotic variance in (\ref{eq::var-ex2}).

First, we have
\begin{equation}\label{eq::var-estimator-noisy-context-1}
\EE\nabla \bg(\bX_t, Y_t(a);\btheta_a^*) = \EE [\bX_t \bX_t^\top-\bSigma_e] = \EE\bS_t\bS_t^\top = \bSigma_{S}.
\end{equation}
Second, by plugging in $\bg$ to the asymptotic variance in (\ref{eq::asymptotic-normality-general}), we have 
\begin{align}
\bar\bI_a& = 
\EE\left[\frac1{\bar \pi(a|\bX_t)} \bg(\bX_t, Y_t(a);\btheta_a^*)\bg(\bX_t, Y_t(a);\btheta_a^*)^\top\right]\nonumber\\
& = \EE\left[\frac1{\bar \pi(a|\bX_t)}\bigg\{(\bX_t \bX_t^\top - \bSigma_e)\btheta_a^* - \bX_t Y_t(a)\bigg\}^{\otimes 2}\right]\nonumber\\
& = \EE\left[\frac1{\bar \pi(a|\bX_t)}\bigg\{\bh_a(\bX_t, \bS_t) - \bX_t\eta_t\bigg\}^{\otimes 2}\right]\nonumber\\
& = \EE_{\bX_t, \bS_t}\left[\EE_{\eta_t}\bigg[\frac1{\bar \pi(a|\bX_t)}\bigg\{\bh_a(\bX_t, \bS_t) - \bX_t\eta_t\bigg\}^{\otimes 2}|\bX_t, \bS_t\bigg]\right]\nonumber\\
& = \EE_{\bX_t, \bS_t}\left[\frac1{\bar \pi(a|\bX_t)}\Big(\bh_a(\bX_t, \bS_t)\bh_a(\bX_t, \bS_t)^\top + \sigma_\eta^2 \bX_t \bX_t^\top\Big)\right].\label{eq::Ibar-a-noisy-context}
\end{align}
Combining the above with (\ref{eq::var-estimator-noisy-context-1}), we verify that the asymptotic variance shown in Corollary \ref{cor::ex2} matches that in Theorem \ref{thm::asymptotic-normality-joint-general}.

\subsection{Proof of Theorem \ref{theorem:necessary-condition-for-Ridge-convergence}}\label{apdx::proof-thm::necessary-condition-for-Ridge-convergence}

We prove by contradiction. Suppose $\hat{\bm{\beta}}_{t}^{\Ridge}$ converge in probability, then there exists a random vector $\bar{\bbeta}$ such that $\hat{\bm{\beta}}_{t}^{\Ridge} \xrightarrow{p} \bar{\bbeta}$. Recall the definition of $\hat{\bm{\beta}}_{t}^{\Ridge}$ in (\ref{eq:ridge-estimator}):
\begin{align}
    \hat{\bm{\beta}}_{t, a}^{\Ridge} = \left(\lambda 
    \bI + \sum_{i=1}^{t-1} \bX_{i, a} \bX_{i, a}^\top\right)^{-1} \sum_{i=1}^{t-1} \bX_{i, a} Y_{i, a}.
\end{align}
Here we define $\bX_{t, a} = \bX_t 1_{\{A_t = a\}}$, and $Y_{t, a} = Y_t 1_{\{A_t = a\}}$.

% We show that for all $a \in \mathcal{A}$, the following statistic converges in probability:
% \begin{align}
%     \frac{1}{t} \left(\lambda I + \sum_{\tau=1}^{t-1} X_{\tau, a} X_{\tau, a}^\top\right) \xrightarrow{p} \Sigma_a(\bar{\bm{\btheta}}),
% \end{align}
% and 
% \begin{align}
%     \frac{1}{t} \sum_{\tau=1}^{t-1} X_{\tau, a} Y_{\tau, a} \xrightarrow{p} \varphi_a(\bar{\bm{\btheta}}).
% \end{align}

We first note that $\bSigma_{a}(\bm{\beta}) = \EE_{\bX_t, A_t \sim \pi(\cdot \mid \bX_t, \bm{\beta})}[ 1_{\{A_t = a\}} \bX_t \bX_t^\top]$ is a continuous function of $\bm{\beta}$ given any $a \in \mathcal{A}$, because we assume that $\pi(a \mid \bx, \bm{\beta})$ is a continuous function of $\bm{\beta}$ given any $\bx \in \mathcal{X}$. From the continuous mapping theorem, we have $\bSigma_{a}(\hat{\bm{\beta}}_{t-1}^{\Ridge}) \xrightarrow{p} \bSigma_a(\bar{\bm{\beta}})$, and
\begin{align}
    \frac{1}{t} \left(\lambda \bI + \sum_{\tau=1}^{t-1} \bSigma_{a}(\hat{\bm{\beta}}_{\tau, a}^{\Ridge}) \right) \xrightarrow{p} \bSigma_a(\bar{\bm{\beta}}).
\end{align}
Further, since $\bX_{t, a}\bX_{t, a}^\top - \bSigma_{a}(\hat{\bm{\beta}}_{t}^{\Ridge})$ is a zero-mean martingale difference sequence with bounded second moment, by the law of large numbers (Lemma \ref{lem::law-of-large-numbers-for-martingale-difference-sequence})\citep{chow1967strong}, we have that 
\begin{align}
    \frac{1}{t} \sum_{\tau=1}^{t-1} \left(\bX_{\tau, a}\bX_{\tau, a}^\top - \bSigma_{a}(\hat{\bm{\beta}}_{\tau, a}^{\Ridge})\right) \xrightarrow{p} 0.
\end{align}
Combining the above, we deduce that 
\begin{align}
    \frac{1}{t} \left(\lambda \bI + \sum_{\tau=1}^{t-1} \bX_{\tau, a}\bX_{\tau, a}^\top \right) \xrightarrow{p} \bSigma_a(\bar{\bm{\beta}}).
\end{align}

Similar argument gives that 
\begin{align}
    \frac{1}{t} \sum_{\tau=1}^{t-1} \left(\bX_{\tau, a}Y_{\tau, a}\right) \xrightarrow{p} \bm{\varphi}_a(\bar{\bm{\beta}}).
\end{align}

Now we are ready to show the contradiction. First, from the assumptions, we have $\pi(a \mid \bx, \bbeta)>0$ for all $a$, $\bx$, $\bbeta$. Thus $\forall \bbeta$, 
\begin{align*}
\bSigma_{\bar a}(\bbeta) & = \EE_{\bX_t, A\sim \pi(\cdot|\bX_t, \bbeta)}[1_{\{A = \bar a\}}\bX_t\bX_t^\top]\\
& = \EE_{\bX_t}\left[\EE_{A\sim \pi(\cdot|\bX_t, \bbeta)}[1_{\{A = \bar a\}}\bX_t\bX_t^\top]\Big|\bX_t\right]\\
& = \EE_{\bX_t}\left[\bX_t\bX_t^\top\cdot \pi(\bar a \mid \bX_t, \bbeta)\right]\succ \mathbf{0}.
\end{align*}
The last expression above is true because otherwise, there exists a constant vector $\bv\in\RR^d$ s.t. $\bv^\top \EE_{\bX_t}[\bX_t\bX_t^\top\cdot \pi(\bar a \mid \bX_t, \bbeta)]\bv = 0$, which is equivalent to $\EE (\bv^\top \bX_t)^2\cdot \pi(\bar a \mid \bX_t, \bbeta) = 0$, and implies that $(\bv^\top \bX_t)^2\cdot \pi(\bar a \mid \bX_t, \bbeta) = 0$ almost surely. Because $\pi$ is always positive, it has to be that $\bv^\top \bX_t = 0$ almost surely, which contradicts Assumption \ref{aspt:Ridge_Conv}.

Since the mapping $\bM \mapsto \bM^{-1}$ is continuous on the set of positive definite matrices, by the continuous mapping theorem, we deduce that 
\begin{align}
    \left(\frac{1}{t} \left(\lambda \bI + \sum_{\tau=1}^{t-1} \bX_{\tau, {\bar a}}\bX_{\tau, {\bar a}}^\top \right) \right)^{-1} \xrightarrow{p} \left(\bSigma_{{\bar a}}(\bar{\bm{\beta}})\right)^{-1}.
\end{align}
Further, we have 
\begin{align}
    \hat{{\bbeta}}_{t, \bar a}^{\Ridge} \xrightarrow{p} \left(\bSigma_{\bar a}(\bar{\bm{\beta}})\right)^{-1} \bm{\varphi}_{\bar a}(\bar{\bm{\beta}}).
\end{align}
At the same time, because we also assume that   $\hat{\bm{\beta}}_{t}^{\Ridge} \xrightarrow{p} \bar{\bbeta}$, it has to be true that 
$$
\left(\bSigma_{\bar a}(\bar{\bm{\beta}})\right)^{-1} \bm{\varphi}_{\bar a}(\bar{\bm{\beta}}) = \bar{\bbeta}_{\bar a},\quad a.s.
$$
which contradicts (\ref{eq:necessary-condition-for-Ridge-convergence}).

%Therefore, $\bSigma_{a}(\bbeta) \succ 0$ for all $a$.
% We first decide on a particular $a \in \mathcal{A}$. Since $\sum_{a \in \mathcal{A}} \Sigma_a(\bar{\bm{\btheta}}) \succeq \sigma_{\min}^2 I$ is positive definite, by the pigeonhole principle, there exists a $a \in \mathcal{A}$ such that 
% \begin{align}
    % \Sigma_{a}(\bar{\bm{\btheta}}) \succeq \sigma_{\min}^2 I / |\mathcal{A}|.
% \end{align}

%\ziping{fix this}

%This forms a contradiction to the assumption that $\hat{{\btheta}}_{t, a}^{\Ridge} \xrightarrow{p} \bbeta_{a}$ and the fact that 
%\begin{align}
 %   \left\|\left(\bSigma_{a}(\bar{\bm{\beta}})\right)^{-1} \bm{\varphi}_{a}(\bar{\bm{\beta}}) - \bar{\bm{\beta}}_{a}\right\|_2 \geq c, \text{ for any choice of } \bar{\bm{\beta}} \in \mathbb{R}^{d}.
%\end{align}

\subsection{Proof of Theorem \ref{thm::inference-with-Sigma_e_hat}}\label{apdx::proof-thm::inference-with-Sigma_e_hat}

We first present the following lemma, the proof is in Appendix \ref{apdx::proof-lem::Convergence-of-weighted-avg-action-selection}.
\begin{lemma}\label{lem::Convergence-of-weighted-avg-action-selection}
    Under the same conditions of Theorem \ref{thm::inference-with-Sigma_e_hat}, as $T\rightarrow \infty$, $\frac1T\sum_{t=1}^TW_t1_{\{A_t= a\}}\xrightarrow{p}1$.
\end{lemma}

Returning to the main proof, we have 
$$
\nabla \tilde \bG_T(\btheta_a^*) \cdot\sqrt{T}(\tilde \btheta_a - \btheta_a^*) = -\sqrt{T}\tilde \bG_T(\btheta_a^*).
$$
We analyze $\nabla \tilde \bG_T(\btheta_a^*)$ and $\sqrt{T}\tilde \bG_T(\btheta_a^*)$ separately. First, %$\nabla \tilde G_T(\btheta_a^*) = \frac1T\sum_{t=1}^T\tilde \bV_t$, where $\tilde\bV_t = W_t1_{\{A_t= a\}}(f(\bX_t)f(\bX_t)^\top - \hat\bSigma_e)$. We have 
$$
\nabla \tilde \bG_T(\btheta_a^*) = \nabla \bG_T(\btheta_a^*) - \left(\frac1T\sum_{t=1}^TW_t1_{\{A_t= a\}}\right)\cdot (\hat\bSigma_e - \bSigma_e)
$$
From Lemma \ref{lem::convergence-true-derivative-general}, we have $\nabla \bG_T(\btheta_a^*)\xrightarrow{p}\bSigma_{S}$. From Lemma \ref{lem::Convergence-of-weighted-avg-action-selection}, we have $\frac1T\sum_{t=1}^TW_t1_{\{A_t= a\}}\xrightarrow{p}1$. In addition, $\hat\bSigma_e - \bSigma_e\xrightarrow{p} 0$. Combining these facts, we obtain that $\nabla \tilde \bG_T(\btheta_a^*)\xrightarrow{p}\bSigma_{S}$.

We now analyze $\sqrt{T}\tilde \bG_T(\btheta_a^*)$. We have 
$$
\sqrt{T}\tilde \bG_T(\btheta_a^*) = \sqrt{T} \bG_T(\btheta_a^*) - \left(\frac1{T}\sum_{t=1}^TW_t1_{\{A_t= a\}}\right)\cdot \sqrt{T}(\hat\bSigma_e - \bSigma_e)\btheta_a^*,
$$
where from Lemma \ref{lem::asymptotic-normality-true-G-general} and (\ref{eq::Ibar-a-noisy-context}) we have
$\sqrt{T} \bG_T(\btheta_a^*)\xrightarrow{d} \cN(\mathbf{0}, \bar\bI_a)$. Also, $\frac1T\sum_{t=1}^TW_t1_{\{A_t= a\}}\xrightarrow{p}1$.

\begin{itemize}
    \item If $n\gg T$, $\sqrt{T}(\hat\bSigma_e - \bSigma_e) = O_P(\sqrt{T/n}) = o_P(1)$, then $\sqrt{T} \tilde \bG_T(\btheta_a^*)\xrightarrow{d} \cN(\mathbf{0}, \bar\bI_a)$.
    
    \item If $n\ll T$, we instead write 
    $$\sqrt{n}\tilde \bG_T(\btheta_a^*) = \sqrt{n} \bG_T(\btheta_a^*) - \left(\frac1{T}\sum_{t=1}^TW_t1_{\{A_t= a\}}\right)\cdot \sqrt{n}(\hat\bSigma_e - \bSigma_e)\btheta_a^*.$$
    Notice that $\sqrt{n} \bG_T(\btheta_a^*) = \sqrt{n/T}\cdot \sqrt{T} G_T(\btheta_a^*)=o_p(1)$, and from the Central Limit Theorem, we have 
    $$
    \sqrt{n}(\hat\bSigma_e - \bSigma_e)\btheta_a^*\xrightarrow{d} \cN(\mathbf{0}, \bar\bH_a),
    $$
    where $\bar \bH_a: = \EE(\tilde\bV_i-\bSigma_e)\btheta_a^*\btheta_a^{*, \top}(\tilde\bV_i-\bSigma_e)$. Thus, $$\sqrt{n}\bG_T(\btheta_a^*)\xrightarrow{d} \bar \bH_a.$$ 
    Combining the limit of $\nabla\tilde \bG_T(\btheta_a^*)$, we use Slutsky's Theorem to obtain the desired result.
    
    \item If $n = \kappa T$ for some constant $\kappa$, we write
    $$\sqrt{T}\tilde \bG_T(\btheta_a^*) = \sqrt{T} \bG_T(\btheta_a^*) - \left(\frac1{T}\sum_{t=1}^TW_t1_{\{A_t= a\}}\right)\cdot\sqrt{\frac{T}{n}} \cdot\sqrt{n}(\hat\bSigma_e - \bSigma_e)\btheta_a^*.$$
    Note that $\sqrt{T} G_T(\btheta_a^*)\xrightarrow{d} \cN(\mathbf{0}, \bar\bI_a)$ and $
    \sqrt{n}(\hat\bSigma_e - \bSigma_e)\btheta_a^*\xrightarrow{d} \cN(\mathbf{0}, \bar\bH_a)
    $ are independent, the latter obtained from the Central Limit Theorem. Here $\bar \bH_a: = \EE(\tilde\bV_i-\bSigma_e)\btheta_a^*\btheta_a^{*, \top}(\tilde\bV_i-\bSigma_e)$.   
    Combining Lemma \ref{lem::Convergence-of-weighted-avg-action-selection}, we obtain that 
    $$
    \sqrt{T}\tilde \bG_T(\btheta_a^*)\xrightarrow{d} \cN\left(\mathbf{0}, \bar \bI_a + \frac{1}{\kappa}\bar \bH_a\right).
    $$
    Combining the limit of $\nabla\tilde \bG_T(\btheta_a^*)$, we use Slutsky's Theorem to obtain the desired result.
\end{itemize}

\section{Proofs of Auxiliary Lemmas}\label{appB}

    \subsection{Proof of Lemma \ref{lem::clip-l2-projection}}\label{apdx::proof-lem::clip-l2-projection}

{
    We first show that $\operatorname{Clip}(\bm{\pi})$ is the $L_2$ projection of $\bm{\pi}$ onto the set $\{ \bm{\pi} \in [0, 1]^{|\mathcal{A}|} \mid \sum_{a} \bm{\pi}_a = 1, \bm{\pi}_a \geq \pi_{\min}\}$.

    The optimization problem is given by
    \begin{align}
        \min_{\bm{\pi}' \in [0, 1]^{|\mathcal{A}|}} \frac{1}{2}\|\bm{\pi} - \bm{\pi}'\|_2^2 \quad \text{s.t.} \quad \sum_{a} \bm{\pi}'_a = 1, \quad \bm{\pi}'_a \geq \pi_{\min}.
    \end{align}
    The Lagrangian is given by
    \begin{align}
        \mathcal{L}(\bm{\pi}', \nu, \mu) = \frac{1}{2}\|\bm{\pi} - \bm{\pi}'\|_2^2 + \nu \left(1 - \sum_{a} \bm{\pi}'_a\right) + \bm{\mu}^\top (\bm{\pi}' - \pi_{\min}).
    \end{align}
    The KKT conditions are given by
    \begin{align}
        \nabla_{\bm{\pi}'} \mathcal{L}(\bm{\pi}', \nu, \bm{\mu}) = \bm{\pi} - \bm{\pi}' - \nu \mathbf{1} - \bm{\mu} = 0, \\
        \nu \left(1 - \sum_{a} \bm{\pi}'_a\right) = 0, \\
        \bm{\mu} \left(\bm{\pi}' - \pi_{\min}\right) = 0, \\
        \bm{\mu} \geq 0, \quad \bm{\pi}' \geq \pi_{\min}, \quad
        \nu \mathbf{1} = 1.
    \end{align}

    From the optimality condition, we have:
    \begin{align}
        \bm{\pi}'_a=\bm{\pi}_a+\nu+\mu_a
    \end{align}

    Complementary slackness indicates that for each $a$ :
    \begin{itemize}
        \item if $\bm{\pi}'_a>\pi_{\min}$, then $\mu_a=0$, thus:
        \begin{align}
            \bm{\pi}'_a=\bm{\pi}_a+\nu
        \end{align}
        \item if $\bm{\pi}'_a=\pi_{\min}$, then:
        \begin{align}
            \pi_{\min}=\bm{\pi}_a+\nu+\mu_a, \quad \mu_a \geq 0 \Rightarrow \bm{\pi}_a+\nu \leq \pi_{\min}
        \end{align}
    \end{itemize}

    Hence, the solution takes the form:
    \begin{align}
        \bm{\pi}'_a=\max \left(\bm{\pi}_a+\nu, \pi_{\min}\right)
    \end{align}
    with the constraint that $\sum_{a=1}^{|\mathcal{A}|} \bm{\pi}'_a=1$.

    By taking the derivative of $\mathcal{L}(\bm{\pi}', \nu, \bm{\mu})$ w.r.t. $\nu$, we have that $\nu$ is the minimum value that satisfies the KKT conditions.

}

    \subsection{Proof of Lemma \ref{lem::consistency-general}}\label{apdx::proof-lem::consistency-general}
    
    We first prove that for $R_{\btheta}$ stated in Assumption \ref{aspt:boundedness-general}, 
    \begin{equation}\label{eq::uniform-convergence-general}
        \sup_{\|\btheta\|_2\leq R_{\btheta}}\|\bG_T(\btheta) - \EE\bG_T(\btheta)\|_2\xrightarrow{p} 0
    \end{equation}
    as $T\rightarrow\infty$. In fact, we only need to prove 
    \begin{equation}\label{eq::uniform-convergence-by-entry-general}
        \sup_{\|\btheta\|_2\leq R_{\btheta}}\|\bG_T^{(i)}(\btheta) - \EE\bG_T^{(i)}(\btheta)\|_2\xrightarrow{p} 0
    \end{equation}
    for all $i = 1, \ldots, d$. Here $\bG_T^{(i)}(\btheta)$ denotes the $i$-th entry of $\bG_T(\btheta)$.
    
    Fix any $\epsilon>0$. Let $\bTheta_\epsilon = \{\btheta_j: j = 1, \ldots, N_\epsilon\}$ be an $\epsilon$-net of the set $\bTheta := \overline{\cB(\mathbf{0}, R_{\btheta})}$ with finite cardinality $N_\epsilon$. This means that $\forall \btheta\in\bTheta_\epsilon$, $\exists j\in[N_\epsilon]$ such that $\|\btheta - \btheta_j\|_2\leq \epsilon$. Then from Assumption \ref{aspt:smoothness-general}, we have 
    \begin{align*}
    |\bg^{(i)}(\bX_t, Y_t(a); \btheta) - \bg^{(i)}(\bX_t, Y_t(a); \btheta_j)| 
    &\leq \|\bg(\bX_t, Y_t(a); \btheta) - \bg(\bX_t, Y_t(a); \btheta_j)\|_2\\
    &\leq \left\|\int_{0}^1\nabla \bg(\bX_t, Y_t(a);\btheta_j + u(\btheta - \btheta_j))\mathrm{d}u\cdot (\btheta - \btheta_j)\right\|_2\\
    &\leq \phi(\bX_t, Y_t(a))\cdot \|\btheta - \btheta_j\|_2\\
    &\leq \epsilon\phi(\bX_t, Y_t(a)).
    \end{align*}
    Here $\bg^{(i)}(\bX_t, Y_t(a); \btheta)$ denotes the $i$-th entry of $\bg(\bX_t, Y_t(a); \btheta)$, for any $\btheta$. Thus, 
    \begin{equation}\label{eq::gi-nested-bounds}
        l_j(\bX_t, Y_t(a))\leq \bg^{(i)}(\bX_t, Y_t(a); \btheta)\leq u_j(\bX_t, Y_t(a)),
    \end{equation}
    where we define $l_j(\bX_t, Y_t(a)) = \bg^{(i)}(\bX_t, Y_t(a); \btheta_j) - \epsilon\phi(\bX_t, Y_t(a))$, $u_j(\bX_t, Y_t(a)) = \bg^{(i)}(\bX_t, Y_t(a); \btheta_j) + \epsilon\phi(\bX_t, Y_t(a))$.
    
    Now, notice that $\bG_T(\btheta) = \frac1T\sum_{t=1}^T\bZ_t(\btheta)$, where $\bZ_t(\btheta) = \frac{1}{\pi_t(A_t)}1_{\{A_t = a\}}\bg(\bX_t, Y_t; \btheta)$. We have 
    \begin{align*}
        \EE[\bZ_t(\btheta)|\cH_{t-1}]
        & = \EE_{\bX_t}\big[\EE[\bZ_t(\btheta)|\cH_{t-1}, \bX_t]\big|\cH_{t-1}\big]\\
        & = \EE_{\bX_t}\bigg[\EE_{A_t\sim \pi_t}\Big[\frac{1}{\pi_t(A_t)}1_{\{A_t = a\}}\Big|\cH_{t-1}, \bX_t\Big]\cdot \EE_{Y_t(a)}[\bg(\bX_t, Y_t(a); \btheta)|\cH_{t-1}, \bX_t]\bigg|\cH_{t-1}\bigg]\\
        & = \EE_{\bX_t}\big[\EE_{Y_t(a)}[\bg(\bX_t, Y_t(a); \btheta)|\cH_{t-1}, \bX_t]\big|\cH_{t-1}\big]\\
        & = \EE\bg(\bX_t, Y_t(a); \btheta).
    \end{align*}
    Here in the second equation, we use Assumption \ref{aspt:unconfoundedness}. This implies that 
    \begin{equation}\label{eq::GT-expectation}
        \EE\bG_T(\btheta) =  \EE\bg(\bX_t, Y_t(a); \btheta) = \EE[\bZ_t(\btheta)|\cH_{t-1}],
    \end{equation}
    $\forall t\in[T]$. Combining (\ref{eq::gi-nested-bounds}), we deduce that 
    \begin{align}
    \sup_{\btheta\in\bTheta}\{\bG_T^{(i)}(\btheta) - \EE\bG_T^{(i)}(\btheta)\}
    & = \sup_{\btheta\in\bTheta}\frac1T\sum_{t=1}^T
    \bigg\{\frac{1}{\pi_t(A_t)}1_{\{A_t = a\}}\bg^{(i)}(\bX_t, Y_t; \btheta)\nonumber \\
    &\quad\quad\quad\quad\quad\quad - \EE\bigg[\frac{1}{\pi_t(A_t)}1_{\{A_t = a\}}\bg^{(i)}(\bX_t, Y_t; \btheta)\bigg|\cH_{t-1}\bigg]\bigg\}\nonumber\\
    & \leq \max_{j\in[N_\epsilon]} \frac1T\sum_{t=1}^T
    \bigg\{\frac{1}{\pi_t(A_t)}1_{\{A_t = a\}}u_j(\bX_t, Y_t(a))\nonumber\\
    &\quad\quad\quad\quad\quad\quad\quad - \EE\bigg[\frac{1}{\pi_t(A_t)}1_{\{A_t = a\}}l_j(\bX_t, Y_t(a))\bigg|\cH_{t-1}\bigg]\bigg\}\nonumber\\
    & \leq \Delta_1 + \Delta_2,\label{eq::maxima-mtg-decomposition}
    \end{align}
    where 
    $$
    \Delta_1 := \max_{j\in[N_\epsilon]} \frac1T\sum_{t=1}^T
    \left\{\frac{1}{\pi_t(A_t)}1_{\{A_t = a\}}u_j(\bX_t, Y_t(a)) - \EE\bigg[\frac{1}{\pi_t(A_t)}1_{\{A_t = a\}}u_j(\bX_t, Y_t(a))\bigg|\cH_{t-1}\bigg]\right\},
    $$
    $$
    \Delta_2:= \max_{j\in[N_\epsilon]} \frac1T\sum_{t=1}^T
    \EE\bigg[\frac{1}{\pi_t(A_t)}1_{\{A_t = a\}}\big(u_j(\bX_t, Y_t(a))-l_j(\bX_t, Y_t(a))\big)\bigg|\cH_{t-1}\bigg].
    $$
    We first analyze $\Delta_1$. In fact we have 
    \begin{equation}\label{eq::Delta1_bound}
        \Delta_1\leq \sum_{j\in[N_\epsilon]} \frac1T\sum_{t=1}^T
    \left\{\frac{1}{\pi_t(A_t)}1_{\{A_t = a\}}u_j(\bX_t, Y_t(a)) - \EE\bigg[\frac{1}{\pi_t(A_t)}1_{\{A_t = a\}}u_j(\bX_t, Y_t(a))\bigg|\cH_{t-1}\bigg]\right\}.
    \end{equation}
    For any $\epsilon'>0$,
    \begin{align}
    & \PP\left(\bigg|\frac1T\sum_{t=1}^T
    \left\{\frac{1_{\{A_t = a\}}}{\pi_t(A_t)}u_j(\bX_t, Y_t(a))\!-\!\EE\bigg[\frac{1_{\{A_t = a\}}}{\pi_t(A_t)}u_j(\bX_t, Y_t(a))\bigg|\cH_{t-1}\bigg]\right\}\bigg|>\epsilon'\right)\nonumber\\
    \leq & \frac1{T^2\epsilon'^2}\EE\left(\sum_{t=1}^T
    \left\{\frac{1_{\{A_t = a\}}}{\pi_t(A_t)}u_j(\bX_t, Y_t(a)) - \EE\bigg[\frac{1_{\{A_t = a\}}}{\pi_t(A_t)}u_j(\bX_t, Y_t(a))\bigg|\cH_{t-1}\bigg]\right\}\right)^2\nonumber\\
    = & \frac1{T^2\epsilon'^2}\sum_{t=1}^T\EE\left(
    \frac{1_{\{A_t = a\}}}{\pi_t(A_t)}u_j(\bX_t, Y_t(a)) - \EE\bigg[\frac{1_{\{A_t = a\}}}{\pi_t(A_t)}u_j(\bX_t, Y_t(a))\bigg|\cH_{t-1}\bigg]\right)^2\nonumber\\
    \leq & \frac1{T^2\epsilon'^2}\sum_{t=1}^T\EE\left(
    \frac{1_{\{A_t = a\}}}{\pi_t(A_t)}u_j(\bX_t, Y_t(a))\right)^2\nonumber\\
    \leq & \frac1{T^2\epsilon'^2}\sum_{t=1}^T \frac1{\pi_{\min}^2}\EE u_j^2(\bX_t, Y_t(a))\nonumber\\
    = & \frac1{T^2\epsilon'^2}\sum_{t=1}^T \frac1{\pi_{\min}^2}\EE \big(\bg^{(i)}(\bX_t, Y_t(a); \btheta_j) + \epsilon\phi(\bX_t, Y_t(a))\big)^2\nonumber\\
    \leq & \frac1{\pi_{\min}^2T^2\epsilon'^2}\cdot T \EE \big[2(\bg^{(i)}(\bX_1, Y_1(a); \btheta_j))^2 + 2\epsilon^2\phi^2(\bX_1, Y_1(a))\big]\nonumber\\
    \leq & \frac{2M_2' + \epsilon^2M_\phi}{\pi_{\min}^2T\epsilon'^2}\rightarrow 0.\label{eq::single-theta-u-convergence}
    \end{align}
    Here $M_2':= \sup_{\|\btheta\|_2\leq R_{\btheta}}\EE\|\bg(\bX_t, Y_t(a);\btheta)\|_2^2$, $M_\phi:= \EE\phi(\bX_t, Y_t(a))^2$. Above, the first inequality is due to Chebyshev's inequality. The first equality is due to the following fact: Denote $u_{j, t} = \frac{1}{\pi_t(A_t)}1_{\{A_t = a\}}u_j(\bX_t, Y_t(a))$. Then for $t_1<t_2$,
    \begin{align*}
       & \EE (u_{j, t_1} - \EE[u_{j, t_1}|\cH_{t_1-1}])(u_{j, t_2} - \EE[u_{j, t_2}|\cH_{t_2-1}]) \\
       = & \EE\big[ \EE[(u_{j, t_1} - \EE[u_{j, t_1}|\cH_{t_1-1}])(u_{j, t_2} - \EE[u_{j, t_2}|\cH_{t_2-1}])\big|\cH_{t_1-1}]\big]\\
       = & \EE\big[ (u_{j, t_1} - \EE[u_{j, t_1}|\cH_{t_1-1}])\cdot \EE[u_{j, t_2} - \EE[u_{j, t_2}|\cH_{t_2-1}]\big|\cH_{t_1-1}]\big]\\
       = & \EE\big[ (u_{j, t_1} - \EE[u_{j, t_1}|\cH_{t_1-1}])\cdot 0\big] = 0.
    \end{align*}
    The second to last inequality is due to the fact that $(a+b)^2\leq 2(a^2+b^2)$ for any $a, b\in\RR$. The final inequality uses Assumption \ref{aspt:boundedness-general} and Assumption \ref{aspt:smoothness-general}, which implies $M_2'<\infty$ and $M_\phi<\infty$.
    
    Combining (\ref{eq::Delta1_bound}) and (\ref{eq::single-theta-u-convergence}), we obtain that $\Delta_1\xrightarrow{p}0$.
    
    Next, we analyze $\Delta_2$. We have 
    \begin{align*}
    \Delta_2
    &\leq \max_{j\in[N_\epsilon]} \frac1T\sum_{t=1}^T\frac1{\pi_{\min}}
    \EE\bigg[u_j(\bX_t, Y_t(a))-l_j(\bX_t, Y_t(a))\bigg|\cH_{t-1}\bigg]\\
    &\leq \max_{j\in[N_\epsilon]} \frac1T\sum_{t=1}^T\frac1{\pi_{\min}}
    \EE\bigg[2\epsilon\phi(\bX_t, Y_t(a))\bigg|\cH_{t-1}\bigg]\\
    & = \frac{2\epsilon}{\pi_{\min}}
    \EE\phi(\bX_1, Y_1(a))\\
    & \leq \frac{2\epsilon}{\pi_{\min}}\cdot \sqrt{\EE\phi(\bX_1, Y_1(a))^2} \\
    & \leq \frac{2\epsilon\sqrt{M_\phi}}{\pi_{\min}}.
    \end{align*}
    Plugging in the analysis of $\Delta_1$ and $\Delta_2$ into (\ref{eq::maxima-mtg-decomposition}), we obtain that $\forall \epsilon>0$, 
    $$
    \PP\left(\sup_{\btheta\in\bTheta}\{\bG_T^{(i)}(\btheta) - \EE\bG_T^{(i)}(\btheta)\}>\frac{2\epsilon\sqrt{M_\phi}}{\pi_{\min}}\right)\rightarrow 0.
    $$
    This implies that $\forall \epsilon>0$, 
    $$
    \PP\left(\sup_{\btheta\in\bTheta}\{\bG_T^{(i)}(\btheta) - \EE\bG_T^{(i)}(\btheta)\}>\epsilon\right)\rightarrow 0.
    $$
    Similarly, we have 
    $$
    \PP\left(\sup_{\btheta\in\bTheta}\{-\bG_T^{(i)}(\btheta) + \EE\bG_T^{(i)}(\btheta)\}>\epsilon\right)\rightarrow 0.
    $$
    Therefore, we have proved (\ref{eq::uniform-convergence-by-entry-general}). Further, (\ref{eq::uniform-convergence-general}) is proved.
    
    Now we return to prove the main results. Define 
    $$
    \hat\btheta_a^{(T)} = \argmin_{\|\btheta\|_2\leq R_{\btheta}}\|\bG_T(\btheta)\|_2,
    $$
    and 
    $$\tilde\btheta_a^{(T)} = \btheta_a^* - [\nabla\bar \bG(\btheta_a^*)]^{-1}\bG_T(\btheta_a^*).$$ 
    Here $\bar\bG(\btheta) := \EE \bG_T(\btheta) = \EE \bg(\bX_t, Y_t(a);\btheta)$ due to (\ref{eq::GT-expectation}). From Assumption \ref{aspt:smoothness-general}, $\nabla\bar \bG(\btheta_a^*)$ is invertible. Combining Lemma \ref{lem::asymptotic-normality-true-G-general}, we have 
    \begin{equation}\label{eq::exist-z-estimator-consistency}
        \tilde\btheta_a^{(T)} - \btheta_a^* = O_p(1/\sqrt{T}).
    \end{equation}
    In addition, from Taylor expansion, we have 
    \begin{align}
    \bG_T(\tilde\btheta_a^{(T)})
    & = \bG_T(\btheta_a^*) + \nabla\bG_T(\btheta_a^*)(\tilde\btheta_a^{(T)} - \btheta_a^*) + o_p(\|\tilde\btheta_a^{(T)} - \btheta_a^*\|_2)\nonumber\\
    & = \bG_T(\btheta_a^*) - \nabla\bG_T(\btheta_a^*)[\nabla\bar \bG(\btheta_a^*)]^{-1}\bG_T(\btheta_a^*) + o_p(1/\sqrt{T})\nonumber\\
    & = \bG_T(\btheta_a^*) - (1 + o_p(1))\bG_T(\btheta_a^*) + o_p(1/\sqrt{T})\nonumber\\
    & = o_p(1/\sqrt{T}).\label{eq::exist-z-estimator}
    \end{align}
    Here the second equality is from the definition of $\tilde\btheta_a^{(T)}$ and due to (\ref{eq::exist-z-estimator-consistency}). The third equality is obtained from the following two facts: From Lemma \ref{lem::convergence-true-derivative-general}, 
    $\nabla\bG_T(\btheta_a^*)\xrightarrow{p}\EE\nabla\bg(\bX_t, Y_t(a);\btheta_a^*)$; From (\ref{eq::GT-expectation}) and Assumption \ref{aspt:smoothness-general}, $\nabla \bar \bG(\btheta_a^*)  = \frac{\partial}{\partial \btheta}\EE\bG_T(\btheta)|_{\btheta =\btheta_a^*}= \frac{\partial}{\partial \btheta}\EE \bg(\bX_t, Y_t(a);\btheta)|_{\btheta =\btheta_a^*} = \EE \frac{\partial}{\partial \btheta}\bg(\bX_t, Y_t(a);\btheta)|_{\btheta =\btheta_a^*}$ nonsingular. Here the exchangeability between expectation and differentiation can be obtained by standard arguments using dominated convergence theorem, see e.g. Theorem 16.8 in \citep{billingsley1995probability}. The last equality is obtained from Lemma \ref{lem::asymptotic-normality-true-G-general}.
    
    Combining (\ref{eq::exist-z-estimator-consistency}) (which implies $\PP(\|\tilde\btheta_a^{(T)}\|_2>R_{\btheta})\rightarrow 0$) and (\ref{eq::exist-z-estimator}), we deduce that 
    \begin{align*}
        \|\bG_T(\hat\btheta_a^{(T)})\|_2 &= \|\bG_T(\hat\btheta_a^{(T)})\|_21_{\{\|\tilde\btheta_a^{(T)}\|_2\leq R_{\btheta}\}} + \|\bG_T(\hat\btheta_a^{(T)})\|_21_{\{\|\tilde\btheta_a^{(T)}\|_2> R_{\btheta}\}}\\
        & \leq \|\bG_T(\tilde\btheta_a^{(T)})\|_2 + o_p(1/\sqrt{T}) = o_p(1/\sqrt{T}).
    \end{align*}
    Thus we have proved the first part of Lemma \ref{lem::consistency-general}. 
    
    Next, for any sequence $\hat\btheta_a^{(T)}$ such that $\hat\btheta_a^{(T)}\leq R_{\btheta}$ and (\ref{eq::estimating-equation-general}) holds, we proceed to prove $\hat\btheta_a^{(T)}\xrightarrow{p}\btheta_a^*$. According to Assumption \ref{aspt:identifiability-general}, for any $\epsilon>0$, we have 
    $$
    \delta_\epsilon:=\inf_{\|\btheta - \btheta_a^*\|_2>\epsilon}\|\EE\bg(\bX_t, Y_t(a); \btheta)\|_2>0.
    $$
    From (\ref{eq::GT-expectation}), we deduce that 
    $$
    \inf_{\|\btheta - \btheta_a^*\|_2>\epsilon}\|\EE\bG_T( \btheta)\|_2 = \inf_{\|\btheta - \btheta_a^*\|_2>\epsilon}\|\EE\bg(\bX_t, Y_t(a);\btheta)\|_2\geq \delta_{\epsilon}.
    $$
    Combine (\ref{eq::uniform-convergence-general}), we have 
    \begin{align*}
    \inf_{\btheta\in\bTheta, \|\btheta - \btheta_a^*\|_2>\epsilon} \|\bG_T(\btheta)\|_2\geq \inf_{\|\btheta - \btheta_a^*\|_2>\epsilon}\|\EE\bG_T( \btheta)\|_2 - \sup_{\btheta\in\bTheta}\|\bG_T(\btheta) - \EE\bG_T(\btheta)\|_2\geq \delta_{\epsilon} - o_p(1).
    \end{align*}
    Here $\bTheta = \overline{\cB(\mathbf{0}, R_{\btheta})}$. This implies 
    \begin{equation}\label{eq::consistency1}
    \lim_{T\rightarrow\infty}\PP\left(\inf_{\btheta\in\bTheta, \|\btheta - \btheta_a^*\|_2>\epsilon} \|\bG_T(\btheta)\|_2\leq \frac12\delta_{\epsilon}\right)=0.
    \end{equation}
    At the same time, from the property of $\hat\btheta_a^{(T)}$, we have for any $\epsilon'>0$,
    $$
    \lim_{T\rightarrow\infty}\PP\left(\|\bG_{T}(\hat\btheta_a^{(T)})\|_2>\epsilon'/\sqrt{T}\right)=0.
    $$
    Thus, 
    \begin{equation}\label{eq::consistency2}
    \lim_{T\rightarrow\infty}\PP\left(\|\bG_{T}(\hat\btheta_a^{(T)})\|_2>\frac12\delta_{\epsilon}\right)= 0.
    \end{equation}
    We combine (\ref{eq::consistency1}) and (\ref{eq::consistency2}) and get
    \begin{align*}
    \PP\big(\|\hat\btheta_a^{(T)} - \btheta_a^*\|_2>\epsilon\big)
    & =  \PP\left(\|\hat\btheta_a^{(T)} -\btheta_a^*\|_2>\epsilon, \|\bG_{T}(\hat\btheta_a^{(T)})\|_2\leq \frac12\delta_{\epsilon}\right)\\
    &\quad + \PP\left(\|\hat\btheta_a^{(T)} - \btheta_a^*\|_2>\epsilon, \|\bG_{T}(\hat\btheta_a^{(T)})\|_2>\frac12\delta_{\epsilon}\right)\\
    &\leq \PP\left(\inf_{\btheta\in\bTheta, \|\btheta - \btheta_a^*\|_2>\epsilon} \|\bG_T(\btheta)\|_2\leq \frac12\delta_{\epsilon}\right) + \PP\left(\|\bG_{T}(\hat\btheta_a^{(T)})\|_2>\frac12\delta_{\epsilon}\right)\\
    &\rightarrow 0
    \end{align*}
    as $T\rightarrow \infty$. As the above holds for any $\epsilon>0$, the consistency of $\hat\btheta_a^{(T)}$ is proved.
    
    \subsection{Proof of Lemma \ref{lem::convergence-true-derivative-general}}\label{apdx::proof-lem::convergence-true-derivative-general}
    
    Recall that $\nabla\bG_T(\btheta_a^*) = \frac1T\sum_{t=1}^T\bV_t$, where\\ $\bV_t:=\frac{1}{\pi_t(A_t)}1_{\{A_t = a\}}\nabla\bg(\bX_t, Y_t(a);\btheta_a^*)$. We have for any nonrandom vectors $\bc, \bc'\in\RR^d$, 
    \begin{align*}
    \EE[c^\top\bV_t c'|\cH_{t-1}]
    & = \EE_{\bX_t}\left[\EE_{A_t\sim \pi_t, Y_t(a)}[\bc^\top\bV_t\bc'|\cH_{t-1}, \bX_t]\bigg|\cH_{t-1}\right]\\
    & = \EE_{\bX_t}\bigg[\EE_{A_t\sim \pi_t}\bigg[\frac{1}{\pi_t(A_t)}1_{\{A_t = a\}}\bigg|\cH_{t-1}, \bX_t\bigg]\cdot\\
    &\quad \quad\quad \quad\EE_{Y_t(a)}[\bc^\top\nabla\bg(\bX_t, Y_t(a);\btheta_a^*)\bc'|\cH_{t-1}, \bX_t]\bigg|\cH_{t-1}\bigg]\\
    & = \EE_{\bX_t}\left[1\cdot \EE_{Y_t(a)}[\bc^\top\nabla\bg(\bX_t, Y_t(a);\btheta_a^*)\bc'|\cH_{t-1}, \bX_t]\bigg|\cH_{t-1}\right]\\
    & = \bc^\top\EE[\nabla\bg(\bX_t, Y_t(a);\btheta_a^*)]\bc'.
    \end{align*}
    Here the second inequality uses Assumption \ref{aspt:unconfoundedness}. From the above we deduce that $\forall \delta>0$,
    \begin{align}
    &\PP\left(|\bc^\top[\nabla\bG_T(\btheta_a^*) - \EE\nabla\bg(\bX_t, Y_t(a);\btheta_a^*)]\bc'|>\delta\right)\nonumber\\
    = & \PP\left(\bigg|\frac1T\sum_{t=1}^T[\bc^\top\bV_t\bc' - \EE[\bc^\top\bV_t\bc'|\cH_{t-1}]]\bigg|>\delta\right)\nonumber\\
    \leq & \frac1{\delta^2T^2}\EE\left(\sum_{t=1}^T[\bc^\top\bV_t\bc' - \EE[\bc^\top\bV_t\bc'|\cH_{t-1}]]\right)^2\nonumber\\
     = & \frac1{\delta^2T^2}\sum_{t=1}^T\EE\left(\bc^\top\bV_t\bc' - \EE[\bc^\top\bV_t\bc'|\cH_{t-1}]\right)^2\nonumber\\
     \leq & \frac1{\delta^2T^2}\sum_{t=1}^T\EE\left(\bc^\top\bV_t\bc' \right)^2\nonumber\\
     \leq & \frac1{\delta^2T^2}\sum_{t=1}^T\frac{1}{\pi_{\min}^2}\EE\left(\bc^\top\nabla\bg(\bX_t, Y_t(a); \btheta_a^*)\bc' \right)^2\rightarrow 0.\label{eq::convergence-gradient-general-cc'}
    \end{align}
    Here the first inequality is because of Chebyshev's inequality. The second equality is due to the following fact: Let $v_{t} = \bc^\top\bV_t\bc'$. Then for $t_1<t_2$,
    \begin{align*}
       & \EE (v_{t_1} - \EE[v_{t_1}|\cH_{t_1-1}])(v_{t_2} - \EE[v_{t_2}|\cH_{t_2-1}]) \\
       = & \EE\big[ \EE[(v_{t_1} - \EE[v_{t_1}|\cH_{t_1-1}])(v_{t_2} - \EE[v_{t_2}|\cH_{t_2-1}])\big|\cH_{t_1-1}]\big]\\
       = & \EE\big[(v_{t_1} - \EE[v_{t_1}|\cH_{t_1-1}])\cdot \EE[v_{t_2} - \EE[v_{t_2}|\cH_{t_2-1}]\big|\cH_{t_1-1}]\big]\\
       = & \EE\big[ (v_{t_1} - \EE[v_{t_1}|\cH_{t_1-1}])\cdot 0\big] = 0.
    \end{align*}
    The last convergence uses Assumption \ref{aspt:smoothness-general}. 
    
    Finally, because (\ref{eq::convergence-gradient-general-cc'}) holds for any $\bc, \bc'$, we conclude our proof.
    
    \subsection{Proof of Lemma \ref{lem::2nd-derivative-boundness-general}}\label{apdx::proof-lem::2nd-derivative-boundness-general}
    
    Note that $\nabla^2\bG_T(\btheta) = \frac1T\sum_{t=1}^T\frac{1_{\{A_t = a\}}}{\pi_t(A_t)}\nabla^2\bg(\bX_t, Y_t(a); \btheta)$. According to Assumption \ref{aspt:smoothness-general}, $\forall \btheta\in\overline{\cB(\btheta_a^*, \epsilon_0)}$, 
    \begin{align*}
    \|\nabla^2\bG_T(\btheta)\|_1&\leq \frac1{\pi_{\min}}\cdot \frac1T\sum_{t=1}^T\|\nabla^2\bg(\bX_t, Y_t(a);\btheta)\|_1\\
    &\leq \frac1{\pi_{\min}T}\sum_{t=1}^Td^2\sup_{i\in[d]}\|\nabla^2\bg^{(i)}(\bX_t, Y_t(a);\btheta)\|_2\\
    &\leq \frac{d^2}{\pi_{\min}T}\sum_{t=1}^T\Phi(\bX_t, Y_t(a)).
    \end{align*}
    Thus,
    $$
    \sup_{\|\btheta - \btheta_a^*\|_2\leq \epsilon_0}\|\nabla^2\bG_T(\btheta)\|_1\leq \frac{d^2}{\pi_{\min}}\cdot \frac1T\sum_{t=1}^T\Phi(\bX_t, Y_t(a)) = \cO_p(1)
    $$
    as $T\rightarrow \infty$.
    
    \subsection{Proof of Lemma \ref{lem::asymptotic-normality-true-G-general}}\label{apdx::proof-lem::asymptotic-normality-true-G-general}
    
    We have $\bG_T(\btheta_a^*) = \frac1T\sum_{t=1}^T\bZ_t$, where we define \\
    $\bZ_t:= \frac{1}{\pi_t(A_t)}1_{\{A_t = a\}}\bg(\bX_t, Y_t(a);\btheta_a^*)$. From the Cramer-Wold theorem, in order to show the desired asymptotic normality, it suffices to show that for any $\bc\in\RR^d$, $\bc^\top\cdot \sqrt{T}\bG_T(\btheta_a^*) = \bc^\top\cdot\frac1{\sqrt{T}}\sum_{t=1}^T\bZ_t\xrightarrow{d} \cN(\mathbf{0}, \bc^\top\bar\bI_a\bc)$.
    
    From  \citep{dvoretzky1972asymptotic}, Theorem 2.2, the above asymptotic result can be obtained by ensuring
    \begin{gather}
       \EE [\bZ_t|\cH_{t-1}] = \mathbf{0}\quad  \forall t\in[T],\label{eq::cond-exp-general}\\
       \frac1T\sum_{t\in[T]}\mathrm{Var}(\bc^\top\bZ_t|\cH_{t-1}) \xrightarrow{p} \bc^\top \bar\bI_a \bc,\label{eq::cond-var-general}\\
       \frac1T\sum_{t\in[T]}\EE\left[(\bc^\top \bZ_t)^21_{\{|\bc^\top\bZ_t|>\sqrt{T}\delta\}}\Big|\cH_{t-1}\right]\xrightarrow{p} 0\quad \forall \delta>0.\label{eq::cond-lindeberg-general}
    \end{gather}
    Below we check these facts one by one.
    
    \textbf{Check (\ref{eq::cond-exp-general})}: We have
    \begin{align}
    \EE [\bZ_t|\cH_{t-1}] 
    &= \EE_{\bX_t}\left[\EE_{A_t\sim \pi_t, Y_t(a)}[\bZ_t|\cH_{t-1}, \bX_t]|\cH_{t-1}\right]\nonumber\\
    & = \EE_{\bX_t}\bigg[\EE_{A_t\sim \pi_t}\bigg[\frac{1}{\pi_t(A_t)}1_{\{A_t = a\}}\Big|\cH_{t-1}, \bX_t\bigg]\nonumber\\
    &\quad\quad\quad\quad\cdot \EE_{Y_t(a)}[\bg(\bX_t, Y_t(a);\btheta_a^*)|\cH_{t-1}, \bX_t]\bigg|\cH_{t-1}\bigg]\nonumber\\
    & = \EE_{\bX_t}\left[1\cdot \EE_{Y_t(a)}[\bg(\bX_t, Y_t(a);\btheta_a^*)|\cH_{t-1}, \bX_t]\bigg|\cH_{t-1}\right]\nonumber\\
    & = \EE[\bg(\bX_t, Y_t(a);\btheta_a^*)] = \mathbf{0}.\nonumber
    \end{align}
    Here, the second equality is because of Assumption \ref{aspt:unconfoundedness}. 
    
    \textbf{Check (\ref{eq::cond-var-general})}: Based on (\ref{eq::cond-exp-general}), 
    \begin{align}
    \frac1T\sum_{t\in[T]}\mathrm{Var}(\bc^\top\bZ_t|\cH_{t-1}) 
    &= \frac1T\sum_{t\in[T]}\EE [\bc^\top \bZ_t\bZ_t^\top \bc|\cH_{t-1}]\nonumber\\
    & = \frac1T\sum_{t\in[T]}\bc^\top \EE [ \bZ_t\bZ_t^\top|\cH_{t-1}]\bc\label{eq::cond-var-1-general}
    \end{align}
    and 
    \begin{align}
    \bc^\top\EE [ \bZ_t\bZ_t^\top|\cH_{t-1}] \bc
    & = \bc^\top\EE_{\bX_t}\big[\EE_{A_t\sim \pi_t, Y_t(a)}[\bZ_t\bZ_t^\top|\cH_{t-1}, \bX_t]\big|\cH_{t-1}\big]\bc\nonumber\\
    & = \bc^\top\EE_{\bX_t}\bigg[\EE_{A_t\sim \pi_t}\bigg[\frac{1}{\pi_t^2(A_t)}1_{\{A_t = a\}}\Big|\cH_{t-1}, \bX_t\bigg]\cdot\nonumber\\
    & \quad\quad\EE_{Y_t(a)}[\bg(\bX_t, Y_t(a);\btheta_a^*)\bg(\bX_t, Y_t(a);\btheta_a^*)^\top|\cH_{t-1}, \bX_t]\bigg|\cH_{t-1}\bigg]\bc\nonumber\\
    & = \bc^\top\EE_{\bX_t}\left[\frac{1}{\pi_t(a)}\cdot \EE_{Y_t(a)}[\bg(\bX_t, Y_t(a);\btheta_a^*)\bg(\bX_t, Y_t(a);\btheta_a^*)^\top|\bX_t]\bigg|\cH_{t-1}\right]\bc\nonumber\\
    & = \bc^\top\EE_{\bX_t}[\bI_{a, t}|\cH_{t-1}]\bc.\label{eq::cond-var-2-general}
    \end{align}
    Here we define $\bI_{a, t} := \frac{1}{\pi_t(a)}\cdot \EE_{Y_t(a)}[\bg(\bX_t, Y_t(a);\btheta_a^*)\bg(\bX_t, Y_t(a);\btheta_a^*)^\top|\bX_t]$. We also define\\ $\bar \bI_{a, t} := \frac{1}{\bar\pi(a)}\cdot \EE_{Y_t(a)}[\bg(\bX_t, Y_t(a);\btheta_a^*)\bg(\bX_t, Y_t(a);\btheta_a^*)^\top|\bX_t]$.
    Then 
    \begin{align*}
      &|\bc^\top\bI_{a, t}\bc - \bc^\top\bar\bI_{a, t}\bc|\\
      &= |\bar \pi(a|\bX_t) - \pi_t(a)|\cdot \frac{1}{\bar \pi(a|\bX_t)\pi_t(a)} \bc^\top\EE[\bg(\bX_t, Y_t(a);\btheta_a^*)\bg(\bX_t, Y_t(a);\btheta_a^*)^\top|\bX_t]\bc \\
      &\leq M_2|\bar \pi(a|\bX_t) - \pi_t(a)|,
    \end{align*}
    where we have used Assumption \ref{aspt:boundedness-general}.
    
    Note that the random variables $\{\pi_t(a) - \bar \pi(a|\bX_t)\}_{t\geq 1}$ are uniformly integrable. Thus, $\pi_t(a) - \bar \pi(a|\bX_t)\xrightarrow{p} 0$ implies 
    $$
    \lim_{t\rightarrow \infty}\EE|\bar \pi(a|\bX_t) - \pi_t(a)| = 0.
    $$
    Combining the above facts, if we denote $U_{t, \bc}: = \bc^\top\bI_{a, t}\bc - \bc^\top\bar\bI_{a, t}\bc$, then
    $$
    \lim_{t\rightarrow \infty}\EE|U_{t, \bc}| = 0.
    $$
    Also, noticing that 
    $$
    \EE \big| \EE[U_{t, \bc}|\cH_{t-1}]\big|\leq 
    \EE \EE\big[|U_{t, \bc}|\big|\cH_{t-1}\big]=\EE |U_{t, \bc}|,
    $$
    We deduce that 
    $$
    \lim_{t\rightarrow \infty} \EE \big| \EE[U_{t, \bc}|\cH_{t-1}]\big| = 0.
    $$
    
    Note that the following property about $L_1$ convergence is true: For a sequence of random variables $U_t'$, if $U_t'\xrightarrow{L_1}0$, then its running average sequence $\bar U_t' = \frac1t\sum_{\tau\in[t]}U_\tau'$ satisfies $\bar U_t'\xrightarrow{L_1}0$. Let $U_t' = U_{t, \bc}$, then we have 
    $$
    \frac1t\sum_{\tau\leq t} \EE[U_{\tau, \bc}|\cH_{\tau-1}] \xrightarrow{L_1} 0.
    $$
    Plugging in the expression of $U_{\tau, \bc}$, we have 
    \begin{equation}\label{eq::l1-convergence-1-general}
    \frac1T\sum_{t\leq T}\EE\big[\bc^\top\bI_{a, t}\bc|\cH_{t-1}\big]
    - \frac1T\sum_{t\leq T}\EE\big[\bc^\top\bar\bI_{a, t}\bc|\cH_{t-1}\big]\xrightarrow{L_1} 0,
    \end{equation}
    and
    \begin{equation}\label{eq::l1-convergence-2-general}
        \frac1T\sum_{t\leq T}\EE\big[\bc^\top\bar\bI_{a, t}\bc|\cH_{t-1}\big] = \frac1T\sum_{t\leq T}\bc^\top\EE\big[\bar\bI_{a, t}\big] \bc = \bc^\top\bar \bI_{a} \bc,
    \end{equation}
    where
    \begin{equation}\label{eq::Ibar-general}
    \bar\bI_{a}:= \EE \bar \bI_{a, t}
     = \EE\frac{1}{\bar \pi(a|\bX_t)} \bg(\bX_t, Y_t(a);\btheta_a^*)\bg(\bX_t, Y_t(a);\btheta_a^*)^\top.
    \end{equation}
    Combining \eqref{eq::cond-var-1-general}, \eqref{eq::cond-var-2-general}, \eqref{eq::l1-convergence-1-general} and \eqref{eq::l1-convergence-2-general}, we obtain that 
    \begin{equation}\label{eq::Iat-l1-convergence-general}
    \frac1T\sum_{t\in[T]}\mathrm{Var}(\bc^\top\bZ_t|\cH_{t-1})\xrightarrow{L_1} \bc^\top\bar\bI_{a}\bc.
    \end{equation}
    Because $L_1$ convergence implies convergence in probability, we have verified \eqref{eq::cond-var-general}.
    
    \textbf{Check (\ref{eq::cond-lindeberg-general})}: Using Chebyshev's inequality,
    \begin{align*}
    &\frac1T\sum_{t\in[T]}\EE\left[(\bc^\top \bZ_t)^21_{\{|\bc^\top\bZ_t|>\sqrt{T}\delta\}}\Big|\cH_{t-1}\right]\\
    \leq  & \frac1T\cdot \frac1{T\delta^2}\sum_{t\in[T]}\EE\left[(\bc^\top \bZ_t)^4\Big|\cH_{t-1}\right]\\
    = &  \frac1{T^{2}\delta^2}\sum_{t\in[T]}\EE\left[\Big(\frac{1}{\pi_t(A_t)}1_{\{A_t = a\}}\bc^\top \bg(\bX_t, Y_t(a); \btheta_a^*)\Big)^4\Big|\cH_{t-1}\right]\\
    \leq & \frac1{T^{2}\delta^2}\sum_{t\in[T]}\frac1{\pi_{\min}^4}\EE[\bc^\top \bg(\bX_t, Y_t(a); \btheta_a^*)]^4\\
     = &\frac1{T\delta^2}\cdot\frac1{\pi_{\min}^4}\EE[\bc^\top \bg(\bX_1, Y_1(a); \btheta_a^*)]^4\rightarrow 0
    \end{align*}
    Here we have used Assumptions \ref{aspt:boundedness-general} and \ref{aspt:min-sampling-prob}.
    
\subsection{Proof of Lemma \ref{lem::policy-with-no-contexts}}\label{apdx::proof-lem::policy-with-no-contexts}

{
By Chebyshev's inequality, $\forall \delta>0$, 
\begin{align}
    &\PP\bigg(\bigg|N_{i, t} - \sum_{i=1}^t\EE[1_{\{A_\tau = i\}}|\cH_{\tau-1}^0]\bigg|\geq \delta\bigg)\nonumber\\
    \leq & \frac1{\delta^2}\EE\bigg(\sum_{\tau=1}^t\big(1_{\{A_\tau = i\}} - \EE[1_{\{A_\tau = i\}}|\cH_{\tau-1}^0]\big)\bigg)^2\nonumber\\
    =& \frac1{\delta^2}\sum_{\tau=1}^t\EE\big(1_{\{A_\tau = i\}} - \EE[1_{\{A_\tau = i\}}|\cH_{\tau-1}^0]\big)^2\leq \frac{t}{\delta^2}.\label{eq::concentration-Nit}
\end{align}
In addition, given the minimum sampling probability $\pi_{\min}$, we have 
\begin{equation}\label{eq::policy-no-context-sum-e-1}
\sum_{\tau = 1}^t\EE[1_{\{A_\tau = i\}}|\cH_{\tau-1}^0]\geq \pi_{\min} t.
\end{equation}
Thus, by setting $\delta = \pi_{\min}t / 2$ in (\ref{eq::concentration-Nit}), we combine with (\ref{eq::policy-no-context-sum-e-1}) and deduce that 
\begin{equation}\label{eq::Nit-lower-bound}
    \PP\bigg(N_{i, t}\leq \frac{\pi_{\min}t}{2}\bigg)\leq \frac{4}{\pi_{\min}^2t}.
\end{equation}
This implies that 
$$
\PP\bigg(\frac{C_t}{N_{i, t}}\geq \frac{2C_t}{\pi_{\min}t}\bigg)\leq \frac{4}{\pi_{\min}^2t},
$$
which proves statement (i). 

At the same time, $\forall \delta>0$,
\begin{align}
&\PP\bigg(N_{i, t}\bigg|\hat\mu_{i, t} - \mu_i^*\bigg|\geq \delta\bigg)\nonumber\\
    =&\PP\bigg(\bigg|\sum_{i=1}^t 1_{\{A_\tau = i\}}Y_{\tau} - \sum_{i=1}^t\EE[1_{\{A_\tau = i\}}Y_{\tau}|\cH_{\tau-1}^0, A_\tau]\bigg|\geq \delta\bigg)\nonumber\\
    \leq & \frac1{\delta^2}\EE\bigg(\sum_{i=1}^t \big(1_{\{A_\tau = i\}}Y_{\tau} - \EE[1_{\{A_\tau = i\}}Y_{\tau}|\cH_{\tau-1}^0, A_\tau]\big)\bigg)^2\nonumber\\
    =& \frac1{\delta^2}\sum_{\tau=1}^t\EE1_{\{A_\tau = i\}}\big(Y_\tau(i) - \mu_i^*\big)^2\leq \frac{\sigma_Y^2t}{\delta^2}.\label{eq::concentration-hatmuit}
\end{align}
Here we have used the definition of $N_{i, t}$ and $\hat\mu_{i, t}$, as well as the fact that due to the unconfoundedness assumption,
$$
\EE[1_{\{A_\tau = i\}}Y_{\tau}|\cH_{\tau-1}^0, A_\tau] = 1_{{\{A_\tau = i\}}}\mu_i^*.
$$
Combining (\ref{eq::Nit-lower-bound}) and (\ref{eq::concentration-hatmuit}), we obtain that $\forall \delta>0$,
$$
\PP\bigg(|\hat\mu_{i, t} - \mu_i^*|\geq \frac{2\delta}{\pi_{\min} t}\bigg)\leq \frac{\sigma_Y^2t}{\delta^2} + \frac{4}{\pi_{\min}^2 t}.
$$
Let $\delta' = \frac{2\delta}{\pi_{\min} t}$, and we obtain that 
$$
\PP\bigg(|\hat\mu_{i, t} - \mu_i^*|\geq \delta'\bigg)\leq \frac{4\sigma_Y^2}{\delta'^2t} + \frac{4}{\pi_{\min}^2 t}\rightarrow 0
$$
as $t\rightarrow \infty$. Thus, statement (ii) is proved.
}

\subsection{Proof of Lemma \ref{lem::clipping-Lipschitz}}
\label{apdx::proof-lem::clipping-Lipschitz}

\begin{proof}

    %\textbf{Check Lipschitz continuity of $\operatorname{Clip}(\bm{\pi})$}: 
    We show the Lipschitz continuity of $\operatorname{Clip}(\bm{\pi})$. Altough $q$ function is not smooth, it is continuous and piecewise differentiable in both $\nu$ and $\bm{\pi}$, so we may analyze the slope of each piece. 

    We first note that for any $\nu$, the function $q(\nu; \cdot)$ is 1-Lipschitz continuous in $\bm{\pi}$. To see this, for any two vectors $\bm{\pi}, \bm{\pi}' \in [0, 1]^{|\mathcal{A}|}$, we have 
    \begin{align}
        |q(\nu; \bm{\pi}) - q(\nu; \bm{\pi}')| &= \left|\sum_{a} \max\{\bm{\pi}_a - \nu, \pi_{\min}\} - \sum_{a} \max\{\bm{\pi}_a' - \nu, \pi_{\min}\}\right| \leq \|\bm{\pi} - \bm{\pi}'\|_1.
    \end{align}

    Since $q(\nu^*(\bm{\pi}); \bm{\pi}) = 1 = q(\nu^*(\bm{\pi}'); \bm{\pi}')$, we have that 
    \begin{align}
        |q(\nu(\bm{\pi}); \bm{\pi}) - q(\nu(\bm{\pi}); \bm{\pi}')| = |q(\nu(\bm{\pi}'); \bm{\pi}') - q(\nu(\bm{\pi}); \bm{\pi}')| \leq \|\bm{\pi} - \bm{\pi}'\|_1.
    \end{align}

    Since $q(\nu; \bm{\pi}')$ is piecewise linear and decreasing in $\nu$, we can lower bound its slope. In particular, over intervals where some entries of $\bm{\pi}'_a - \nu > \pi_{\min}$, the derivative of $q$ w.r.t. $\nu$ is:
    \begin{align}
        \frac{\partial q(\nu; \bm{\pi}')}{\partial \nu} = -\left|\left\{a: \bm{\pi}_a' - \nu > \pi_{\min}\right\}\right| \leq -1.
    \end{align}
    Because at least one action must be unclipped, so the slope is at most $-1$. Thus:
    \begin{align}
        |q(\nu(\bm{\pi}); \bm{\pi}') - q(\nu(\bm{\pi}'); \bm{\pi}')| \geq |\nu(\bm{\pi}) - \nu(\bm{\pi}')| \cdot 1.
    \end{align}
    Therefore, we have that 
    \begin{align}
        |\nu(\bm{\pi}) - \nu(\bm{\pi}')| \leq \|\bm{\pi} - \bm{\pi}'\|_1 \leq |\mathcal{A}| \|\bm{\pi} - \bm{\pi}'\|_2.
    \end{align}
    The second inequality follows from the Cauchy-Schwarz inequality.

    To complete the proof, we see that 
    \begin{align}
        \|\operatorname{Clip}(\bm{\pi}) - \operatorname{Clip}(\bm{\pi}')\|_2 
        &\leq \|\bm{\pi} - \bm{\pi}'\|_2 + |\nu^*(\bm{\pi}) - \nu^*(\bm{\pi}')|\\
        &\leq \|\bm{\pi} - \bm{\pi}'\|_2 + |\mathcal{A}| \|\bm{\pi} - \bm{\pi}'\|_2 \\
        &= (|\mathcal{A}|+1) \|\bm{\pi} - \bm{\pi}'\|_2. \label{eq::clipping-Lipschitz-bound}
    \end{align}
\end{proof}

\subsection{Proof of Lemma \ref{lem::Convergence-of-weighted-avg-action-selection}}\label{apdx::proof-lem::Convergence-of-weighted-avg-action-selection}

Note that 
\begin{align*}
\EE[W_t1_{\{A_t = a\}} - 1|\cH_{t-1}]
& = \EE_{\bS_t, \bX_t}\bigg[\EE_{A_t\sim \pi_t}\big[\frac{1}{\pi_t(a)}1_{\{A_t = a\}}\big|\cH_{t-1}, \bS_t, \bX_t\big]\bigg|\cH_{t-1}\bigg] - 1\\
& = 1-1 = 0.
\end{align*}
Thus, for any constant $\delta>0$,
\begin{align*}
&\PP\left(\bigg|\frac1T\sum_{t=1}^TW_t1_{\{A_t= a\}}-1\bigg|>\delta\right)\\
\leq&\frac1{\delta^2T^2}\EE\bigg[\sum_{t=1}^T(W_t1_{\{A_t= a\}}-1)\bigg]^2\\
=&\frac1{\delta^2T^2}\EE\bigg[\sum_{t=1}^T(W_t1_{\{A_t= a\}}-\EE[W_t1_{\{A_t = a\}}|\cH_{t-1}])\bigg]^2\\
=&\frac1{\delta^2T^2}\sum_{t=1}^T\EE\bigg(W_t1_{\{A_t= a\}}-\EE[W_t1_{\{A_t = a\}}|\cH_{t-1}]\bigg)^2\\
\leq &\frac1{\delta^2T^2}\EE W_t^2\leq \frac1{\delta^2T^2}\cdot\frac{T}{\pi_{\min}^2}\rightarrow 0.
\end{align*}
Here the first inequality is due to Chebyshev's Inequality, the second equality is because the cross terms has zero expectation after expansion due to martingale properties. The lemma follows.

\section{Additional Technical Lemmas}
\begin{lemma}[Matrix Azuma \citep{tropp2012user}]
    \label{lem::matrix-azuma}

    Consider a finite adapted sequence $\{X_k\}_{k = 1}^t$ of self-adjoint $d\times d$ matrices with respect to the filtration $\{\mathcal{F}_k\}_{k = 1}^t$, and a fixed sequence $\{A_k\}_{k = 1}^t$ of self-adjoint matrices that satisfy
    \begin{align}
        \EE[A_k \mid \mathcal{F}_{k-1}] = 0, \quad \text{and} \quad X_k^2 \preceq A_k^2 \text{ a.s.}
    \end{align}
    Compute the variance parameter 
    \begin{align}
        \sigma_t^2 = \left\|\sum_{i = 1}^t A_i^2 \right\|_2.
    \end{align}
    Then, for all $\epsilon \geq 0$, 
    \begin{align}
        \PP\left( \left\|\sum_{k = 1}^t X_k \right\|_2 \geq \epsilon \right) \leq d \exp\left(-\frac{\epsilon^2}{8\sigma_t^2}\right).
    \end{align}
\end{lemma}

Applying the union bound on all $t' \in [t]$, we have the following corollary.

\begin{corollary}
    \label{cor::matrix-azuma}
    Under the same conditions as in Lemma \ref{lem::matrix-azuma}, we have that with a probability at least $1-\delta/(t-1)$, for all $\tau \in t, t+1, \ldots$,
    \begin{align}
        \frac{1}{\tau}\left\|\sum_{k = 1}^{\tau} X_k \right\|_2 \leq \sqrt{\frac{16 \sigma_{\tau}^2}{\tau^2} \log\left(\frac{\tau d}{\delta} \right)}.
    \end{align}
\end{corollary}
\begin{proof}
    The proof is to apply the union bound on all $\tau \in t, t+1, \ldots$ with each event having a probability of at least $1-\delta/\tau^2$. The total failure probability is at most $\sum_{\tau = t}^{\infty}\frac{1}{\tau^2} \leq \frac{\delta}{t-1}$.
\end{proof}

\begin{lemma}[Law of large numbers for martingale difference sequence \citep{chow1967strong}]
    \label{lem::law-of-large-numbers-for-martingale-difference-sequence}
    Let $Y_n = \sum_{t=1}^n \bX_t$ be a martingale difference sequence, such that 
    \begin{align}
        \sum_{t=1}^{\infty} \EE\left[|\bX_t|^{2\alpha}\right] / k^{1+\alpha} < \infty.
    \end{align}
    Then,
    \begin{align}
        \frac{1}{n} \sum_{t=1}^n Y_t \xrightarrow{a.s.} 0.
    \end{align}
\end{lemma}

\begin{theorem}[Stochastic approximation convergence (Theorem 2 \citep{borkar2008stochastic})]
\label{thm:sac_converge}
    Let $\{x_n\}_{n \geq 0}$ be the stochastic approximation process in $\mathbb{R}^d$ given by 
    \begin{align}
    \bx_{n+1}=\bx_n+a(n)\left[\bh\left(x_n\right)+\bM_{n+1}\right], n \geq 0
    \end{align}
    with prescribed $\bx_0$ with the following assumptions holding:
    \begin{itemize}
        \item The map $\bh: \mathbb{R}^d \mapsto \mathbb{R}^d$ is Lipschitz: $\|\bh(\bx) - \bh(y)\| \leq L \|\bx-\by\|$ for some $0 < L < \infty$.
        \item Stepsizes $\{a(n)\}$ are positive scalars satisfying
        $$
            \sum_n a(n) = \infty, \sum_n a(n)^2 < \infty
        $$
        \item $\{\bM_n\}$ is a martingale difference sequence with respect to the increasing family of $\sigma$-fields 
        $$
            \mathcal{F}_n \stackrel{\text { def }}{=} \sigma\left(\bx_m, \bM_m, m \leq n\right)=\sigma\left(\bx_0, \bM_1, \ldots, \bM_n\right), n \geq 0.
        $$
        That is 
        $$
            \EE[\bM_{n+1} \mid \cF_n] = 0 \text{ a.s. }, n \geq 0.
        $$
        \item Furthermore, $\{\bM_n\}$ are square-integrable with
        $$
            \EE[\|\bM_{n+1}\|^2 \mid \cF_{n}] \leq K (1+ \|\bx_n\|^2) \text{ a.s.,} n \geq 0.
        $$
        for some constant $K$.
        Then the sequence $\{\bx_n\}$ converges to a compact connected internally chain transitive invariant set of $\dot{\bx}(t)=h(\bx(t)), t \geq 0$.
        Additionally, there is a unique solution $\bx^*$, then $\{\bx_n\}$ converges to $\bx^*$.
    \end{itemize}
    
\end{theorem}

\section{Additional Details on Simulation Studies}
\label{apdx::simulation-details}

\subsection{Environment Settings}
\label{apdx::simulation-details::environment-settings}

We give the details of the simulation environment settings. 

The first type of environments is the noisy contextual linear bandit environment including \texttt{NC-Hard1}, \texttt{NC-Hard2}, \texttt{NC-Gaussian}. Each environment has a ground-truth parameter $\btheta_a^*$. At each time $t$, the following variables are generated:
\begin{equation}\label{eq::model-1}
    \begin{aligned}
        \text{True context: } & \bS_t \sim \cD_S,\\
        \text{Predicted context: } & f(\bX_t) = \bS_t + \bepsilon_t, \text{ where } \bepsilon_t \sim \cD_{\epsilon}(\cdot \mid \bS_t),\\
        \text{Reward: } & Y_t = \langle\btheta_{A_t}^*, \bS_t\rangle + \eta_t, \text{ where } \eta_t \sim \cD_{\eta},
    \end{aligned}.
\end{equation}
where $A_t$ is the algorithm-chosen action. Recall that $\Sigma_S = \EE[\bS_t \bS_t^{\top}]$, $\Sigma_e$ is the covariance matrix of $\bepsilon_t$ (assumed to be independent of $\bS_t$), and $\Sigma_{\eta}$ is the covariance matrix of $\eta_t$.

Both hard environments have one-dimensional context and true parameters ($d = 1$), two actions $|\cA| = 2$, and two contexts $\cS = \{0, -1\}$. The context distribution $\cD_S$ is uniform over $\cS$. \texttt{hard-1} and \texttt{hard-2} have the true parameters $\btheta_0^* = (3, 1)$ and $\btheta_1^* = (-3, -1)$ respectively. They further share the same prediction error distribution $\cD_{\epsilon}(\cdot \mid \bS_t)$, given by (\ref{eq::model-1-error-distribution}). We set the reward noise $\cD_{\eta} = \cN(0, \sigma_{\eta}^2)$.
\begin{align}
\begin{array}{ll}
            \PP(f(\bX_t) = 1 \mid \bS_t = 0) = 2/3 & \quad \PP(f(\bX_t) = -2 \mid \bS_t = 0) = 1/3 \\
            \PP(f(\bX_t) = -2 \mid \bS_t = -1) = 2/3 & \quad \PP(f(\bX_t) = 1 \mid \bS_t = -1) = 1/3.
\end{array}
\label{eq::model-1-error-distribution}
\end{align}

For the \texttt{NC-Gaussian} environment, we randomly sample the true parameters $\btheta_a^*$ from $\cN(0, \Sigma_{\btheta})$ for each $a\in\cA$ independently. We choose $\cD_S = \cN(0, \Sigma_S)$, $\cD_{\epsilon}(\cdot \mid \bS_t) = \cN(0, \Sigma_e)$, $\cD_{\eta} = \cN(0, \sigma_{\eta}^2)$, respectively.

The second type of environments is the misspecified contextual linear bandit environment including \texttt{MC-Polynomial}, \texttt{MC-Neural}. In these two environments, we have $|\cA| = 2$, $d = 1$. The context is sampled from $\cD_X = \cN(0, \Sigma_X)$. In the \texttt{MC-Polynomial} environment, the true reward function is given by 
$$
    y(\bx, a) = \langle\btheta_{a, 1}^*, \bx\rangle + \langle\btheta_{a, 2}^*, \bx^2\rangle + \dots + \langle\btheta_{a, d}^*, \bx^d\rangle,
$$
where $d$ is the degree of the polynomial. The true parameters $\btheta_{a, i}^*$ are randomly sampled from $\cN(0, \Sigma_{\btheta})$ for each $a\in\cA$ and $i\in\{1, 2, \ldots, d\}$ independently. In the \texttt{MC-Neural} environment, the true reward function is given by a two layer neural network with one hidden layer of size $d$.
$$
    y(\bx, a) = \operatorname{ReLU}(\langle\btheta_{a}^*, \bx\rangle),
$$
where $\operatorname{ReLU}(x) = \max(0, x)$ is the ReLU activation function.

The true parameters $\btheta_{a}^*$ are randomly sampled from $\cN(0, \Sigma_{\btheta})$ for each $a\in\cA$ independently. We choose $\cD_X = \cN(0, \Sigma_X)$, $\cD_{\eta} = \cN(0, \sigma_{\eta}^2)$, respectively.

\subsection{Additional Information on OPE}
\label{apdx::simulation-details::ope}

% \ziping{reference broken here}
In the OPE setting, we compare the proposed inference method with the CADR (Contextual Adaptive Doubly Robust) method  \citep{bibaut2021post} under various choice of prediction model including linear model, tree-based model, and a dumpy model that always outputs 0. We run CADR on the same dataset collected by Boltzmann exploration w.r.t. Ridge regression in five environments introduced above. 

To implement the CADR method, we define the following functions:
\begin{align}
    \Psi\left(g, Q_Y\right):= \EE_{A_t \sim g(A_t \mid \bX_t)}[Q_Y(A_t, \bX_t)].
\end{align}
\begin{align}
    D^{\prime}(g, \bar{Q})(x, a, y):=\frac{g^*(a \mid x)}{g(a \mid x)}(y-\bar{Q}(a, x))+\int \bar{Q}\left(a^{\prime}, x\right) g^*\left(a^{\prime} \mid x\right) d \mu_{\mathcal{A}}\left(a^{\prime}\right).
\end{align}
\begin{align}
    D(g, \bar Q)(x, a, y) = D'(g, \bar Q)(x, a, y) - \Psi\left(g, \bar Q\right).
\end{align}

Let $g_1, \dots, g_T$ be the logging policy that collects the data, and $g^*$ be the target policy that we aim to evaluate. 

For each step $t = 1, \dots, T$, the CADR method computes the following quantities:
\begin{itemize}
    \item Train $\hat Q_{t-1}: \cX \times \cA \to \RR$ on the dataset $((\bX_s, A_s, Y_s))_{s = 1}^{t-1}$ using the outcome regression estimator.
    \item Set $D'_{t, s} = D(g_s, \hat Q_{t-1})(\bX_t, A_t, Y_t)$ for each $s = t, \dots, T$.
    \item Set 
    \begin{align}
        \hat \sigma_t^2 = \frac{1}{t-1} \sum_{s = 1}^{t-1} \frac{g_t(A_s \mid \bX_s)}{g_s(A_s \mid \bX_s)} (D'_{t, s})^2 - \left(\frac{1}{t-1} \sum_{s = 1}^{t-1} \frac{g_t(A_s \mid \bX_s)}{g_s(A_s \mid \bX_s)} D'_{t, s} \right)^2.
    \end{align}
\end{itemize}
In the end, the CADR method outputs the following estimate:
\begin{align}
    \hat \Psi_T = \frac{\Gamma_T}{T} \sum_{t=1}^T \hat \sigma_t^{-1} D'_{t, t}, \text{ where } \Gamma_T = \left(\frac{1}{T}\sum_{t=1}^T \hat \sigma_t^{-1}\right)^{-1}.
\end{align}
and confidence interval 
\begin{align}
    \text{CI}_{\alpha} = [\hat \Psi_T \pm \xi_{1-\alpha/2} \Gamma_T / \sqrt{T}].
\end{align}

% \subsection{Additional Results}
% \label{apdx::simulation-details::additional-results}
    
    \end{appendix}
    
    %%%%%%%%%%%%%%%%%%%%%%%%%%%%%%%%%%%%%%%%%%%%%%
    %% Support information, if any,             %%
    %% should be provided in the                %%
    %% Acknowledgements section.                %%
    %%%%%%%%%%%%%%%%%%%%%%%%%%%%%%%%%%%%%%%%%%%%%%
    % \begin{acks}[Acknowledgments]
    % The authors would like to thank the anonymous referees, an Associate
    % Editor and the Editor for their constructive comments that improved the
    % quality of this paper.
    % \end{acks}
    
    % %%%%%%%%%%%%%%%%%%%%%%%%%%%%%%%%%%%%%%%%%%%%%%
    % %% Funding information, if any,             %%
    % %% should be provided in the                %%
    % %% funding section.                         %%
    % %%%%%%%%%%%%%%%%%%%%%%%%%%%%%%%%%%%%%%%%%%%%%%
    % \begin{funding}
    % This work is partly supported by NSF Grant DMS-2515285.
    % \end{funding}
    
    %%%%%%%%%%%%%%%%%%%%%%%%%%%%%%%%%%%%%%%%%%%%%%
    %% Supplementary Material, including data   %%
    %% sets and code, should be provided in     %%
    %% {supplement} environment with title      %%
    %% and short description. It cannot be      %%
    %% available exclusively as external link.  %%
    %% All Supplementary Material must be       %%
    %% available to the reader on Project       %%
    %% Euclid with the published article.       %%
    %%%%%%%%%%%%%%%%%%%%%%%%%%%%%%%%%%%%%%%%%%%%%%
    % \begin{supplement}
    % \stitle{Title of Supplement A}
    % \sdescription{Short description of Supplement A.}
    % \end{supplement}
    % \begin{supplement}
    % \stitle{Title of Supplement B}
    % \sdescription{Short description of Supplement B.}
    % \end{supplement}

%% if your bibliography is in bibtex format, uncomment commands:
\bibliographystyle{unsrtnat} % Style BST file (imsart-number.bst or imsart-nameyear.bst)
\bibliography{main}       % Bibliography file (usually '*.bib')
\end{document}